\documentclass[]{amsart}
\usepackage{}
\usepackage{amsfonts}
\usepackage{latexsym,bm}
\usepackage{amssymb}
\usepackage{amsmath}
\usepackage{amsthm}
\usepackage{mathrsfs}
\usepackage[all]{xy}
\usepackage{colortbl,caption}
\usepackage{bbm}
\numberwithin{equation}{section}

\theoremstyle{plain}
\newtheorem{thm}{Theorem}[section]
\newtheorem{lem}[thm]{Lemma}
\newtheorem{prop}[thm]{Proposition}
\newtheorem{cor}[thm]{Corollary}
\newtheorem{lem-defi}[thm]{Lemma-Definition}
\newtheorem{prop-defi}[thm]{Proposition-Definition}

\theoremstyle{remark}
\newtheorem{rmk}[thm]{Remark}
\newtheorem{set}[thm]{Setting}
\newtheorem{e.g.}[thm]{Example}
\theoremstyle{definition}
\newtheorem{defi}[thm]{Definition}
\newtheorem{nota}[thm]{Notation}
\newtheorem{conv}[thm]{Convention}
\newtheorem{assu}[thm]{Assumption}
\newtheorem{cons}[thm]{Construction}

\newcommand{\mcA}{{\mathcal A}}

\newcommand{\mcC}{{\mathcal C}}
\newcommand{\mcD}{{\mathcal D}}
\newcommand{\mcE}{{\mathcal E}}
\newcommand{\mcF}{{\mathcal F}}

\newcommand{\mcH}{{\mathcal H}}
\newcommand{\mcI}{{\mathcal I}}

\newcommand{\K}{{\mathcal K}}

\newcommand{\mcL}{{\mathcal L}}

\newcommand{\mcN}{{\mathcal N}}
\newcommand{\mcO}{{\mathcal O}}
\newcommand{\mcP}{{\mathcal P}}

\newcommand{\mcS}{{\mathcal S}}

\newcommand{\mcV}{{\mathcal V}}
\newcommand{\mcW}{{\mathcal W}}
\newcommand{\mcX}{{\mathcal X}}
\newcommand{\mcY}{{\mathcal Y}}
\newcommand{\mcZ}{{\mathcal Z}}

\newcommand{\C}{{\mathbb C}}

\newcommand{\mbF}{{\mathbb F}}
\newcommand{\mbG}{{\mathbb G}}

\newcommand{\mbN}{{\mathbb N}}

\newcommand{\mbP}{{\mathbb P}}
\newcommand{\mbQ}{{\mathbb Q}}
\newcommand{\mbR}{{\mathbb R}}

\newcommand{\mbZ}{{\mathbb Z}}

\newcommand{\msD}{{\mathscr D}}
\newcommand{\msE}{{\mathscr E}}

\newcommand{\mfg}{{\mathfrak g}}
\newcommand{\mfh}{{\mathfrak h}}

\newcommand{\mfl}{{\mathfrak l}}
\newcommand{\mfm}{{\mathfrak m}}
\newcommand{\mfn}{{\mathfrak n}}

\newcommand{\mfp}{{\mathfrak p}}
\newcommand{\mfq}{{\mathfrak q}}

\newcommand{\bC}{{\textbf C}}

\newcommand{\bS}{{\textbf S}}

\newcommand{\bZ}{{\textbf Z}}

\newcommand{\Pic}{\text{Pic}}
\newcommand{\Aut}{\text{Aut}}

\newcommand{\ad}{\text{ad}}
\newcommand{\Symb}{\text{Symb}}
\newcommand{\ev}{\text{ev}}
\newcommand{\aut}{{\rm aut}}

\newcommand{\wcF}{\widetilde{\mathcal{F}}}

\newcommand{\wpsi}{\widetilde{\psi}}
\newcommand{\wv}{\widetilde{v}}

\newcommand{\wbF}{\widetilde{\mathbb{F}}}
\newcommand{\wcI}{\widetilde{\mathcal{I}}}
\newcommand{\wtheta}{\widetilde{\theta}}

\newcommand{\Hom}{\text{Hom}}
\newcommand{\rank}{\text{rank}}
\newcommand{\ba}{\bar{\alpha}}
\newcommand{\bb}{\bar{\beta}}

\newcommand{\wbZ}{\widehat{\textbf{Z}}}
\newcommand{\wbZa}{\widehat{\textbf{Z}}^{\alpha}}
\newcommand{\wcC}{\widehat{\mathcal{C}}}
\newcommand{\wcCa}{\widehat{\mathcal{C}}^\alpha}

\begin{document}

\title[Fano deformation rigidity]{Fano deformation rigidity of rational homogeneous spaces of submaximal Picard numbers}
\author{Qifeng Li}
\maketitle

\begin{abstract}
We study the question whether rational homogeneous spaces are rigid under Fano deformation. In other words, given any smooth connected family $\pi:\mcX\rightarrow\mcZ$ of Fano manifolds, if one fiber is biholomorphic to a rational homogeneous space $S$, whether is $\pi$ an $S$-fibration? The cases of Picard number one were studied in a series of papers by J.-M. Hwang and N. Mok. For higher Picard number cases, we notice that the Picard number of a rational homogeneous space $G/P$ satisfies $\rho(G/P)\leq \rank(G)$. Recently A. Weber and J. A. Wi\'{s}niewski proved that rational homogeneous spaces $G/P$ with Picard numbers $\rho(G/P)=\rank(G)$ (i.e. complete flag manifolds) are rigid under Fano deformation. In this paper we show that the rational homogeneous space $G/P$ is rigid under Fano deformation, providing that $G$ is a simple algebraic group of type $ADE$, the Picard number $\rho(G/P)=\rank(G)-1$ and $G/P$ is not biholomorphic to $\mbF(1, 2, \mbP^3)$ or $\mbF(1, 2, Q^6)$. The variety $\mbF(1, 2, \mbP^3)$ is the set of flags of projective lines and planes in $\mbP^3$, and $\mbF(1, 2, Q^6)$ is the set of flags of projective lines and planes in 6-dimensional smooth quadric hypersurface. We show that $\mbF(1, 2, \mbP^3)$ have a unique Fano degeneration, which is explicitly constructed. The structure of possible Fano degeneration of $\mbF(1, 2, Q^6)$ is also described explicitly.
To prove our rigidity result, we firstly show that the Fano deformation rigidity of a homogeneous space of type $ADE$ can be implied by that property of suitable homogeneous submanifolds. Then we complete the proof via the study of Fano deformation rigidity of rational homogeneous spaces of small Picard numbers. As a byproduct, we also show the Fano deformation rigidity of other manifolds such as $\mbF(0, 1, \ldots, k_1, k_2, k_2+1, \ldots, n-1, \mbP^n)$ and $\mbF(0, 1, \ldots, k_1, k_2, k_2+1, \ldots, n, Q^{2n+2})$ with $0\leq k_1<k_2\leq n-1$.
\end{abstract}

\smallskip


\smallskip


\tableofcontents


\section{Introduction}

We work over the field $\C$ of complex numbers. A Fano manifold $M$ is said to be rigid under Fano deformation if any smooth connected family $\pi: \mcX\rightarrow\mcZ$ of Fano manifolds with $M$ being a fiber must be an $M$-fibration. If the fiber $\mcX_z:=\pi^{-1}(z)$ at some point $z\in\mcZ$ is not biholomorphic to $M$, we say $\mcX_t$ is a Fano degeneration of $M$.

Our interest in this paper is the Fano deformation rigidity of rational homogeneous spaces. The Fano deformation rigidity of rational homogeneous spaces of Picard number one is studied by J.-M. Hwang and N. Mok in \cite{Hw97}\cite{HM98}\cite{HM02}\cite{HM05}. Among the rational homogeneous spaces of Picard number one, $\mbF(1, Q^5)$ is the only variety that is not rigid under Fano deformation, where $\mbF(1, Q^5)$ is the family of projective lines on a 5-dimensional smooth quadric hypersurface. Moreover, the variety $\mbF(1, Q^5)$ has a unique Fano degeneration, see \cite{PP10} and \cite{HL}. In particular, we have

\begin{thm}\cite{Hw97}\cite{HM98}\cite{HM02}\cite{HM05}\label{thm. intro Picard numbe one}
Let $\bS$ be a rational homogeneous space of Picard number one. If $\bS\ncong\mbF(1, Q^5)$, then $\bS$ is rigid under Fano deformation.
\end{thm}

To our knowledge the first result on higher Picard number cases is due to J. A. Wi\'{s}niewski \cite{Wis91}.

\begin{thm}\cite{Wis91}\label{thm. intro (Am, a1am)}
The variety $F(1, n, \C^{n+1})$ is rigid under Fano deformation, where $F(1, n, \C^{n+1})$ is the set of flags of $1$-dimensional and $n$-dimensional vector subspaces in $\C^{n+1}$.
\end{thm}

A rational homogeneous space is denoted by $G/P$, where $G$ is a semisimple algebraic group and $P$ is a parabolic subgroup of $G$. The Picard number of $G/P$ satisfies that  $\rho(G/P)\leq \rank(G)$, where $\rank(G)$ is the dimensional of any maximal torus of $G$. Recently A. Weber and J. A. Wi\'{s}niewski \cite{WW17} verified Fano deformation rigidity of the cases with $\rho(G/P)=\rank(G)$. More precisely,

\begin{thm}\cite{WW17}\label{thm. intro complete flag mfds}
The rational homogenous space $G/B$ is rigid under Fano deformation, where $G$ is a semisimple algebraic group and $B$ is a Borel subgroup.
\end{thm}

Motivated by Theorem \ref{thm. intro Picard numbe one} and Theorem \ref{thm. intro complete flag mfds}, one naturally ask what about the intermediate cases? A previous result of the author \cite[Theorem 1]{Li18} shows that product structure is preserved under Fano deformation. In particular, a rational homogeneous space $\bS$, satisfying $\bS=\bS_1\times\bS_2$, is rigid under Fano deformation if and only if so are $\bS_1$ and $\bS_2$. It reduces the problem to the case when $G$ is simple.

Our main result is on the cases with submaximal Picard number, i.e. $\rho(G/P)=\rank(G)-1$. More precisely, we have the following

\begin{thm}\label{thm. intro submaximal Picard numbers}
Let $G$ be a simple algebraic group of type $ADE$ and $P$ be a parabolic subgroup of $G$ such that the Picard number $\rho(G/P)=\rank(G)-1$. If $G/P$ is not biholomorphic to $\mbF(1, 2, \mbP^3)$ or $\mbF(1, 2, Q^6)$, then it is rigid under Fano deformation, where $\mbF(1, 2, \mbP^3)$ (resp. $\mbF(1, 2, Q^6)$) is the set of flags of projective lines and planes on $\mbP^3$ (resp. a 6-dimensional smooth quadric hypersurface).
\end{thm}

It was observed by A. Weber and J. A. Wi\'{s}niewski \cite{WW17} that $F^d(1, 2; \C^4)$ is a Fano deformation of $\mbF(1, 2, \mbP^3)$, where $F^d(1, 2; \C^4)$ is defined as follows.

\begin{cons}\label{cons. intro Fd(1, 2; C4)}
Let $\omega$ be a symplectic form on $\C^4$, i.e. $\omega$ is a nondegenerate antisymmetric form on $\C^4$. Denote by $\mcL^\omega\subset T\mbP^3$ the associated contact distribution on $\mbP^3:=\mbP(\C^4)$, and write $\mcL_\sigma:= T\mbP^3/\mcL^\omega$. We define $F^d(1, 2; \C^4):=\mbP(\mcL_\sigma\oplus\mcL^\omega)$.
\end{cons}

Indeed we can show moreover the following

\begin{thm}\label{thm. intro unique degeneration of F(1, 2, C4)}
The variety $F^d(1, 2; \C^4)$ is the unique Fano degeneration of $\mbF(1, 2, \mbP^3)$.
\end{thm}

We also describe the structure possible Fano degeneration of $\mbF(1,2, Q^6)$, see Proposition \ref{prop. degeneration (D4, a2a3a4)}.

The strategy to prove Theorem \ref{thm. intro submaximal Picard numbers} is as follows. Firstly, we show that the Fano deformation rigidity of a rational homogenous space is implied the that property of a suitable class of its homogeneous submanifolds. Then we show the Fano deformation rigidity of these homogeneous submanifolds.

To explain the sketch, we need some convention on notations. Given a simple algebraic group $G$ and a Borel subgroup $B$. Denote by $R$ the set of simple roots and $\Gamma$ the Dynkin diagram. There is a one to one correspondence between subsets $I$ of $R$ and parabolic subgroups $P_I$ containing $B$ such that $P_R=B$, $P_\emptyset=G$ and $P_I\subset P_{I'}$ if and only if $I'\subset I$. There is a one to one correspondence between rational homogeneous spaces $G/P_I$ and marked Dynkin diagrams defined by marking nodes $I$ in the Dynkin diagram of $G$. One can see the diagrams on page \pageref{diag. marked Dynkina An} intuitively. The following proposition reduces the Fano deformation rigidity of $G/P_I$ to that property of its homogeneous submanifolds.

\begin{prop}\label{prop. intro reduce to homog submfds}
Let $G$ be a simple algebraic group of type $ADE$, and $I$ be a subset of $R$ with cardinality $|I|\geq 3$. Suppose that for any $\alpha\neq\beta\in I$, there exists a subset $A\subset I$ such that $\alpha,\beta\in A$ and the rational homogeneous space $P_{I\setminus A}/P_I$ is rigid under Fano deformation. Then $G/P_I$ is rigid under Fano deformation.
\end{prop}

Note that in the proposition above the variety $P_{I\setminus A}/P_I$ is a rational homogeneous space whose Picard number is $|A|\leq |I|=\rho(G/P_I)$. By Proposition \ref{prop. intro reduce to homog submfds}, an easy analysis of marked Dynkin diagrams shows that Theorem \ref{thm. intro submaximal Picard numbers} is a direct consequence of Theorems \ref{thm. intro Picard numbe one}, \ref{thm. intro (Am, a1am)}, \ref{thm. intro complete flag mfds} and the following

\begin{prop}\label{prop. intro A4 D5 submaximal rigidity}
The rational homogeneous spaces $A_4/P_{I'}$ and $D_5/P_{I''}$ are rigid under Fano deformation, where $|I'|=3$ and $|I''|=4$ respectively.
\end{prop}

As an example we analysis the Fano deformation rigidity of $D_4/P_I$ with $I=\{\alpha_1, \alpha_3, \alpha_4\}$. Given any two different roots $\alpha, \beta\in I$, the rational homogeneous space $P_{I\setminus\{\alpha, \beta\}}/P_I$ is biholomorphic to $A_3/P_{\{\alpha_1, \alpha_3\}}$, which is rigid under Fano deformation. Hence, $D_4/P_I$ is rigid under Fano deformation.

Our argument to Fano deformation rigidity of $A_4/P_{\{\alpha_1, \alpha_2, \alpha_4\}}$, which is a special case of Proposition \ref{prop. intro A4 D5 submaximal rigidity}, works equally well for $A_m/P_{\{\alpha_1, \alpha_2, \alpha_m\}}$ with $m\geq 3$. In other words, we have

\begin{prop}\label{prop. intro (Am, a1a2am) rigidity}
The rational homogeneous spaces $A_m/P_{\{\alpha_1, \alpha_2, \alpha_m\}}$ with $m\geq 3$ are rigid under Fano deformation.
\end{prop}

Applying Proposition \ref{prop. intro reduce to homog submfds}, we have the following consequence.

\begin{thm}\label{thm. intro J connected without end nodes}
Let $G$ be a simple algebraic group of type $ADE$, $\Gamma$ be the Dynkin diagram of $G$, and  $I$ be a subset of the set of simple roots $R$. Denote by $J:=R\setminus I$ and $\bar{\alpha}$ the node with three branches in $\Gamma$ (of type $DE$). Suppose $J$ contains no end nodes of $\Gamma$, the subdiagram with nodes $J$ are connected, and there is at most one $\beta\in J$ with Cartan pairing $\langle\beta, \bar{\alpha}\rangle\neq 0$. Then the rational homogeneous space $G/P_I$ is rigid under Fano deformation.
\end{thm}

If $G$ is of type $AD$ in Theorem \ref{thm. intro J connected without end nodes}, the manifolds $G/P_I$ are exact $\mbF(0, 1, \ldots, k_1, k_2, k_2+1, \ldots, n-1, \mbP^n)$ and $\mbF(0, 1, \ldots, k_1, k_2, k_2+1, \ldots, n, Q^{2n+2})$ with $0\leq k_1<k_2\leq n-1$.

Now let us explain the proof of Propositions \ref{prop. intro reduce to homog submfds}, \ref{prop. intro A4 D5 submaximal rigidity} and \ref{prop. intro (Am, a1a2am) rigidity}. It is well-known that the local deformation rigidity of rational homogeneous spaces follows from the vanishing $H^1(G/P_I, T_{G/P_I})=0$, which is a consequence of Borel-Weil-Bott theorem. So we only need to discuss in the following Setting \ref{setup. intro Fano deformation} and show $\mcX_0\cong \bS$ in each corresponding case.

\begin{set}\label{setup. intro Fano deformation}
Let $\pi: \mcX\rightarrow\Delta\ni 0$ be a holomorphic map such that $\mcX_t\cong\bS:=G/P_I$ for $t\neq 0$ and $\mcX_0$ is a connected Fano manifold, where $G$ is a connected simple algebraic group of $ADE$ type and $I\subset R$. Here $R$ is the set of simple roots and we define $J:=R\setminus I$.
\end{set}

The key point to prove Propositions \ref{prop. intro reduce to homog submfds}, \ref{prop. intro A4 D5 submaximal rigidity} and \ref{prop. intro (Am, a1a2am) rigidity} is the study of symbol algebras. Given a distribution $\mcV$ on a complex manifold $Y$, the weak derived system $\mcV^{-k}$ gives rise to a filtration $\mcV^0\subset\mcV^{-1}\subset\mcV^{-2}\subset\cdots$, where $\mcV^0:=0$, $\mcV^{-1}:=\mcV$, and $\mcV^{-k-1}:=\mcV^{-k}+[\mcV^{-1}, \mcV^{-k}]$ for $k\geq 1$. In an open neighborhood of a general point $y\in Y$ these $\mcV^{-k}$'s are subbundles of $TY$. The graded vector space $\Symb_y(\mcV):=\oplus_{k\geq 1}\mcV^{-k}_y/\mcV^{-k+1}_y$ is a graded nilpotent Lie algebra, and called the symbol algebra of $\mcV$ at $y$.

Let $\mfg_{-1}(\bS)$ be the sum of all $G$-invariant minimal distributions on $\bS$. The subscript $-1$ in the notation $\mfg_{-1}(\bS)$ comes from the grading induced by $I$, see Subsection \ref{subsection. Homogeneous VMRT and minimal distribution}. There is a meromorphic distribution $\mfg_{-1}(\mcX)\subset T^\pi$ such that its singular locus is a (possibly reducible) proper closed subvariety of $\mcX_0$ and its restriction on $\mcX_t$ with $t\neq 0$ coincides with the distribution $\mfg_{-1}(\bS)$.

It is known that $\Symb_s(\mfg_{-1}(\bS))\cong\mfg_-(I)$, where $s$ is any point of $\bS$, $\mfg_-(I)$ is the nilradical of the Lie algebra of $P_I^-$, and $P_I^-$ the opposite parabolic group of $P_I$. By the works of A. \v{C}ap and H. Schichl \cite{CS00} and K. Yamaguchi \cite{Yam93}, we can conclude the following

\begin{prop}\label{prop. intro symbol algebra criterion}
Suppose in Setting \ref{setup. intro Fano deformation} that $|I|\geq 3$ and $\Symb_x(\mfg_{-1}(\mcX_0))\cong\mfg_-(I)$ at general points $x\in\mcX_0$. Then $\mcX_0\cong \bS$.
\end{prop}

We can complete the proof of Propositions \ref{prop. intro reduce to homog submfds} and \ref{prop. intro A4 D5 submaximal rigidity} by applying Proposition \ref{prop. intro symbol algebra criterion} and the following

\begin{prop}\label{prop. intro symbol alg standard}
If the manifold $\bS$ in Setting \ref{setup. intro Fano deformation} is the variety $G/P_I$ in Propositions \ref{prop. intro reduce to homog submfds}, \ref{prop. intro A4 D5 submaximal rigidity} or \ref{prop. intro (Am, a1a2am) rigidity}, then $\Symb_x(\mfg_{-1}(\mcX_0))\cong\mfg_-(I)$ at general points $x\in\mcX_0$.
\end{prop}

To prove Proposition \ref{prop. intro symbol alg standard}, we need the algebraic and geometric feature of each situation. As an example, we suppose $\bS=A_m/P_{\{\alpha_1, \alpha_2, \alpha_m\}}$ in Setting \ref{setup. intro Fano deformation}. It can be shown that any two points in $\mcX_0$ can be jointed by chains of rational curves tangent to $\mfg_{-1}(\mcX_0)$. Hence the tangent bundle $T\mcX_0$ is $k$-th weak derivative of $\mfg_{-1}(\mcX_0)$ for some $k$. In particular, $\dim \Symb_x(\mfg_{-1}(\mcX_0))=\dim\mcX_0=\dim\mfg_-(I)$ at a general point $x\in\mcX_0$. One the other hand, if the symbol algebra $\Symb_x(\mfg_{-1}(\mcX_0))\ncong\mfg_-(I)$, then an easy calculation of Lie algebras shows that $\dim \Symb_x(\mfg_{-1}(\mcX_0))<\dim\mfg_-(I)$. The contradiction implies that $\Symb_x(\mfg_{-1}(\mcX_0))\cong\mfg_-(I)$.

The organization of this paper is as follows. In Section \ref{section. Nilpotent algebras and symbol algebras} by studying the $G$-action on family of rational curves and the $G$-invariant minimal distributions on $\bS$ we give a characterization of $\mfg_-(I)$, which is a variation of Serre's theorem on simple Lie algebras. In Section \ref{section. An reduction theorem for rigidity under Fano deformations} we firstly study the basic properties of Fano deformations and symbol algebras in Setting \ref{setup. intro Fano deformation}, and then prove Proposition \ref{prop. intro symbol algebra criterion}. With the help of Proposition \ref{prop. intro symbol algebra criterion} and the characterization of $\mfg_-(I)$, we give the proof of Proposition \ref{prop. intro reduce to homog submfds} in Section \ref{section. An reduction theorem for rigidity under Fano deformations}. In Section \ref{section. Rigidity and degeneration under Fano deformation} we prove the rigidity results by applying Proposition \ref{prop. intro reduce to homog submfds}. In Subsection \ref{subsection. Rigidity property and its reductions} we prove Theorems \ref{thm. intro submaximal Picard numbers} and \ref{thm. intro J connected without end nodes} by assuming Propositions \ref{prop. intro A4 D5 submaximal rigidity} and \ref{prop. intro (Am, a1a2am) rigidity}. The proof of Proposition \ref{prop. intro (Am, a1a2am) rigidity} is given in Subsection \ref{subsection. rigidity of (A4, (a2, a3, a4))}. Theorem \ref{thm. intro unique degeneration of F(1, 2, C4)} is proved in Section \ref{section. unique degeneration of (A3, a1, a2)}, and with the help of this theorem we prove Proposition \ref{prop. intro A4 D5 submaximal rigidity} in Subsection \ref{subsection. rigidity of (A4, (a2, a3, a4))}. Finally we analysis the possible Fano degeneration of $\mbF(1, 2, Q^6)$.

\section{Geometry on rational homogeneous spaces} \label{section. Nilpotent algebras and symbol algebras}

\subsection{Distributions and families of lines} \label{subsection. Homogeneous VMRT and minimal distribution}

In this subsection, we collect the geometric properties on rational homogeneous spaces, which are useful in this paper. These results are classical, and most of them are stated without proof.

\begin{set}\label{setup. G semisimple}
Let $G$ be a connected semisimple algebraic group of adjoint type such that each simple factor is of type $ADE$, $B$ be a Borel subgroup, and $R$ be the set of simple roots. Fix a subset $I$ of $R$ and denote by $J:=R\setminus I$.
\end{set}

Denote by $P_I:=\bigcap\limits_{\alpha\in I}P_\alpha$, where $P_\alpha$ is the associated maximal parabolic subgroup of $G$ which contains $B$. Denote by $P_I^-$ the opposite parabolic subgroup of $P_I$, and $G_0:=P_I\cap P_I^-$.

\begin{defi}\label{defi. gk(I) g-(I) and g-(G)}
Denote by $\mfg$ the Lie algebra of $G$. Let $\Lambda$ be the set of all roots of $G$ and $\mfh$ the fixed Cartan subalgebra of $\mfg$. For each $\eta\in\Lambda$, denote by $\mfg_\eta$ the 1-dimensional linear subspace of $\mfg$ with weight $\eta$. We can write $\eta=\sum\limits_{\alpha\in R}n_\alpha\alpha$, where either all $n_\alpha$ are nonnegative integers or all $n_\alpha$ are nonpositive integers. Define $\deg_I\eta=\sum\limits_{\alpha\in I}n_\alpha$. For each $k\in\mbZ$ denote by $\Lambda_k(I)$ the set of elements $\eta\in\Lambda$ with $\deg_I(\eta)=k$. Equip a grading on $\mfg$ such that $\mfg_k(I):=\bigoplus\limits_{\eta\in\Lambda_k(I)}\mfg_\eta$ for $k\neq 0$ and $\mfg_0(I):=\mfh\oplus(\bigoplus\limits_{\eta\in\Lambda_0(I)}\mfg_\eta)$. Then $\mfg$ becomes a graded Lie algebra. Moreover $\mfg_0$, $\mfp_I:=\bigoplus\limits_{k\geq 0}\mfg_k$ and $\mfp^-_I:=\bigoplus\limits_{k\leq 0}\mfg_k$ are Lie algebras of $G_0$, $P_I$ and $P_I^-$ respectively. When there is no confusion, we omit $I$ in the expressions, for example $\mfg_k:=\mfg_k(I)$. In case $I=R$, we may also write $\mfg_-(G):=\mfg_-(R)$ in order to emphasize on the group $G$.
\end{defi}

A rational homogeneous space can be expressed by a marked Dynkin diagram. To explain the order of simple roots and the way to express a rational homogeneous space, we draw the marked Dynkin diagram corresponding to $G/P_{\{\alpha_1, \alpha_2\}}$ as follows, where $G=A_n, D_m$ or $E_k$ with $n\geq 2$, $m\geq 4$ and $k=6, 7, 8$ respectively.

\begin{eqnarray}\label{diag. marked Dynkina An}
\xymatrix{ A_n:  &
{\begin{array}{@{}c@{}}\\\bullet\\\alpha_1\end{array}} \POS+/r2pt/ \ar@{-}[r]+/l2pt/ &
{\begin{array}{@{}c@{}}\\\bullet\\\alpha_2\end{array}} \POS+/r2.5pt/ \ar@{-}[r]+/l2pt/ &
{\begin{array}{@{}c@{}}\\\circ\\\alpha_3\end{array}} \POS+/r2pt/ \ar@{-}[r]+/l2pt/ &
{\begin{array}{@{}c@{}}\\\circ\\\alpha_4\end{array}} \POS+/r2pt/ \ar@{--}[r]+/l2pt/ &
{\begin{array}{@{}c@{}}\\\circ\\\alpha_{n-1}\end{array}} \POS+/r2pt/ \ar@{-}[r]+/l2pt/ &
{\begin{array}{@{}c@{}}\\\circ\\\alpha_n\end{array}}
}
\end{eqnarray}

\begin{eqnarray}\label{diag. marked Dynkina Dm}
\xymatrix{
 D_m: & {\begin{array}{@{}c@{}}\\\bullet\\\alpha_1\end{array}} \POS+/r2pt/ \ar@{-}[r]+/l2pt/ &
{\begin{array}{@{}c@{}}\\\bullet\\\alpha_2\end{array}} \POS+/r2.5pt/ \ar@{-}[r]+/l2pt/ &
{\begin{array}{@{}c@{}}\\\circ\\\alpha_3\end{array}} \POS+/r2pt/ \ar@{-}[r]+/l2pt/ &
{\begin{array}{@{}c@{}}\\\circ\\\alpha_4\end{array}} \POS+/r2pt/ \ar@{--}[r]+/l2pt/ &
{\begin{array}{@{}c@{}}\\\circ\\\alpha_{m-2}\,\,\,\,\,\,\end{array}} \POS+/r2pt/ \ar@{-}[r]+/l2pt/&
{\begin{array}{@{}c@{}}\\\circ\\\alpha_{m-1}\end{array}} \\
&&&&& {\begin{array}{@{}c@{}}\\\,\,\,\,\,\,\circ\\\alpha_m\end{array}}\POS+/r3.5pt/\ar@{-}[u]+/l0pt/ &
}
\end{eqnarray}

\begin{eqnarray}\label{diag. marked Dynkina Ek}
\xymatrix{ E_k: &
{\begin{array}{@{}c@{}}\\\bullet\\\alpha_1\end{array}} \POS+/r2pt/ \ar@{-}[r]+/l2pt/ &
{\begin{array}{@{}c@{}}\\\circ\\\alpha_3\end{array}} \POS+/r2.5pt/ \ar@{-}[r]+/l2pt/ &
{\begin{array}{@{}c@{}}\\\circ\\\alpha_4\end{array}} \POS+/r2pt/\ar@{--}[r]+/l2pt/ &
{\begin{array}{@{}c@{}}\\\circ\\\alpha_{k-1}\end{array}} \POS+/r2pt/ \ar@{-}[r]+/l2pt/ &
{\begin{array}{@{}c@{}}\\\circ\\\alpha_k\end{array}}  \\
&&&{\begin{array}{@{}c@{}}\\\bullet\\\alpha_2\end{array}} \POS+/r0pt/\ar@{-}[u]+/l0pt/&&
}
\end{eqnarray}

Since we assume $G$ to be of adjoint type, the restriction of Adjoint representation induces an injective homomorphism $G_0\subset GL(\mfg_-(I))$, where $\mfg_-(I):=\bigoplus\limits_{k\geq 1}\mfg_{-k}(I)$.

\begin{nota}\label{nota. Dynkin subdiagram and associated semisimple subgroups}
Given any $\alpha\in R$, denote by $N(\alpha)$ the set of simple roots that are next to $\alpha$ in the Dynkin diagram $\Gamma_R$ of $G$, and set $N_J(\alpha):=N(\alpha)\cap J$. For each subset $A$ of $R$, denote by a semisimple subgroup $G_A$ of $G$ associated to the Dynkin subdiagram $\Gamma_A$ of $\Gamma_R$.
\end{nota}

\begin{defi}\label{defi. homogenous VMRT}
Set $\bS:=G/P_I$. Given a subset $A\subset I$, denote by $\bS^A$ the central fiber of $\Phi^A: \bS\rightarrow G/P_{I\setminus A}$, which is a rational homogeneous space of Picard number $|A|$. Given any $\alpha\in I$, each fiber of $\Phi^\alpha$ is biholomorphic to $\bS^\alpha$, which is covered by lines under its minimal embedding. Denote by $\K^\alpha(\bS)$ the family of these lines (associated with $\alpha$) on $\bS$. Indeed $\K^\alpha(\bS)=G/P_{(I\cup N(\alpha))\setminus\{\alpha\}}$, which can be concluded from the following commutative diagram of Tits fibrations
\begin{eqnarray}\label{eqn. Tits fibrations}
\xymatrix{G/P_{I\cup N(\alpha)}\ar[d]_-{\mu}\ar[r]^-{\ev} & G/P_I\ar[d] \\
G/P_{(I\cup N(\alpha))\setminus\{\alpha\}}\ar[r] & G/P_{I\setminus\{\alpha\}}.
}
\end{eqnarray}

Denote by $\mcC^\alpha(\bS)\subset\mbP(T\bS)$ the variety of tangent directions of $\K^\alpha(\bS)$, i.e. at each point $x\in\bS$,
$$\mcC^\alpha_x(\bS)=\bigcup\limits_{C\in\K^\alpha_x(\bS)}\mbP(T_x C)\subset\mbP(T_x(\bS)),$$
Denote by $\bZ^\alpha:=\mcC_p^\alpha(\bS)\subset\mbP(T_p\bS)$, where $p$ is  the base point of $\bS$.
\end{defi}

\begin{rmk}
Tits \cite{Tits55} studied diagrams in the style of \eqref{eqn. Tits fibrations} and he called $\mu(ev^{-1}(z))$ the shadow of $z\in G/P_I$. This variety is biholomorphic to $\mcC^\alpha_x(\bS)$ for $x\in ev^{-1}(z)\subset G/P_I$. The notation $\mcC^\alpha_x(\bS)$ called the variety of minimal rational tangents (VMRT for short) at $x$ of the minimal rational component $\K^\alpha(\bS)$. One could find more details about minimal rational components and VMRT in \cite{Hw01}. For more details about the properties of those $\mcC^\alpha$, one can consult \cite{LM03}.
\end{rmk}

The following results are straight-forward.

\begin{lem}
(i) The group $G_J$ is the semisimple part of the reductive group $G_0$.

(ii) The simple factor of $G_J$ is of type $ADE$;

(iii) There is a natural $G_0$-action on $\bC^\alpha$, and the action is transitive.
\end{lem}

\begin{lem}\label{lem. minimal a-distribution}
The tangent bundle of $\bS$ is identified with $G\times^{P_I}(\mfg/\mfp_I)$. For each $\alpha\in I$ there exists a unique $G$-invariant holomorphic distribution
$$\mfg^\alpha(\bS):=G\times^{P_I}((\mfg_{-1}(\alpha)+\mfp_I)/\mfp_I).$$ The $G$-invariant holomorphic distribution $$\mfg_{-1}(\bS):=G\times^{P_I}((\mfg_{-1}(I)+\mfp_I)/\mfp_I)$$ satisfies that
$$\mfg_{-1}(\bS)=\bigoplus\limits_{\alpha\in I}\mfg^{\alpha}(\bS)=\sum\limits_{\alpha\in I}\mfg^\alpha(\bS)\subset T\bS.$$
\end{lem}

\begin{lem}\label{lem. property of a-VMRT}
Take any $\alpha\in I$. Then

(1) $\mcC^\alpha(\bS)\subset \mbP(\mfg^\alpha(\bS))$;

(2) The inclusion $\bZ^\alpha\subset\mbP(\mfg_{-1}(\alpha))$ is $G_0$-equivariant;

(3) $\bZ^\alpha$ is the unique closed $G_0$-orbit in $\mbP(\mfg_{-1}(\alpha))$;

(4) $\bZ^\alpha$ is nondegenerate in $\mbP(\mfg_{-1}(\alpha))$.

(5) The $G_J$-action on $\bZ^\alpha$ induces the isomorphism
$$\bZ^\alpha\cong G_J/P_{N_J(\alpha)}\cong \prod\limits_{\beta\in N_J(\alpha)}G_J/P_\beta.$$
Particularly each $\bZ^\alpha_\beta:=G_J/P_\beta\cong G_0/P_\beta$ is a rational homogeneous space of Picard number one, where $\alpha\in I$ and $\beta\in N_J(\alpha)$.
\end{lem}

\begin{rmk}
Given $\alpha\in I$ and $\beta\in N_J(\alpha)$, as in Definition \ref{defi. homogenous VMRT} we have a family of rational curves $\K^\beta(\bZ^\alpha)$ and its associated variety of tangent directions $\mcC^\beta(\bZ^\alpha)$ on the rational homogeneous space $\bZ^\alpha$. As in Lemma \ref{lem. minimal a-distribution} we can construct the distribution $\mfg^\beta(\bZ^\alpha)$ on $\bZ^\alpha$, which is the minimal $G_0$-invariant (hence $G_J$-invariant) distribution associated with the root $\beta\in N_J(\alpha)\subset J$.
Denote by $\wbZa\subset\mfg_{-1}(\alpha)$ the affine cone of $\bZ^\alpha\subset\mbP(\mfg_{-1}(\alpha))$. Then can define $\wcCa(\bS)$, $\mcC^\beta(\wbZa)$, $\wcC^\beta(\wbZa)$ and $\mfg^\beta(\wbZa)$ in an obvious way.
\end{rmk}

\begin{nota}
Write $J=\bigcup\limits_{1\leq i\leq \tau}J_i$, which is a disjoint union such that $\Gamma_{J_1},\ldots,\Gamma_{J_\tau}$ are the connected components of the Dynkin diagram of $\Gamma_J$.
\end{nota}

The following is straight-forward.

\begin{lem}\label{lem. Hab}
Take any $\alpha\in I$, and any $\beta\in N_J(\alpha)$. Then there exists a unique $J_i$ containing $\beta$. Moreover, $\beta$ is an end vertex of the Dynkin diagram $\Gamma_{J_i}$, and $\bZ^\alpha_\beta\cong G_{J_i}/P_\beta$.
\end{lem}


The following result on automorphism groups of rational homogeneous spaces is straight-forward.

\begin{lem}\label{lem. aut(G PI)}\label{cor. Aut(Ha)}\label{cor. gb(Ha) is Aut(Ha) invariant}
The natural homomorphism $G\rightarrow\Aut^o(\bS)$ is bijective. Take $\alpha\in I$ and let $\Aut^o(\wbZa, \mfg_{-1}(\alpha))$ be the identity component of
$$\Aut(\wbZa, \mfg_{-1}(\alpha)):=\{\varphi\in GL(\mfg_{-1}(\alpha))\mid \varphi\cdot\wbZa=\wbZa\}.$$ Then the natural homomorphism $G_0\rightarrow\Aut^o(\wbZa)$ is surjective and $\Aut^o(\wbZa)=\Aut^o(\wbZa, \mfg_{-1}(\alpha))$. Take any $\beta\in N_J(\alpha)$. Then the distribution $\mfg^\beta(\bZ^\alpha)$ on $\bZ^\alpha$ is $\Aut^o(\bZ^\alpha)$-invariant.
\end{lem}

Given $\alpha, \beta\in I$, $k\geq 1$ and $[C^\alpha]\in \K^\alpha(\bS)$, we can describe the splitting type along $C^\alpha\cong\mbP^1$ of distributions $\mfg_{-k}^\beta(\bS)$ on $\bS$. To obtain it, we need to apply Grothendieck¡¯s splitting theorem for principal bundles on $\mbP^1$ with reductive structure groups and associated vector bundles \cite{Groth57}.

\begin{prop}\label{prop. Grothendieck vector bundles on P1} (Grothendieck). Let $\mcO(1)^*$ be the $\C^*$-principal bundle on $\mbP^1$ corresponding to the line bundle $\mcO(1)$. Let $L$ be a reductive complex Lie group. Up to conjugation, any $L$-principal bundle on $\mbP^1$ is associated to $\mcO(1)^*$ by a group homomorphism from $\C^*$ to a maximal torus of $L$.
If $E$ is the coroot of $sl(2)$, such a group homomorphism is determined by the image of $E$ in $\mfh$, a fixed Cartan subalgebra of $L$. Given a representation of $L$ with weights $\mu_1, \ldots, \mu_\ell\in\mfh^*$, the associated vector bundle on $\mbP^1$ splits as $\mcO(\mu_1(E))\oplus\cdots\oplus\mcO(\mu_\ell(E))$, where $\mu_j(E)$ denotes the value of $\mu_j$ on the image of $E$ in $\mfh$.
\end{prop}

Given $\beta\in I$ and $[C^\alpha]\in \K^\alpha(\bS)$, we can identify $C^\alpha$ with $\exp(sl_\beta(2))/\exp(\mfg_\beta\oplus[\mfg_\beta, \mfg_{-\beta}])$, where $sl_\beta(2):=\mfg_\beta\oplus\mfg_{-\beta}\oplus[\mfg_\beta, \mfg_{-\beta}]\subset\mfg$ is a subalgebra isomorphic to $sl(2)$, and $\mfg_\beta\oplus[\mfg_\beta, \mfg_{-\beta}]=\mfp_I\cap sl_\beta(2)$ is a Borel subalgebra. Then as a direct consequence of Proposition \ref{prop. Grothendieck vector bundles on P1}, we have the following result.

\begin{prop}\label{prop. gk(b) along Ca splitting types}
Given $\alpha, \beta\in I$, $k\geq 1$ and $[C^\alpha]\in \K^\alpha(\bS)$, we have
\begin{eqnarray*}\label{eqn. gk(b) along Ca splitting types}
\mfg_{-k}^\beta(\bS)|_{C^\alpha}=\bigoplus\limits_{\gamma\in \Lambda_k(\beta)}\mcO_{\mbP^1}(\langle\gamma, \alpha\rangle).
\end{eqnarray*}
\end{prop}

\subsection{Characterization of the nilradical of a parabolic subalgebra} \label{subsection. Description of symbol algebra}

We want to give a description of the algebra $\mfg_-(I):=\bigoplus\limits_{k\geq 1}\mfg^I_{-k}$. When $I=R$, it is described by Serre's theorem on semisimple Lie algebra in the following way.

\begin{prop}\cite[Section 18]{Hum72} \label{prop. Serre's theorem}
Let $R$ be a set of simple roots for $\mfg$ and choose a nonzero element $x_\alpha\in\mfg_{-\alpha}$ for each $\alpha\in R$. Then the subalgebra $\mfg_-(R)$ of $\mfg$ is the quotient of the free Lie algebra generated by $\{x_\alpha\mid \alpha\in R\}$ by the relations
\begin{eqnarray*}
\ad(x_\alpha)^{-\langle\beta, \alpha\rangle + 1}(x_\beta)=0 \mbox{ for all } \alpha\neq\beta\in R.
\end{eqnarray*}
\end{prop}

\begin{prop}\label{prop. characterizing g-}
Denote by $\mbF(\mfg_{-1}(I))$ the free graded Lie algebra generated by $\mfg_{-1}(I)$. Fix an arbitrary $z_\alpha\in\wbZa\setminus\{0\}$ for each $\alpha\in I$. Let $\mcI:=\mcI(z_\alpha, \alpha\in I)$ be the ideal of $\mbF(\mfg_{-1}(I))$ generated by the following relations:

(i) for all $\alpha'\neq\alpha''\in I$ and all $(v', v'')\in G_0\cdot (z_{\alpha'}, z_{\alpha''})\in\wbZ^{\alpha'}\times\wbZ^{\alpha''}$,
$$(\ad v')^{-\langle\alpha'', \alpha'\rangle + 1}(v'')=0;$$

(ii) for all $\alpha\in I$, $\beta\in N_J(\alpha)$, $v\in\wbZa\setminus\{0\}$, and $u\in\mfg^\beta_v(\wbZa)$,
$$(\ad v)^{-\langle\beta, \alpha\rangle}(u)=0.$$
Then $\mfg_-(I):=\bigoplus\limits_{i\geq 1}\mfg_{-i}$ is isomorphic to $\mbF(\mfg_{-1}(I))/\mcI$ as graded nilpotent Lie algebra. In particular, up to isomorphism $\mbF(\mfg_{-1}(I))/\mcI(z_\alpha, \alpha\in I)$ is independent of the choice of those $z_\alpha\in\wbZa\setminus\{0\}$.
\end{prop}

\begin{proof}
Step 1. We will show that $\mfg_-(I)$ satisfies conditions $(i)$ and $(ii)$.

The inclusion $\mfg_-(R)\subset\mfp^-_I$ induces a semidirect product decomposition of Lie algebra structure $\mfg_-(R)=\mfn_0\rtimes\mfg_-(I)$, where $\mfn_0:=\mfg_-(R)\cap\mfg_0(I)$. For each $\alpha\in R$ we choose a nonzero element $x_\alpha\in\mfg_{-\alpha}$.

For those $\alpha\in I$, we write $z_\alpha:=x_\alpha$. Since the point $\mbP(\mfg_{-\alpha})\in \bZ^\alpha\subset\mbP(\mfg_{-1}^\alpha)$, we have $z_\alpha\in \wbZa\setminus\{0\}$. For those $\beta\in J:=R\setminus I$, we have $x_\beta\in\mfn_0$. Denote by $\pi: \mfn_0\subset\mfg_0(I)\rightarrow\aut(\mfg_-(I))$ the homomorphism induced by the adjoint representation, and write $\eta_\beta:=\pi(x_\beta)\in\aut(\mfg_-(I))$. Then by Proposition \ref{prop. Serre's theorem}, we have
\begin{eqnarray}
(\ad z_{\alpha'})^{-\langle\alpha'', \alpha'\rangle+1}(z_{\alpha''})=0, & \mbox{ for all } \alpha'\neq\alpha''\in I, \label{eqn. va1 va2 bracket} \\
(\ad z_{\alpha})^{-\langle\beta, \alpha\rangle}(\eta_\beta(z_{\alpha}))=0, & \mbox{ for all } \alpha\in I \mbox{ and } \beta\notin I. \label{eqn. va eta bracket}
\end{eqnarray}
Since the Lie algebra $\mfg_-(I)$ is a $G_0$-module, the conclusion \eqref{eqn. va1 va2 bracket} implies the condition (i) in the statement of Proposition \ref{prop. characterizing g-}.

Now let us check the condition (ii). By \eqref{eqn. va eta bracket}, $\eta_\beta(z_\alpha)=0$ for $\beta\in J\setminus N_J(\alpha)$.  Now suppose that $\beta\in N_J(\alpha)$. Then $\eta_\beta(v_\alpha)=[x_\beta, x_\alpha]$ is a nonzero vector in $\mfg_{-1}(\alpha)=\sum\limits_{\gamma\in\Lambda_{-1}(\alpha)}\mfg_\gamma$. Moreover, $\mbP(\eta_\beta(z_\alpha))$ is a point in $H^\beta_{[z_\alpha]}(\bZ^\alpha)\subset\mbP(\mfg_{-1}(\alpha))$. Since $\mfg_-(I)$ is a $G_0$-module, we have
\begin{eqnarray*}
(\ad (\varphi\cdot z_\alpha))^{-\langle\beta, \alpha\rangle}(\varphi\cdot \eta_\beta(z_\alpha))=0 \quad \mbox{ for all } \varphi\in G_0.
\end{eqnarray*}
Denote by $P'_\beta:=P_\beta\cap G_0$. Then $P'_\beta\cdot z_\alpha\subset\C z_\alpha$, and $P'_\beta\cdot\eta_\beta(z_\alpha)=\wbZ^\beta_{z_\alpha}(\wbZa)$. Since $\wbZ^\beta_{z_\alpha}(\wbZa)$ is nondegenerate in the subspace $\mfg^\beta_{z_\alpha}(\wbZa)$ of $\mfg_{-1}(\alpha)$, we have
\begin{eqnarray*}
\ad(z_\alpha)^{\langle\beta, \alpha\rangle}(u)=0 \mbox{ for all } u\in\mfg^\beta_{z_\alpha}(\wbZa).
\end{eqnarray*}
It follows that
\begin{eqnarray*}
(\ad (\varphi\cdot z_\alpha))^{-\langle\beta, \alpha\rangle}(\varphi\cdot u)=0 \mbox{ for all } \varphi\in G_0 \mbox{ and all } u\in\mfg^\beta_{z_\alpha}(\wbZa).
\end{eqnarray*}
Since $\bZ^\alpha$ is $G_0$-transitive and $\mfg^\beta(\wbZa)$ is $G_0$-equivariant, the condition $(ii)$ holds.

\medskip

Step 2. Show that the isomorphism $\mbF(\mfg_{-1}(I))/\mcI(z_\alpha, \alpha\in I)$ is independent of the choice of those $z_\alpha\in\wbZa\setminus\{0\}$.

Now take any $z'_\alpha\in\wbZa\setminus\{0\}$ for each $\alpha\in I$. Since the inclusion $\wbZa\subset\mfg_{-1}(\alpha)$ is $G_0$-equivariant and $\wbZa\setminus\{0\}$ is a single $G_0$-orbit, there exists an isomorphism $\varphi^\alpha: \mfg_{-1}(\alpha)\rightarrow\mfg_{-1}(\alpha)$ of $G_0$-modules sending $\wbZa$ onto itself and $\varphi^\alpha(z_\alpha)=z'_\alpha$. These $\varphi^\alpha$ induce an isomorphism $\mbF(\mfg_{-1}(I))\rightarrow\mbF(\mfg_{-1}(I))$ whose restriction sending $\mcI(z_\alpha, \alpha\in I)$ onto $\mcI(z'_\alpha, \alpha\in I)$. Hence we have an isomorphism
$$\mbF(\mfg_{-1}(I))/\mcI(z_\alpha, \alpha\in I)\cong \mbF(\mfg_{-1}(I))/\mcI(z'_\alpha, \alpha\in I).$$

\medskip

Step 3. Show the isomorphism $\mbF(\mfg_{-1}(I))/\mcI\cong\mfg_-(I)$.

By Step 1 and Step 2 we can set $z_\alpha:=x_\alpha\in\wbZa\setminus\{0\}$ for each $\alpha\in I$ and get a surjective homomorphism of $G_0$-modules $\psi: \mbF(\mfg_{-1}(I))/\mcI\rightarrow\mfg_-(I)$. It should be noticed that
\begin{eqnarray*}
&& \mcF:=\mfg_0(I)\oplus(\mbF(\mfg_{-1}(I))/\mcI)\cong(\mfg_0(I)\oplus\mbF(\mfg_{-1}(I)))/\mcI, \mbox{ and } \\
&& \mfp^-_I:=\mfg_0(I)\oplus\mfg_-(I)=\bigoplus\limits_{i\leq 0}\mfg_i(I)
\end{eqnarray*}
are both graded Lie algebras as well as $G_0$-modules. Moreover, $\psi$ induces a surjective homomorphism between Lie algebras as well as between $G_0$-modules: $\psi': \mcF\rightarrow\mfp^-_I$.

Similarly as in Step 1 let $\mfn_0$ be the Lie subalgebra of $\mfg_0(I)$ generated by those $x_\beta$ with $\beta\in J$. We have $\mfg_-(R)=\mfn_0\oplus\mfg_-(I)\subset\mfp^-_I$, and set $\wcF:=\mfn_0\oplus(\mbF(\mfg_{-1}(I))/\mcI)\subset\mcF$. Then the restriction of $\psi'$ induces a surjective homomorphism of Lie algebras
$$\wpsi: \wcF\rightarrow\mfg_-(R).$$

Denote by $\wbF$ the free graded Lie algebra generated by those $x_\gamma$ with $\gamma\in R$. Let $\wcI$ be the ideal of $\wbF$ generated by the set
\begin{eqnarray*}
\{(\ad x_{\gamma'})^{-\langle\gamma'', \gamma'\rangle+1}(x_{\gamma''})\mid \gamma'\neq\gamma''\in R\}.
\end{eqnarray*}
There is a commutative diagram of Lie algebras as follows:
\begin{eqnarray*}\label{eqn. commutative diagram nilpotent algebras}
\xymatrix{\wbF\ar[d]_-{\theta_1}\ar[rd]^-{\theta_2} & \\
\wcF\ar[r]_-{\wpsi} & \mfg_-(R).
}
\end{eqnarray*}
We claim that $\theta_1(\wcI)=0$. Equivalently we claim that for all $\gamma'\neq\gamma''\in R$,
\begin{eqnarray}\label{eqn. claim ideal to ideal}
\theta_1((\ad x_{\gamma'})^{-\langle\gamma'', \gamma'\rangle+1}(x_{\gamma''}))=0.
\end{eqnarray}

Case 1. Suppose $\gamma',\gamma''\in I$. Recall our definition of $\mcI$ for $z_\alpha:=x_\alpha\in\wbZa\setminus\{0\}$. Then in this case \eqref{eqn. claim ideal to ideal} follows from the condition $(i)$ of Proposition \ref{prop. characterizing g-}.

Case 2. Suppose $\gamma'\in I$ and $\gamma''\in J:=R\setminus I$. The condition $(ii)$ of Proposition \ref{prop. characterizing g-} implies \eqref{eqn. claim ideal to ideal} under the additional assumption that $\gamma''\in N_J(\gamma')$ by noting that $\theta_1(x_{\gamma'})\in\mfg_{-1}(I)$ and $\theta_1(x_{\gamma''})\in\mfn_0$.

Now for $\gamma'\in I$ and $\gamma''\in J\setminus N_J(\gamma')$, we have $\langle \gamma'', \gamma'\rangle=0$. The \eqref{eqn. claim ideal to ideal} becomes that $[\theta_1(x_{\gamma''}), \theta_1(x_{\gamma'})]_{\wcF}=0$. The latter can be deduced from the $\mfg_0$-action (hence the $\mfn_0$-action) on $\mfg_{-1}(I)$.

Case 3. Suppose $\gamma'\in J:=R\setminus I$ and $\gamma''\in I$. Similarly in this case \eqref{eqn. claim ideal to ideal} can also be deduced from the $\mfg_0$-action (hence the $\mfn_0$-action) on $\mfg_{-1}(I)$.

Case 4. Suppose $\gamma', \gamma''\in J:=R\setminus I$.
In this case \eqref{eqn. claim ideal to ideal} can also be deduced from the Lie algebra structure of $\mfn_0$ (coming from that of $\mfg_0$).

In summary, the claim $\theta_1(\wcI)=0$ holds. Then it induces a homomorphism
$$\wtheta_1: \wbF/\wcI\rightarrow\wcF.$$

By the construction of $\wcF$, the morphism $\theta_1$ is surjective. Hence $\wtheta_1$ is surjective. By Proposition \ref{prop. Serre's theorem}, $\theta_2$ induces an isomorphism $\wtheta_2: \wbF/\wcI\cong\mfg_-(R)$. Hence $\wpsi$ is an isomorphism preserving gradings, and its restriction gives an isomorphism of graded nilpotent Lie algebras $\mbF(\mfg_{-1}(I))/\mcI\cong\mfg_-(I)$.
\end{proof}

\section{Fano deformation of rational homogeneous spaces} \label{section. An reduction theorem for rigidity under Fano deformations}

From now on, we study the family $\mcX$ over $\Delta$ in Setting \ref{setup. intro Fano deformation}. The organization of this section is as follows. In subsection \ref{subsection. minimal rational curves on the family}, we study the basic property of minimal rational curves and Cartier divisors on the family $\mcX/\Delta$. In subsection \ref{subsection. Deformation rigidity via parabolic geometry}, we study the property of symbol algebras and prove Proposition \ref{prop. intro symbol algebra criterion}, which is reformulated in Proposition \ref{prop. symbol algebra standard iff variety standard}. In subsection \ref{subsection. reduction to homogeneous submanifolds}, we prove Theorem \ref{thm. reduction theorem}, which implies
Proposition \ref{prop. intro reduce to homog submfds} as a corollary.

\subsection{Minimal rational curves on the family}\label{subsection. minimal rational curves on the family}

The following result is due to Wi\'{s}niewski \cite{Wis09}.

\begin{prop}\cite[Theorem 1]{Wis09}\label{prop. invariance of Mori cone by Wisniewski}
We can identify the Mori cones $\overline{NE}(\mcX/\Delta)=\overline{\mcX_t}$ for all $t\in\Delta$.
\end{prop}

The following is a classical result on the rational homogeneous space $\bS:=G/P_I$.

\begin{lem}\label{lem. Mori cone on rat homog spaces}
The Mori cone $\overline{NE}(\bS)$ is a simplicial cone generated by those $R_\alpha:=\mbR^+[\K^\alpha(\bS)]$ with $\alpha\in I$ i.e. $\dim\overline{NE}(\bS)$ equals to the cardinality of $I$, and $\overline{NE}(\bS)=\sum\limits_{\alpha\in I} R_\alpha$, where $\K^\alpha(\bS)$ is as in Definition \ref{defi. homogenous VMRT}. The set of extremal faces of $\overline{NE}(\bS)$ can be identified with the set of subsets of $I$.
\end{lem}

As a direct consequence of Proposition \ref{prop. invariance of Mori cone by Wisniewski} and Lemma \ref{lem. Mori cone on rat homog spaces}, we have the following result.

\begin{prop-defi}\label{prop. extending phi(A) to pi(A)}
For each $A\subset I$, denote by $\Phi^{A}:\bS\rightarrow G/P_{I\setminus A}$ the Mori contraction associated with the extremal face $\sum\limits_{\alpha\in A} R_\alpha$ of $\overline{NE}(\bS)$. We can extend it to be a relative Mori contraction $\pi^{A}: \mcX\rightarrow\mcX^{A}$. We denote by $\pi^A_t:=\pi^A|_{\mcX_t}$ for each $t\in\Delta$.
\end{prop-defi}

\begin{nota}
Given a subset $A\subset I$ and a point $x\in\mcX$, denote by $F^A_x$ the fiber of $\pi^A: \mcX\rightarrow\mcX^A$ passing through $x$. In particular, if $x\notin\mcX_0$ then $F^A_x\cong\bS^A$, where $\bS^A$ is defined in Definition \ref{defi. homogenous VMRT}.
\end{nota}

\begin{prop}\label{prop. Fax=Sa for x in X0 general}
Take any $\alpha\in I$. Then $F^\alpha_x\cong\bS^\alpha$ for $x\in\mcX_0$ general.
\end{prop}

\begin{proof}
The fiber $F^\alpha_x$ is a smooth Fano deformation of $\bS^\alpha$. Then the conclusion follows from the Fano deformation rigidity of $\bS^\alpha$, which is obtained by J.-M. Hwang and N. Mok \cite[Main Theorem]{HM02}.
\end{proof}

By Proposition \ref{prop. invariance of Mori cone by Wisniewski}, Proposition \ref{prop. Fax=Sa for x in X0 general} and intersection theory on rational homogeneous spaces, we have the following result.

\begin{prop-defi}\label{prop. divisors intersection theory on family X}
Take any $\alpha\in I$. Denote by $\K^\alpha(\mcX)$ the irreducible component of $\text{Chow}(\mcX)$ extending $\K^\alpha(\bS)$. Take any $[C]\in\K^\alpha(\mcX)$. Then $C$ is an irreducible and reduced rational curve on $\mcX_t$ for a unique $t\in\Delta$. If either $t\neq 0$ or $[C]$ is general in $\K^\alpha(\mcX_0)$, then $C\cong\mbP^1$. Moreover, there exists a unique $\mcL^\alpha\in\Pic(\mcX/\Delta)$ such that $(\mcL^\alpha\cdot\K^\beta(\mcX))=\delta_{\alpha\beta}$ for all $\beta\in I$.
\end{prop-defi}


\begin{prop}\label{prop. minimal rational chain connected}
Any two points $x, y\in \mcX_0$ can be connected by chains of elements in $\bigcup\limits_{\alpha\in I}\K^\alpha(\mcX_0)$.
\end{prop}

\begin{proof}
Consider the rational map $\psi: \mcX_0\dashrightarrow Z$ defined by equivalence relation induced by $\bigcup\limits_{\alpha\in I}\K^\alpha(\mcX_0)$. For the existence and the property of such a rational map, see \cite[Theorem IV.4.16]{Kol96}. Suppose $\dim Z\geq 1$. Take a general divisor $D\subset Z$ and a general point $x\in Z\setminus D$. Then $\psi^{-1}(z)$ is a closed subvariety of $\mcX_0$ which has empty intersection with the indeterminant locus of $\psi$. Thus $\psi^{-1}(z)\cap E=\emptyset$, where $E:=\overline{\psi^{-1}(D)}\subset\mcX_0$ is an effective divisor. For each $x\in\psi^{-1}(z)$ and each $\alpha\in I$, we have $\K^\alpha_x(\mcX_0)\neq\emptyset$. By definition $C_x^\alpha\subset\psi^{-1}(z)$ (hence $C_x^\alpha\cap E=\emptyset$) for all $[C_x^\alpha]\in\K^\alpha_x(\mcX_0)$. It follows that $(E\cdot\K^\alpha(\mcX_0))=0$ for all $\alpha\in I$, which implies that $E=0$. It contradicts the choice of $E$. Then the conclusion follows.
\end{proof}

\subsection{Properties of symbol algebras} \label{subsection. Deformation rigidity via parabolic geometry}

\begin{defi}
Given a distribution $\mcV$ on a complex manifold $Y$, define the weak derived system $\mcV^{-k}$ inductively by
\begin{eqnarray*}
&&\mcV^0:=0,\\
&& \mcV^{-1}:=\mcV, \\
&& \mcV^{-k-1}:=\mcV^{-k}+[\mcV^{-1}, \mcV^{-k}], \quad k\geq 1.
\end{eqnarray*}
Denote by $\mcV^{-\infty}:=\lim\limits_{k\rightarrow\infty}\mcV^{-k}$. There exists a positive integer $d$ such that $\mcV^{-d+i}=\mcV^{-d}$ for all $i\geq 0$. In particular, $\mcV^{-\infty}=\mcV^{-d}$ and it is integrable on $Y$. In an open neighborhood of a general point $y\in Y$ these $\mcV^{-k}$'s are subbundles of $TY$. We define the symbol algebra of $\mcV$ at $y$ as the graded nilpotent Lie algebra $\Symb_y(\mcV):=\bigoplus\limits_{1\leq k\leq d}\mcV^{-k}_y/\mcV^{-k+1}_y$. We say $\mcV$ is bracket-generating if $\mcV^{-\infty}=TY$. When $\mcV$ is bracket-generating, $\dim\Symb_y(\mcV)=\dim T_yY=\dim Y$.
\end{defi}

\begin{nota}\label{nota. distribution D(A) and symbol algebra}
Take a subset $A\subset I$. The distribution
$$\mfg_{-1}^A(\bS):=G\times^{P_I}(\sum\limits_{\alpha\in A}(\mfg_{-1}(\alpha)+\mfp_I)/\mfp_I)$$
 on $\bS$ can be extended to be a meromorphic distribution $\mcD^A$ on $\mcX$, which is well-defined on general points of $\mcX_0$ and all points of $\bigcup\limits_{t\neq 0}\mcX_t$. Take a general point $x\in\mcX_0$. Denote by $\mfm_x(A)$ the symbol algebra of $\mcD^{A}_x$, i.e. $\mfm_x(A):=\Symb_x(\mcD^A)$. We say $\mfm_x(A)$ is standard if it is isomorphic to the symbol algebra of the distribution $\mfg_{-1}^A(\bS)$ on $\bS$. Otherwise, we say $\mfm_x(A)$ is degenerate. When $A=I$, we omit the superscript $I$ and write $\mcD:=\mcD^I$ briefly.
\end{nota}

\begin{prop}\label{prop. dim mx(I) = dim X0}\label{prop. rationally chain connected distribution version}
The unique integrable meromorphic distribution on $\mcX_0$ containing $\mcD$ is the tangent bundle. Consequently,  Then the distribution $\mcD$ is bracket-generating on $\mcX_0$ and $\dim\mfm_x(I)=\dim\mcX_0$ for $x\in\mcX_0$ general.
\end{prop}

\begin{proof}
Let $\mcV$ be an integrable meromorphic distribution on $\mcX_0$ containing $\mcD$, and $M$ be a general leaf. Take $\alpha\in I$ and $[C]\in\K^\alpha(\mcX_0)$ with $C\cap M\neq\emptyset$. Then at a point $x\in C\cap M$ we have $T_xC\subset\mcD_x\subset\mcV_x$. Thus $C$ is contained in the analytic closure of $M$. By Proposition \ref{prop. minimal rational chain connected}, the leaf closure of $M$ is $\mcX_0$, completing the proof.
\end{proof}

To continue, we need to recall some concepts and results related with Cartan connections.

\begin{defi}\label{d.fg}
Fix a positive integer $\nu$. Let $\mfl_{-} = \mfl_{-1} \oplus \cdots \oplus \mfl_{-\nu}$ be a graded nilpotent Lie algebra. Denote by ${\rm gr}\Aut(\mfl_{-})$ the group of  Lie algebra automorphisms of $\mfl_{-}$ preserving the gradation and by ${\rm gr}\aut(\mfl_{-})$ its Lie algebra. Fix a connected algebraic subgroup $L_0 \subset {\rm gr}\Aut(\mfl_{-})$ and its Lie algebra $\mfl_0 \subset {\rm gr}\aut(\mfl_{-})$. For each positive integer $i,$ the $i$-th  prolongation of $\mfl_0$ is inductively defined as
\begin{align*}
\begin{aligned}
\mfl_i &:= \Big\{ \phi \in \Hom(\mfl_{-}, \bigoplus\limits_{-\nu \leq j < i} \mfl_j)_i := \bigoplus\limits_{k=1}^{\nu} \Hom(\mfl_{-k}, \mfl_{-k + i}), \\
 & \qquad \phi([v_1, v_2]_{\mfl_{-}}) = \phi(v_1)(v_2) - \phi(v_2)(v_1), \quad \text{for any} \quad v_1, v_2 \in \mfl_{-}   \Big\}.
\end{aligned}
\end{align*}
Here $[\quad, \quad]_{\mfl_{-}}$ denotes the Lie bracket on $\mfl_{-}$ and,  if $\phi(v_{1})\in \mfl_{-}$, then $$\phi(v_{1})(v_{2}):=[\phi(v_{1}), v_{2}]_{\mfl_{-}}.$$ For convenience, we put $\mfl_{-\nu - j} =0$ for every positive integer $j$ and write $$\mfl_{-} = \bigoplus_{k \in \mbN} \mfl_{-k}.$$ The graded vector space $$\mfl:= \bigoplus_{k \in \mbZ} \mfl_k$$ is a graded Lie algebra and called the universal prolongation of $(\mfl_0, \mfl_{-})$.
\end{defi}

The following result on prolongations is due to K. Yamaguchi \cite{Yam93}.

\begin{prop}\cite[Theorem 5.2]{Yam93}\label{prop. prolongation due to Yamaguchi}
Suppose in Setting \ref{setup. G semisimple} that $G$ is simple.

$(i)$ Suppose $G/P_I$ is not biholomorphic to a projective space. Then $\mfg$ is the universal prolongation of $(\mfg_-(I), \mfg_0(I))$.

$(ii)$ Suppose $|I|\geq 2$ and $(G, I)$ is not one of the following:
\begin{eqnarray}
&&  (A_m, \{\alpha_1, \alpha_i\}), \quad 2\leq i\leq m;  \label{eqn. prolongation exception (Am a1ai)}\\
&& (A_m, \{\alpha_i, \alpha_m\}), \quad 1\leq i\leq m-1; \label{eqn. prolongation exception (Am aiam)}
\end{eqnarray}
Then $\mfg_0(I)$ is isomorphic to $\aut(\mfg_-(I))$, the Lie algebra of ${\rm gr}\Aut(\mfg_{-})$.
\end{prop}

\begin{defi}\label{d.Cartan}
Let $L$ be a connected algebraic group and $L^0 \subseteq L$ be a connected algebraic subgroup. Let $\mfl^0 \subset \mfl$ be their Lie algebras. A Cartan connection of type $(L, L^0)$ on a complex manifold $M$ with $\dim M =\dim L/L^0$ is a principal $L^0$-bundle $E \rightarrow M$ with a $\mfl$-valued $1$-form $\Upsilon$ on $E$ with the following properties.
\begin{itemize}
\item[(i)] For $A \in \mfl^0$, denote by $\zeta_A$ the fundamental vector field on $E$ induced by the right $L^0$-action on $E$. Then $\Upsilon(\zeta_{A})=A$ for each $A \in \mfl^0$.
\item[(ii)] For $a \in L^0$, denote by $R_a: E \to E$ the right action of $a$. Then $R_{a}^{*}\Upsilon={\rm Ad}(a^{-1}) \circ \Upsilon$ for each $a\in L^0$.
\item[(iii)] The linear map $\Upsilon_{y}: T_{y} E \rightarrow \mfl$ is an isomorphism for each $y \in E$. \end{itemize}
The Cartan connection $(E\rightarrow M, \Upsilon)$ is flat if the curvature $\kappa:=d\Upsilon+\frac{1}{2}[\Upsilon, \Upsilon]$ vanishes.
\end{defi}

\begin{e.g.}
Let $L$ and $L^0$ be as in Definition \ref{d.Cartan}, and denote by $\omega^{MC}$ the Maurer-Cartan form on $L$. Then $(L\rightarrow L/L^0, \omega^{MC})$ is a flat Cartan connection of type $(L,L^0)$.
\end{e.g.}

\begin{defi}
Let $\mfl_{-}= \oplus_{k\in \mbN} \mfl_{-k}$ be a graded nilpotent Lie algebra with $\mfl_{-j} =0$ for all $j$ larger than for a fixed positive integer $\nu$.
A filtration of type $\mfl_{-}$ on a complex manifold $M$ is a filtration $(F^j M, j \in \mbZ)$ on $M$ such that
\begin{itemize}
\item[(i)] $F^k M =0$ for all $k \geq 0$; \item[(ii)] $F^{-k} M = TM$ for all $k \geq \nu$; and
\item[(iii)] for any $x\in M$, the symbol algebra $${\rm gr}_{x}(M)
:= \bigoplus_{i \in \mbN} F^{-i}_{x}M/F^{-i+1}_{x}M$$  is isomorphic to $\mfl_{-}$ as graded Lie algebras. \end{itemize}
The graded frame bundle of the manifold $M$ with a filtration of type $\mfl_{-}$  is the  ${\rm gr}\Aut(\mfl_{-})$-principal bundle ${\rm grFr}(M)$ on $M$ whose fiber at $x$ is the set of graded Lie algebra isomorphisms from
$\mfl_{-}$ to ${\rm gr}_x(M)$. Let $L_0 \subset {\rm gr}\Aut(\mfl_{-})$ be a connected algebraic subgroup. An $L_0$-structure (subordinate to the filtration) on $M$ means an $L_0$-principal subbundle  $E \subset {\rm grFr}(M)$.
\end{defi}

\begin{rmk}\label{rmk. construction Cartan connections due to Cap and Schichl}
Now let us summarize the work of  A. \v{C}ap and H. Schichl \cite{CS00} on the construction of Cartan connections of type $(G, P_I)$. For more detail of our summarization, see Sections 3.20-3.23 in \cite{CS00}.  Let $G/P_I$ be as in Setting \ref{setup. G semisimple} and suppose that $\mfg$ is the universal prolongation of $(\mfg_-(I), \mfg_0(I))$. Suppose there is a differential system $\mcV$ and a principal bundle $E$ on a complex manifold $M$ such that the weak derivatives of $\mcV$ induces a filtration of type $\mfg_-(I)$ and $E\subset {\rm grFr}(M)$ is an $G_0$-structure on $M$. Then we can construct a Cartan connection of type of $(G, P_I)$ on $M$. The construction is canonical in the sense that it works well for a family, which will be explained in the proof of Proposition \ref{prop. Cartan connection to Fano deformation rigidity}, and that the Cartan connection we construct on $G/P_I$ itself is $(G\rightarrow G/P_I, \omega^{MC})$.
\end{rmk}

Now we state a setting that is slightly more general than Setting \ref{setup. intro Fano deformation}.

\begin{set}\label{setup. singular Y0}
Suppose in Setting \ref{setup. G semisimple} that $G$ is simple and $G/P_I$ is not biholomorphic to a projective space.
Let $\psi: \mcY\rightarrow\Delta\ni0$ be a holomorphic map from an irreducible analytic variety $\mcY$ to $\Delta$ such that $\mcY_t\cong G/P_I$ for $t\neq 0$ and $\mcY_0$ is an irreducible reduced projective variety.
\end{set}

\begin{prop}\label{prop. Cartan connection to Fano deformation rigidity}
Suppose in Setting \ref{setup. singular Y0} that there exists a proper closed algebraic subset $Z\subset\mcY_0$ and a holomorphic fiber bundle $\mcE\rightarrow \mcY\setminus Z$ such that $\mfm_x(I)\cong\mfg_-(I)$ for all $x\in\mcY_0\setminus Z$ and $\mcE_t\rightarrow\mcY_t\setminus Z$ is an $G_0$-structure for all $t\in\Delta$. Then $\mcY_0\cong G/P_I$.
\end{prop}

\begin{proof}

By Proposition \ref{prop. prolongation due to Yamaguchi} the Lie algebra $\mfg$ is the universal prolongation of $(\mfg_-, \mfg_0)$. By Sections 3.20 -- 3.23 in \cite{CS00} we can construct a Cartan connection of type $(G, P_I)$ in the neighborhood of a general point $x\in\mcY_0$. Furthermore, the construction works well for the family $\mcY$ over $\Delta$. In other words, there exists an analytic open subset $\mcY^o$ of $\mcY$, a principal $P_I$-bundle $\Psi: \mcP\rightarrow\mcY^o$, and a holomorphic 1-form $\omega: T\mcP\rightarrow\mfg$ such that
\begin{item}
\item (1) $\mcY^o\supset\mcY_t$ for all $t\neq 0$;
\item (2) $\mcY^o_0:=\mcY^o\cap\mcY_0$ is an analytic open neighborhood of the general point $x\in\mcY_0$;
\item (3) for each $t\in\Delta$ (including $t=0$), $(\Psi_t, \omega_t)$ is a Cartan connection of type $(G, P_I)$;
\item (4) for each $t\neq 0$, the Cartan connection $(\Psi_t, \omega_t)$ is flat.
\end{item}

By the continuity on $t\in\Delta$ of the curvature $\kappa_t:=d\omega_t+\frac{1}{2}[\omega_t, \omega_t]$, the Cartan connection $(\Psi_0, \omega_0)$ is also  flat. By \cite[Corollary 5.4]{Yam93} the Lie algebra of infinitesimal automorphisms of $\mcY_0$, which preserves the symbol algebras on $\mcY^o_0$ and the $G_0$-structure, is isomorphic to $\mfg$.

By upper semi-continuity of $\dim H^0(\mcY_t, T\mcY_t)$, $\dim \aut(\mcY_0)\geq\dim\mfg$, where $\aut(\mcY_0)$ is the Lie algebra of automorphism group of $\mcY_0$. Hence $\aut(\mcY_0)\cong\mfg$ and $G$ acts on $\mcY_0$ with isotropy subgroup at a general point $x\in \mcY_0$ being conjugate to $P_I$. It follows that $\mcY_0\cong G/P_I$.
\end{proof}

\begin{prop}\label{prop. symbol algebra standard iff variety standard}
In Setting \ref{setup. singular Y0} suppose $|I|\geq 2$ and $(G, I)$ is neither \eqref{eqn. prolongation exception (Am a1ai)} nor \eqref{eqn. prolongation exception (Am aiam)} listed in Proposition \ref{prop. prolongation due to Yamaguchi}. Then the followings are equivalent:

$(i)$ $\mcY_0\cong G/P_I$;

$(ii)$ $\mfm_x(I)$ is standard at general points $x\in\mcY_0$.
\end{prop}

\begin{proof}
It is straight-forward to see $(i)\Rightarrow (ii)$. Now let us prove $(ii)\Rightarrow (i)$. Let $\mcY^o$ be the open subset of $\mcY$ where the symbol algebras of $\mcD$ are isomorphic to $\mfg_-(I)$. In particular, $\mcY_t\subset\mcY^o$ for all $t\neq 0$ and $\mcY^o_0$ is a dense open subset of $\mcY_0$. Denote by $\mcF$ a connected component of the graded frame bundle of the family $\mcY^o$ over $\Delta$.

By Proposition \ref{prop. prolongation due to Yamaguchi} the group $G_0\cong {\rm gr}\Aut^o(\mfg_{-}(I))$. Thus the $G_0$-structure $\mcF_t$ on $\mcY_t$ with $t\neq 0$ is holomorphically extended to be the $G_0$-structure $\mcF_0$ on $\mcY_0^0$. The conclusion follows from Proposition \ref{prop. Cartan connection to Fano deformation rigidity}.
\end{proof}

The key point to obtain $\mcY_0\cong G/P_I$ in Setting \ref{setup. singular Y0} is invariance of symbol algebras. Once this is done, it is not hard to extend the $G_0$-structure $E\subset\text{grFr}(G/P_I)$ holomorphically to general points on $\mcY_0$, even in case \eqref{eqn. prolongation exception (Am a1ai)} or \eqref{eqn. prolongation exception (Am aiam)} listed in Proposition \ref{prop. prolongation due to Yamaguchi}. For instance we have the following result.

\begin{prop}\label{prop. (Am, a1a2) case symbol algebra standard iff X0 standard}
Suppose in Setting \ref{setup. singular Y0} that $\bS\cong A_m/P_I$ and $\mfm_x(\alpha_1, \alpha_2)\cong\mfg_-(I)$, where $m\geq 2$, $I=\{\alpha_1, \alpha_2\}$, and $x\in\mcY_0$ is general. Then $\mcY_0\cong A_m/P_I$.
\end{prop}

\begin{proof}
The distributions $\mcD^{\alpha_1}$ and $\mcD^{\alpha_2}$ are integrable on $\mcY_0$. Thus the isomorphism $\mfm_x(\alpha_1, \alpha_2)\cong\mfg_-(I)$ implies that $F: \mcD^{\alpha_1}\otimes\mcD^{\alpha_2}\rightarrow T^\pi/\mcD$ is surjective on the general point $x\in\mcY_0$, where $F$ is the restriction of the Frobenius bracket of $\mcD=\mcD^{\alpha_1}+\mcD^{\alpha_2}\subset T^\pi$.

Denote by $Z\subset\mcY$ the set of points $z$ such that $\mfm_x(I)\ncong\mfg_-(I)$. Then $Z$ is a proper closed algebraic subset of $\mcY_0$. Take any $y\in\mcY\setminus Z$ and define $\mcE_y$ to be the set of grading preserving isomorphisms $\varphi: \mfm_x(I)\rightarrow\mfg_-(I)$ such that $\varphi(\mcD^{\alpha_i}_y)=\mfg_{-1}(\alpha_i)$ for $i=1, 2$. Then $\mcE$ is an $G_0$-structure on the family $\mcY\setminus Z$ over $\Delta$, and the conclusion follows from Proposition \ref{prop. Cartan connection to Fano deformation rigidity}.
\end{proof}


\subsection{Reduction to homogeneous submanifolds} \label{subsection. reduction to homogeneous submanifolds}

The following is straight-forward.

\begin{lem-defi}\label{lem-defi. J-connected pair}
Take $\alpha\neq\beta\in I$ in Setting \ref{setup. G semisimple}. Then the followings are equivalent:

$(i)$ the manifold $\bS^{\alpha, \beta}\cong\bS^\alpha\times\bS^\beta$;

$(ii)$ the roots $\alpha$ and $\beta$ lie in different connected components of the Dynkin diagram of $G_{J\cup\{\alpha, \beta\}}$.

If $(i)$ and $(ii)$ do not hold, we say $(\alpha, \beta)$ is a $J$-connected pair.
\end{lem-defi}

Our main aim in this subsection is to show that

\begin{thm}\label{thm. reduction theorem}
In Setting \ref{setup. intro Fano deformation} suppose $|I|\geq 3$ and $F^{\alpha, \beta}_x\cong \bS^{\alpha, \beta}$ for any $J$-connected pair $\alpha\neq\beta\in I$ and general $x\in\mcX_0$. Then the manifold $\mcX_0\cong \bS$.
\end{thm}

As a direct consequence of Theorem \ref{thm. reduction theorem}, we have the following result.

\begin{cor}\label{cor. reduction results refined version}
In Setting \ref{setup. G semisimple} suppose $|I|\geq 2$ and that for any $\alpha\neq\beta\in I$, there exists a subset $A\subset I$ such that $\alpha,\beta\in A$ and the rational homogeneous space $\bS^A$ is rigid under Fano deformation. Then $G/P_I$ is rigid under Fano deformation.
\end{cor}

\begin{proof}
By Proposition \ref{prop. invariance of product structure under Fano deformation} in the following, we can assume the group $G$ is simple. Then we can discuss in Setting \ref{setup. intro Fano deformation}. Given any subset $A\subset I$, a general fiber of $\pi^A_0: \mcX_0\rightarrow\mcX^A_0$ is a Fano deformation of $\bS^A$. Then the conclusion follows from Theorem \ref{thm. reduction theorem}.
\end{proof}

\begin{prop}\cite[Theorem 1]{Li18}\label{prop. invariance of product structure under Fano deformation}
Let $\phi: \mcZ\rightarrow\Delta\ni 0$ be a holomorphic map with all fiber being connected Fano manifolds. Suppose that $\mcZ_0\cong\mcZ'_0\times\mcZ''_0$. Then there are holomorphic maps $\phi': \mcW'\rightarrow\Delta$ and $\phi'': \mcW''\rightarrow\Delta$ such that all fiber of $\phi'$ and $\phi''$ are connected Fano manifolds, $\mcW'_0\cong\mcZ'_0$, $\mcW''_0\cong\mcZ''_0$, and $\mcZ=\mcW'\times_{\Delta}\mcW''$.
\end{prop}

Now we turn to the proof of Theorem \ref{thm. reduction theorem}. By Proposition \ref{prop. symbol algebra standard iff variety standard}, it suffices to show that the symbol algebra $\mfm_x(I)$ is standard for $x\in\mcX_0$ general. To verify it, we will apply Proposition \ref{prop. characterizing g-}.

\begin{lem}\label{lem. TxFa and Da direct sum}
In Setting \ref{setup. intro Fano deformation} the followings hold at general points $x\in\mcX_0$:
\begin{eqnarray}
&& \sum\limits_{\alpha\in I}T_x F^\alpha_x=\bigoplus\limits_{\alpha\in I}T_x F^\alpha_x\subset T_x\mcX_0, \label{eqn. TxF(a) direct sum} \\
&& \mcD_x=\sum\limits_{\alpha\in I}\mcD^\alpha_x=\bigoplus\limits_{\alpha\in I}\mcD^\alpha_x\subset T_x\mcX_0, \label{eqn. Da direct sum}
\end{eqnarray}
where the distributions $\mcD^\alpha$ and $\mcD$ are as in Notation \ref{nota. distribution D(A) and symbol algebra}.
\end{lem}

\begin{proof}
The relative Mori contractions $\pi^\alpha: \mcX\rightarrow\mcX^\alpha$ and $\pi^{I\setminus\{\alpha\}}: \mcX\rightarrow\mcX^{I\setminus\{\alpha\}}$ induce a morphism
\begin{eqnarray*}
\pi': & \mcX_0\rightarrow\mcX^\alpha_0\times\mcX^{I\setminus\{\alpha\}}_0 \\
& x\mapsto (\pi^\alpha(x), \pi^{I\setminus\{\alpha\}}(x)),
\end{eqnarray*}
which contracts no curves. Then $T_xF^\alpha_x\cap T_xF^{I\setminus\{\alpha\}}_x=\{0\}$ for $\alpha\in I$ and $x\in\mcX_0$, which implies \eqref{eqn. TxF(a) direct sum}. Now \eqref{eqn. Da direct sum} follows from the inclusion $\mcD^\alpha_x\subset T_xF^\alpha_x$ and \eqref{eqn. TxF(a) direct sum}.
\end{proof}

\begin{lem}\label{lem. Cax is Za}
Take $\alpha\in I$ and $x\in\mcX_0$ general in setting of Theorem \ref{thm. reduction theorem}. Then $\mcC^\alpha_x\subset \mbP(\mcD^\alpha_x)$ is projectively equivalent to $\bZ^\alpha\subset \mbP(\mfg_{-1}(\alpha))$.
\end{lem}

\begin{proof}
By Proposition \ref{prop. Fax=Sa for x in X0 general}, the fiber $F^\alpha_x$ at a general point $x\in\mcX_0$ is biholomorphic to $\bS^\alpha$. Thus $\mcC^\alpha_x\cong\bZ^\alpha$.
\end{proof}

\begin{lem}\label{lem. extend gb(Ha) standardly}
In setting of Theorem \ref{thm. reduction theorem}, take $\alpha\in I$ and $\beta\in N_J(\alpha)$ and a general point $x\in\mcX_0$. Then the distribution $\mfg^\beta(\wbZ^\alpha)$ is extended a holomorphic distribution $\mcD^\beta(\wcC^\alpha_x)$ on $\mcC^\alpha_x$, and under the identification $\wbZ^\alpha=\wcC^\alpha_x$ we have $\mfg^\beta(\wbZ^\alpha)=\mcD^\beta(\wcC^\alpha_x)$.
\end{lem}

\begin{proof}
It is a direct consequence of Lemma \ref{lem. Cax is Za} and Lemma \ref{cor. gb(Ha) is Aut(Ha) invariant}.
\end{proof}

Now we are ready to check condition $(ii)$ of Proposition \ref{prop. characterizing g-}, while condition $(i)$ is to be checked later.

\begin{lem}\label{lem. condition (ii) for g- in reduction thm}
In setting of Theorem \ref{thm. reduction theorem}, take $x\in\mcX_0$ general, and any $\alpha\in I$, $\beta\in N_J(\alpha)$, $v\in\wbZa_x\setminus\{0\}$ and $u\in\mfg^\beta_v(\wbZa_x)$, we have
\begin{eqnarray}\label{eqn. condition (ii) for g- in reduction thm}
(\ad v)^{-\langle\beta, \alpha\rangle}(u)=0 \mbox{ in } \mfm_x(I).
\end{eqnarray}
\end{lem}

\begin{proof}
Let $\gamma\in I\setminus\{\alpha\}$ be any root that is $J$-connected with $\alpha$. By assumption of Theorem \ref{thm. reduction theorem},
\begin{eqnarray*}\label{eqn. (adv)...(u)=0 in m(a, r)}
(\ad v)^{-\langle\beta, \alpha\rangle}(u)=0 \mbox{ in } \mfm_x(\alpha, \gamma)
\end{eqnarray*}
Then the inclusion $\mcD^{\alpha, \gamma}\subset\mcD^I$ implies that \eqref{eqn. condition (ii) for g- in reduction thm} holds.
\end{proof}

To check the condition $(i)$ of Proposition \ref{prop. characterizing g-}, we need to write $I$ as a disjoint union $I(j)$ in a special way.

\begin{cons}\label{cons. I(j)}
Fix any element $\bar{\alpha}\in I$ and define $I(1):=\{\bar{\alpha}\}$. Now for each $j\geq 1$ define by induction that
\begin{eqnarray*}
\quad\quad I(j+1):=\{\alpha\in I\setminus\bigcup\limits_{s\leq j}I(s)\mid (\alpha, \beta) \mbox{ is $J$-connected for some } \beta\in I(j)\}.
\end{eqnarray*}
\end{cons}

\begin{lem}\label{lem. 6'}
In setting of Construction \ref{cons. I(j)}, the followings hold.

(1) The set $I$ is the disjoint union of $I(j), j\geq 1$.

(2) Given any $j\geq 2$ with $I(j)\neq\emptyset$ and any $\alpha\in I(j)$, there exists a unique $\beta\in(\bigcup\limits_{s\leq j}I(s))\setminus\{\alpha\}$ such that $(\alpha, \beta)$ is $J$-connected. Moreover, this unique $\beta$ belongs to $I(j-1)$.

(3) Given any $J$-connected pair $(\alpha, \beta)$, there exists a unique $j\geq 1$ such that $\{\alpha, \beta\}\subset I(j)\cup I(j+1)$. Moreover, either $\alpha\in I(j), \beta\in I(j+1)$ or $\beta\in I(j), \alpha\in I(j+1)$.
\end{lem}

\begin{proof}
The assertion (1) holds because the Dynkin diagram $\Gamma_{I\cup J}$ is connected. To prove (2), it suffices to notice that $\Gamma_{I\cup J}$ contains no loop and each element in $\bigcup\limits_{2\leq s\leq j}I(s)$ is connected with the unique element $\bar{\alpha}\in I(1)$ by the elements in $J\cup(\bigcup\limits_{s\leq j}I(s))$. The assertion (3) is a direct consequence of (1) and (2).
\end{proof}

Now we are ready to check the condition $(i)$ of Proposition \ref{prop. characterizing g-} in our situation.

\begin{lem}\label{lem. condition (i) for g- in reduction thm}
In setting of Theorem \ref{thm. reduction theorem}, take a general point $x\in\mcX_0$. We can define a $G_0$-representation on $\mcD^\alpha_x$ and fix some $0\neq v_\alpha\in\wcCa_x$ for each $\alpha\in I$ such that for all $\alpha'\neq\alpha''\in I$ and all $(v', v'')\in G_0\cdot (v_{\alpha'}, v_{\alpha''})\in\wbZ^{\alpha'}\times\wbZ^{\alpha''}$,
\begin{eqnarray}\label{eqn. condition (i) for g- in reduction thm}
(\ad v')^{-\langle\alpha'', \alpha'\rangle + 1}(v'')=0 \mbox{ in } \mfm_x(I).
\end{eqnarray}
\end{lem}

\begin{proof}
Now we will define a $G_0$-representation on $\mcD^\alpha_x$ and fix some $0\neq v_\alpha\in\wcCa_x$ for each $\alpha\in I=\bigcup\limits_{j\geq 1}I(j)$ and show they satisfy \eqref{eqn. condition (i) for g- in reduction thm} by induction on $j\geq 1$.

By our construction, $I(1)=\{\bar{\alpha}\}$ consists of a unique element. Since $|I|>1$, the set $I(2)\neq\emptyset$. Fix any $\bar{\beta}\in I(2)$. By definition $(\bar{\alpha}, \bar{\beta})$ is $J$-connected. By assumption of Theorem \ref{thm. reduction theorem}, $F^{\bar{\alpha}, \bar{\beta}}_x$ with $x\in\mcX_0$ general is biholomorphic to $P_{I\setminus\{\bar{\alpha}, \bar{\beta}\}}/P_I$. This is also biholomorphic to $G_{J\cup\{\bar{\alpha}, \bar{\beta}\}}/P_{\{\bar{\alpha}, \bar{\beta}\}}$, see Notation \ref{nota. Dynkin subdiagram and associated semisimple subgroups}. By Lemma \ref{lem. aut(G PI)}, $G_{J\cup\{\bar{\alpha}, \bar{\beta}\}}\rightarrow\Aut^o(F_x^{\bar{\alpha}, \bar{\beta}})$ is a surjective homomorphism with a finite kernel. Then we obtain the $G_0(J\cup\{\bar{\alpha}, \bar{\beta}\})$ representations on $\mcD^{\bar{\alpha}}_y$ and $\mcD^{\bar{\beta}}_y$ on any point $y\in F^{\bar{\alpha}, \bar{\beta}}_x$, which preserves $\wcC^{\bar{\alpha}}_y$ and $\wcC^{\bar{\beta}}_y$. Here $G_0(J\cup\{\bar{\alpha}, \bar{\beta}\})$ is the Lie subgroup of $G_{J\cup\{\bar{\alpha}, \bar{\beta}\}}$ associated with Lie subalgebra $\mfg_0\subset\text{Lie}(G_{J\cup\{\bar{\alpha}, \bar{\beta}\}})$. The $G_0(J\cup\{\bar{\alpha}, \bar{\beta}\})$ representations induce the required $G_0(=G_0(R))$ representations on $\mcD^{\bar{\alpha}}_y$ and $\mcD^{\bar{\beta}}_y$ respectively. Note that $G_0(J\cup\{\bar{\alpha}, \bar{\beta}\})$ is the quotient group of $G_0:=G_0(R)$ by some torus in the center. This torus acts trivially on $\mcD^{\bar{\alpha}}_y$ and $\mcD^{\bar{\beta}}_y$. Denote by
\begin{eqnarray*}
&& \varphi^{\bar{\alpha}}_{\{\bar{\alpha}, \bar{\beta}\}}: G_0\rightarrow\Aut^o(\wcC^{\bar{\alpha}}_x, \mcD^{\bar{\alpha}}_x)\subset GL(\mcD^{\bar{\alpha}}_x), \\
&& \varphi^{\bar{\beta}}_{\{\bar{\alpha}, \bar{\beta}\}}: G_0\rightarrow\Aut^o(\wcC^{\bar{\beta}}_x, \mcD^{\bar{\beta}}_x)\subset GL(\mcD^{\bar{\beta}}_x).
\end{eqnarray*}

Applying Proposition \ref{prop. characterizing g-} to $F^{\bar{\alpha}, \bar{\beta}}_x\cong P_{I\setminus\{\bar{\alpha}, \bar{\beta}\}}/P_I$, we can conclude that there exists $0\neq v_{\bar{\alpha}}\in\wcC^{\bar{\alpha}}_x$ and $0\neq v_{\bar{\beta}}\in\wcC^{\bar{\beta}}_x$ such that for any $(w_{\bar{\alpha}}, w_{\bar{\beta}})\in G_0\cdot(v_{\bar{\alpha}}, v_{\bar{\beta}})\in\wcC^{\ba}_x\times\wcC^{\bb}_x$,
\begin{eqnarray*}
&& (\ad w_{\ba})^{-\langle\bb, \ba\rangle+1}(w_{\bb})=0 \mbox{ in } \mfm_x(\ba, \bb), \\
&& (\ad w_{\bb})^{-\langle\ba, \bb\rangle+1}(w_{\ba})=0 \mbox{ in } \mfm_x(\ba, \bb).
\end{eqnarray*}
Then the inclusion $\mcD^{\bar{\alpha}, \bar{\beta}}\subset\mcD:=\mcD^I$ implies that
\begin{eqnarray*}
&& (\ad w_{\ba})^{-\langle\bb, \ba\rangle+1}(w_{\bb})=0 \mbox{ in } \mfm_x(I), \\
&& (\ad w_{\bb})^{-\langle\ba, \bb\rangle+1}(w_{\ba})=0 \mbox{ in } \mfm_x(I).
\end{eqnarray*}

In case $I(2)$ consists of the unique element $\bar{\beta}$, we have constructed the $G_0$-representation for both $I(1)$ and $I(2)$.

Now suppose (for the moment) that $|I(2)|\geq 2$. By Lemma \ref{cor. Aut(Ha)}, $G_0\rightarrow\Aut^o(\wcC^{\ba})=\Aut^o(\wcC^{\ba}, \mfg^{\ba}_{-1})$ is surjective.
Take any $\gamma\in I(2)\setminus\{\bar{\beta}\}$. Then as previous argument for $(\bar{\alpha}, \beta)$ we get $G_0$-representations
\begin{eqnarray*}
&& \varphi^{\bar{\alpha}}_{\{\bar{\alpha}, \gamma\}}: G_0\rightarrow\Aut^o(\wcC^{\bar{\alpha}}_x, \mcD^{\bar{\alpha}}_x)\subset GL(\mcD^{\bar{\alpha}}_x), \\
&& \varphi^{\gamma}_{\{\bar{\alpha}, \gamma\}}: G_0\rightarrow\Aut^o(\wcC^{\gamma}_x, \mcD^{\gamma}_x)\subset GL(\mcD^{\gamma}_x).
\end{eqnarray*}
There is an automorphism
\begin{eqnarray*}
\psi(\ba; \bb, \gamma): & \Aut^o(\wcC^{\ba}_x, \mcD^{\ba}_x)\rightarrow\Aut^o(\wcC^{\ba}_x, \mcD^{\ba}_x)
\end{eqnarray*}
such that the following diagram commutes:
\begin{eqnarray*}
\xymatrix{ G_0\ar[rd]_-{\varphi^{\bar{\alpha}}_{\{\bar{\alpha}, \gamma\}}}\ar[r]^-{\varphi^{\bar{\alpha}}_{\{\bar{\alpha}, \bar{\beta}\}}} & \Aut^o(\wcC^{\ba}_x, \mcD^{\ba}_x)\ar[d]^-{\psi(\ba; \bb, \gamma)} \\
 & \Aut^o(\wcC^{\ba}_x, \mcD^{\ba}_x).
}
\end{eqnarray*}
Since $G_0$ is reductive, there is an automorphism $\theta(\ba; \bb, \gamma): G_0\rightarrow G_0$ such that the following diagram commutes
\begin{eqnarray*}
\xymatrix{ G_0\ar[d]_-{\theta(\ba; \bb, \gamma)}\ar[r]^-{\varphi^{\bar{\alpha}}_{\{\bar{\alpha}, \gamma\}}} & \Aut^o(\wcC^{\ba}_x, \mcD^{\ba}_x)\\
 G_0\ar[ru]_-{\varphi^{\bar{\alpha}}_{\{\bar{\alpha}, \bar{\beta}\}}}.&
}
\end{eqnarray*}
In other words, we lift the automorphism $\psi(\ba; \bb, \gamma)$ of $\Aut^o(\wcC^{\ba}_x, \mcD^{\ba}_x)$ to an automorphism $\theta(\ba; \bb, \gamma)$ of $G_0$.

Define $\tau(\gamma):=\varphi^{\gamma}_{\{\bar{\alpha}, \gamma\}}\circ\theta(\ba; \bb, \gamma)^{-1}: G_0\rightarrow G_0\rightarrow\Aut^o(\wcC^{\gamma}_x, \mcD^{\gamma}_x)\subset GL(\mcD^{\gamma}_x)$ and $\tau(\ba):=\varphi^{\ba}_{\{\bar{\alpha}, \gamma\}}\circ\theta(\ba; \bb, \gamma)^{-1}$. In particular, we have $\tau(\ba)=\varphi^{\ba}_{\{\ba, \bb\}}$. Applying Proposition \ref{prop. characterizing g-} to $F^{\bar{\alpha}, \gamma}_x\cong P_{I\setminus\{\bar{\alpha}, \gamma\}}/P_I$, we can conclude that there exists $0\neq v'_{\ba}\in\wcC^{\ba}_x$ and $0\neq v'_\gamma\in\wcC^\gamma_x$ such that for any $(w_{\bar{\alpha}}, w_\gamma)\in G_0\cdot(v'_{\bar{\alpha}}, v'_\gamma)\in\wcC^{\ba}_x\times\wcC^\gamma_x$ (under the representation $\varphi^{\ba}_{\{\ba, \gamma\}}$ and $\varphi^\gamma_{\{\ba, \gamma\}}$),
\begin{eqnarray}
&& (\ad w_{\ba})^{-\langle\gamma, \ba\rangle+1}(w_\gamma)=0 \mbox{ in } \mfm_x(\ba, \gamma), \label{eqn. (ad wa)...(wr)=0 in m(a, r)}\\
&& (\ad w_\gamma)^{-\langle\ba, \gamma\rangle+1}(w_{\ba})=0 \mbox{ in } \mfm_x(\ba, \gamma). \label{eqn. (ad wr)...(wa)=0 in m(a, r)}
\end{eqnarray}

Denote by $R(\ba, \gamma):=\varphi_{\{\ba, \gamma\}}(G_0)\cdot([v'_{\ba}], [v'_\gamma])\subset\mcH^{\ba}_x\times\mcH^\gamma_x$. Then $R(\ba, \gamma)$ is a closed $G_0$-orbit, and the two projections $R(\ba, \gamma)\rightarrow\mcH^{\ba}_x$ and $R(\ba, \gamma)\rightarrow\mcH^\gamma_x$ are surjective. In particular, for the previously chosen element $0\neq v_{\ba}\in\wcC^{\ba}_x$ there exists $0\neq v_\gamma\in\wcC^\gamma_x$ such that $([v_{\ba}], [v_\gamma])\in R(\ba, \gamma)$. Furthermore,
\begin{eqnarray*}
R(\ba, \gamma)=\varphi_{\{\ba, \gamma\}}(G_0)\cdot([v'_{\ba}], [v'_\gamma])=\varphi_{\{\ba, \gamma\}}(G_0)\cdot([v_{\ba}], [v_\gamma])\subset\mcH^{\ba}_x\times\mcH^\gamma_x.
\end{eqnarray*}
Since $\theta:=\theta(\ba; \bb, \gamma)$ is an automorphism of $G_0$, we know that
\begin{eqnarray*}
\tau_{\{\ba, \gamma\}}(G_0)=\varphi_{\{\ba, \gamma\}}(\theta^{-1}(G_0))=\varphi_{\{\ba, \gamma\}}(G_0),
\end{eqnarray*}
where $\tau_{\{\ba, \gamma\}}(G_0):=(\tau(\ba), \tau(\gamma))$. It follows that
\begin{eqnarray*}
\tau_{\{\ba, \gamma\}}(G_0)\cdot([v_{\ba}], [v_\gamma])=\varphi_{\{\ba, \gamma\}}(G_0)\cdot([v_{\ba}], [v_\gamma])=R(\ba, \gamma).
\end{eqnarray*}
Hence for all $(w_{\ba}, w_\gamma)\in\tau_{\{\ba, \gamma\}}(G_0)\cdot(v_{\ba}, v_\gamma)\subset\wcC^{\ba}_x\times\wcC^\gamma_x$ the formulae \eqref{eqn. (ad wa)...(wr)=0 in m(a, r)} and \eqref{eqn. (ad wr)...(wa)=0 in m(a, r)} hold. Then the inclusion $\mcD^{\bar{\alpha}, \gamma}\subset\mcD:=\mcD^I$ implies that for all $(w_{\ba}, w_\gamma)\in\tau_{\{\ba, \gamma\}}(G_0)\cdot(v_{\ba}, v_\gamma)\subset\wcC^{\ba}_x\times\wcC^\gamma_x$,
\begin{eqnarray*}
&& (\ad w_{\ba})^{-\langle\gamma, \ba\rangle+1}(w_\gamma)=0 \mbox{ in } \mfm_x(I), \label{eqn. (ad wa)...(wr)=0 in m(I)}\\
&& (\ad w_\gamma)^{-\langle\ba, \gamma\rangle+1}(w_{\ba})=0 \mbox{ in } \mfm_x(I). \label{eqn. (ad wr)...(wa)=0 in m(I)}
\end{eqnarray*}

\medskip

Now we have obtained $G_0$-representations on $\mcD^\alpha_x$ and chosen $0\neq v_\alpha\in\wcCa$ for all $\alpha\in I(1)\cup I(2)$ such that \eqref{eqn. condition (i) for g- in reduction thm} holds for $J$-connected pair $\alpha', \alpha''\in I(1)\cup I(2)$. Repeat the argument above, we can obtain $\tau_\alpha: G_0\rightarrow\Aut^o(\wcCa_x, \mcD^\alpha_x)\subset GL(\mcD^\alpha_x)$ and choose $0\neq v_\alpha\in\wcCa_x$ for all $\alpha\in I=\bigcup\limits_{j\geq 1}I(j)$ such that \eqref{eqn. condition (i) for g- in reduction thm} holds for all $J$-connected pair $(\alpha', \alpha'')\in I\times I$.

\medskip

Now take any pair $\alpha\neq\beta\in I\times I$ which is not $J$-connected. By Lemma-Definition \ref{lem-defi. J-connected pair}, $F^{\alpha, \beta}_y=F^\alpha_y\times F^\beta_y$ at any $y\in\bigcup\limits_{t\neq 0}\mcX_t$. By Proposition \ref{prop. invariance of product structure under Fano deformation},
\begin{eqnarray}\label{eqn. product structure of fibers}
F^{\alpha, \beta}_x=F^\alpha_x\times F^\beta_x \mbox{ at any } x\in\mcX_0.
\end{eqnarray}
Now for $x\in\mcX_0$ general, $\mcD^\alpha_x$, $\mcD^\beta_x$ and $\mcD_x$ are well-extended. By \eqref{eqn. product structure of fibers} the Levi bracket of vector fields satisfies
\begin{eqnarray*}
[\mcD^\alpha_x, \mcD^\beta_x]\subset\mcD^\alpha_x+\mcD^\beta_x\subset\mcD_x,
\end{eqnarray*}
which implies that for any $(w_\alpha, w_\beta)\in\wcCa_x\times\wcC^\beta_x\subset\mcD^\alpha_x\times\mcD^\beta_x$
\begin{eqnarray*}\label{eqn. (ad wa)(wb)=0 for (a, b) not J-connected}
[w_\alpha, w_\beta]=0 \mbox{ in } \mfm_x(I).
\end{eqnarray*}
In summary, \eqref{eqn. condition (i) for g- in reduction thm} holds for all pairs $(\alpha', \alpha'')\in I\times I$ with $\alpha'\neq\alpha''$.
\end{proof}

Now we are ready to complete the proof of Theorem \ref{thm. reduction theorem}

\begin{proof}[Proof of  Theorem \ref{thm. reduction theorem}]
Take a general point $x\in\mcX_0$. By Lemma \ref{lem. condition (ii) for g- in reduction thm} and Lemma \ref{lem. condition (i) for g- in reduction thm}, the symbol algebra $\mfm_x(I)$ satisfies conditions $(i)$ and $(ii)$ in Proposition \ref{prop. characterizing g-}. Then by Proposition \ref{prop. characterizing g-} the symbol algebra $\mfm_x(I)$ is a quotient algebra of $\mfg_-(I)$. By Proposition \ref{prop. rationally chain connected distribution version}, $\dim \mfm_x(I)=\dim\mfg_-(I)$, which implies $\mfm_x(I)\cong\mfg_-(I)$. Then the conclusion follows from Proposition \ref{prop. symbol algebra standard iff variety standard}.
\end{proof}

\section{Rigidity and degeneration under Fano deformation}\label{section. Rigidity and degeneration under Fano deformation}

\subsection{Proof of Main results}\label{subsection. Rigidity property and its reductions}

Now we will prove Theorems \ref{thm. intro submaximal Picard numbers} and \ref{thm. intro J connected without end nodes} by assuming Propositions \ref{prop. intro A4 D5 submaximal rigidity} and \ref{prop. intro (Am, a1a2am) rigidity}. It is devoted to the proof of Proposition \ref{prop. intro A4 D5 submaximal rigidity} from next subsection until the end of the paper.

\begin{proof}[Proof of Theorem \ref{thm. intro submaximal Picard numbers}]
By assumption we can write the rational homogeneous space to be $\bS:=G/P_{R\setminus\{\beta_0\}}$, where $\beta_0$ is a root in $R$. When $\rho(\bS)\leq 3$, $\bS$ is biholomorphic to $\mbP^2$, $\mbF(1, 2, \mbP^3)$, $\mbF(1, 2, Q^6)$ or $\mbF(0, 2, Q^6)$. It remains to check the Fano deformation rigidity of  $\mbF(0, 2, Q^6)=D_4/P_I$ with $I=\{\alpha_1, \alpha_3, \alpha_4\}$. Take any two different roots $\beta_1, \beta_2\in I$. The manifold $\bS^{\beta_1, \beta_2}$ is biholomorphic to $\mbP(T_{\mbP^3})$, which is rigid under Fano deformation by Theorem \ref{thm. intro (Am, a1am)}. By Corollary \ref{cor. reduction results refined version} $\mbF(0, 2, Q^6)$ is rigid under Fano deformation.

Now we will apply Corollary \ref{cor. reduction results refined version} to $\bS$ with $\rho(G/P_{R\setminus\{\beta_0\}})\geq 4$. Take any $J$-connected pair $(\beta_1, \beta_2)\in I\times I$. By our assumption, one of the followings hold:

$(i)$ the Dynkin diagram $\Gamma_{\beta_0,\beta_1,\beta_2}=\Gamma_{\beta_0}\cup\Gamma_{\beta_1,\beta_2}$ is of type $A_1\times A_2$;

$(ii)$ the Dynkin diagram $\Gamma_{\beta_0,\ldots,\beta_3}$ is of type $A_4$ for some $\beta_3\in I\setminus\{\beta_1, \beta_2\}$;

$(iii)$ the Dynkin diagram $\Gamma_{\beta_0,\ldots,\beta_4}$ is of type $D_5$ for some $\beta_3,\beta_4\in I\setminus\{\beta_1, \beta_2\}$.

By Theorem \ref{thm. intro complete flag mfds} and Proposition \ref{prop. intro A4 D5 submaximal rigidity}, the manifolds $\bS^{\beta_1, \beta_2}$, $\bS^{\beta_1, \beta_2, \beta_3}$ and $\bS^{\beta_1,\ldots,\beta_4}$ corresponding to $(i)$, $(ii)$ and $(iii)$ respectively are rigid under Fano deformation. Then so is $G/P_{R\setminus\{\beta_0\}}$ by Corollary \ref{cor. reduction results refined version}.
\end{proof}

\begin{proof}[Proof of Theorem \ref{thm. intro J connected without end nodes}]
In this situation for any $J$-connected pair $(\alpha, \beta)\in I\times I$, the unique connected component of the Dynkin diagram $\Gamma_{J\cup\{\alpha, \beta\}}$ containing both $\alpha$ and $\beta$ is one of the following types:

$(i)$ $(A_m, \{\alpha_1, \alpha_m\})$ with $m\geq 2$;

$(ii)$ $(A_m, \{\alpha_1, \alpha_2\})$ with $m\geq 3$.

By our assumption, in case $(ii)$ there exists $\gamma\in I\setminus\{\alpha, \beta\}$ such that the unique connected component of the Dynkin diagram $\Gamma_{J\cup\{\alpha, \beta, \gamma\}}$ containing all of $\alpha$, $\beta$ and $\gamma$ is of type $(A_{m+1}, \{\alpha_1, \alpha_2, \alpha_{m+1}\})$ up to symmetry. Then the conclusion follows from Corollary \ref{cor. reduction results refined version}.
\end{proof}

Indeed by a careful analysis of Dynkin diagrams we can apply the same proof to deduce the following rigidity result.

\begin{thm}\label{thm. statement rigidity candidiate of each Ii at least three etc}
Let $G$ be a simple algebraic group of type $ADE$, $I\subset R$ be a subset and $J:=R\setminus I$. Write $I$ as the disjoint union $\cup I_i$, where each $\Gamma_{I_i}$ is a connected component of $\Gamma_I$. Suppose that

$(1)$ the end nodes of Dynkin diagram of $G$ is contained in $I$,

$(2)$ each $I_i$ satisfies that either $I_i\cap \partial R\neq\emptyset$ or its cardinality $|I_i|\geq 3$,

$(3)$ in case $G$ is of type $D$ or $E$, there exists at most one $\beta\in J$ such that $\langle\beta, \bar{\alpha}\rangle\neq 0$, where $\bar{\alpha}$ is the node in Dynkin diagram of $G$ with three branches.

Then the rational homogeneous space $G/P_I$ is rigid under Fano deformation.
\end{thm}

\begin{rmk}
As a direct consequence of Proposition \ref{prop. invariance of product structure under Fano deformation}, we can know that $\bS$ is rigid under Fano deformation if $\bS=\bS_1\times\cdots\times\bS_k$ and each $\bS_i$ is as in the statement of one of Theorems \ref{thm. intro complete flag mfds}, \ref{thm. intro submaximal Picard numbers}, \ref{thm. intro J connected without end nodes} or \ref{thm. statement rigidity candidiate of each Ii at least three etc}.
\end{rmk}

\subsection{Rigidity of $A_m/P_{\{\alpha_1, \alpha_2,\alpha_m\}}$} \label{subsection. rigidity of (Am, a1a2am)}

The aim of this subsection is to show the following rigidity property.

\begin{thm}\label{thm. manifold (Am, a1, a2, am) is Fano rigid}
The flag manifold $A_m/P_{\{\alpha_1, \alpha_2, \alpha_m\}}$ is rigid under Fano deformation.
\end{thm}

In other words, we want to prove $\mcX_0\cong A_m/P_{\{\alpha_1, \alpha_2,\alpha_m\}}$ in Setting \ref{setup. intro Fano deformation} under additional assumption that $\bS=A_m/P_{\{\alpha_1, \alpha_2, \alpha_m\}}$. Firstly, we have the following rigidity result on fibers.

\begin{prop}\label{prop. standard fibers in (Am, al, a2, am) case}
Suppose $\bS=A_m/P_{\{\alpha_1, \alpha_2, \alpha_m\}}$ in Setting \ref{setup. intro Fano deformation}. Then the followings hold for $x\in\mcX_0$ general:
\begin{eqnarray}
&& F^{\alpha_1}_x\cong\mbP^1, \quad F^{\alpha_2}_x\cong\mbP^{m-2}, \quad F^{\alpha_m}_x\cong\mbP^{m-2}; \label{eqn. (Am, al, a2, am) case fiber F(ai)}\\
&& F^{\alpha_1, \alpha_m}_x\cong F^{\alpha_1}_x\times F^{\alpha_m}_x\cong\mbP^1\times\mbP^{m-2}; \label{eqn. (Am, al, a2, am) case fiber F(a1am)}\\
&& F^{\alpha_2, \alpha_m}_x\cong P_{\alpha_1}/P_{\{\alpha_1, \alpha_2, \alpha_m\}}. \label{eqn. (Am, al, a2, am) case fiber F(a2am)}
\end{eqnarray}
\end{prop}

\begin{proof}
The conclusions \eqref{eqn. (Am, al, a2, am) case fiber F(ai)} and \eqref{eqn. (Am, al, a2, am) case fiber F(a2am)} follows from Fano deformation rigidity of projective spaces and $A_k/P_{\{\alpha_1, \alpha_k\}}$ respectively, see Theorem \ref{thm. intro (Am, a1am)}. The conclusion \eqref{eqn. (Am, al, a2, am) case fiber F(a1am)} follows from \eqref{eqn. (Am, al, a2, am) case fiber F(ai)} and Proposition \ref{prop. invariance of product structure under Fano deformation}.
\end{proof}

As a direct consequence of Proposition \ref{prop. standard fibers in (Am, al, a2, am) case}, we have the following result.

\begin{cor}\label{cor. algebra type (Am, a1, a2, am) basic facts}
Suppose $\bS=A_m/P_{\{\alpha_1, \alpha_2, \alpha_m\}}$ in Setting \ref{setup. intro Fano deformation}. Then the followings hold for $x\in\mcX_0$ general.

(1) The symbol algebras $\mfm_x(\alpha_1), \mfm_x(\alpha_2)$ and $\mfm_x(\alpha_m)$ are standard. More precisely, they are abelian algebras of dimension $1, m-2$ and $m-2$ respectively.

(2) The symbol algebras $\mfm_x(\alpha_1, \alpha_m)$ and $\mfm_x(\alpha_2, \alpha_m)$ are standard. More precisely,

\quad (i) there is a decomposition of abelian algebra $\mfm_x(\alpha_1, \alpha_m)=\mfm_x(\alpha_1)\oplus\mfm_x(\alpha_m)$;

\quad (ii) $\dim\mfm_{-2}(\alpha_2, \alpha_m)=1$ and the bilinear map
\begin{eqnarray*}
& \mfm_x(\alpha_2)\times\mfm_x(\alpha_m)\rightarrow(\mfm_x(\alpha_2, \alpha_m))_{-2} \\
& (x, y)\mapsto[x, y]
\end{eqnarray*}
induces an isomorphism of vector spaces $\mfm_x(\alpha_2)\cong\Hom(\mfm_x(\alpha_m), (\mfm_x(\alpha_2, \alpha_m))_{-2})$.
\end{cor}

\begin{prop}\label{prop. fibers F(a1a2) in Am case}
Suppose $\bS\cong A_m/P_{\{\alpha_1, \alpha_2, \alpha_m\}}$ in Setting \ref{setup. intro Fano deformation}. Then $F^{\alpha_1, \alpha_2}\cong P_{I\setminus\{\alpha_1, \alpha_2\}}/P_I$ for $x\in\mcX_0$ general.
\end{prop}

\begin{proof}
Take $x\in\mcX_0$ general. We claim that
\begin{eqnarray}\label{eqn. m(a1, a2) standard in Am case}
\mbox{ the symbol algebra } \mfm_x(\alpha_1, \alpha_2) \mbox{ is standard.}
\end{eqnarray}
For the simplicity of discussion, we omit the subscript $x$ in the notations of symbol algebras such as $\mfm_x(\alpha_1, \alpha_2)$ and $\mfm_x(\alpha_m)$.

Now suppose that $\mfm(\alpha_1, \alpha_2)$ is not standard. Then there exists $0\neq v_2\in\mfm(\alpha_2)$ such that $[\mfm(\alpha_1), v_2]=0$. Since $\mfm(\alpha_2, \alpha_m)$ is standard, there exists $0\neq v_3\in\mfm(\alpha_m)$ such that $v_4:=[v_2, v_3]\neq 0$ and $\mfm_{-2}(\alpha_2, \alpha_m)=\C v_4$. In particular, there is a decomposition of vector spaces
$$\mfm(\alpha_2, \alpha_m)=\mfm(\alpha_2)\oplus\mfm(\alpha_m)\oplus\C v_4.$$

Take $0\neq v_1\in\mfm(\alpha_1)$. Then we have
\begin{eqnarray}\label{eqn. v1 barcket v4 = 0}
[v_1, v_4]=[v_1, [v_2, v_3]]=[[v_1, v_2], v_3]+[v_2, [v_1, v_3]]=0.
\end{eqnarray}
In other words, $[\mfm(\alpha_1), \C v_4]=0$. Let $\mcA(\alpha_1, \alpha_2, \alpha_m)$ be the vector subspace of $\mfm(\alpha_1, \alpha_2, \alpha_m)$ generated by $\mfm(\alpha_1, \alpha_2)$, $\mfm(\alpha_m)$ and $\C v_4$.
Denote by
\begin{eqnarray*}
&& \mfm(1; \alpha_1, \alpha_2):=\mfm(\alpha_1)\oplus\mfm(\alpha_2), \\
&& \mfm(k; \alpha_1, \alpha_2):=[\mfm(1; \alpha_1, \alpha_2), \mfm(k-1; \alpha_1, \alpha_2)] \mbox{ for each } k \geq 2.
\end{eqnarray*}
Thus $\mfm(\alpha_1, \alpha_2)=\sum\limits_{k=1}^{\infty}\mfm(k; \alpha_1, \alpha_2)$.
We claim that (when \eqref{eqn. m(a1, a2) standard in Am case} fails),
\begin{eqnarray}\label{eqn. claim A(a1, a2, am) subalgebra}
\mcA(\alpha_1, \alpha_2, \alpha_m) \mbox{ is a Lie subalgebra of } \mfm(\alpha_1, \alpha_2, \alpha_m).
\end{eqnarray}
Indeed by Corollary \ref{cor. algebra type (Am, a1, a2, am) basic facts} we already know that
\begin{eqnarray*}
\mcA(\alpha_1, \alpha_2, \alpha_m)=\mfm(\alpha_1, \alpha_2) + \mfm(\alpha_1, \alpha_m) + \mfm(\alpha_2, \alpha_m).
\end{eqnarray*}
It follows that
\begin{eqnarray*}
[\mfm(\alpha_m) + \C v_4, \mfm(\alpha_m) + \C v_4]\subset\mfm(\alpha_2, \alpha_m)\subset\mcA(\alpha_1, \alpha_2, \alpha_m).
\end{eqnarray*}
Hence to prove the claim \eqref{eqn. claim A(a1, a2, am) subalgebra} it remains to show that
\begin{eqnarray}\label{eqn. claim m(k, a1, a2) bracket m(am)+ c v4 into A(a1, a2, am)}
[\mfm(k; \alpha_1, \alpha_2), \mfm(\alpha_m) + \C v_4]\subset\mcA(\alpha_1, \alpha_2, \alpha_m) \mbox{ for all } k \geq 1.
\end{eqnarray}

Now let us prove \eqref{eqn. claim m(k, a1, a2) bracket m(am)+ c v4 into A(a1, a2, am)} by induction on $k$. The case $k=1$ of \eqref{eqn. claim m(k, a1, a2) bracket m(am)+ c v4 into A(a1, a2, am)} follows from
\begin{eqnarray}
&& [\mfm(\alpha_1), \mfm(\alpha_m) + \C v_4]=0, \label{eqn. m(a1) bracket m(am)+C v4 =0 in Am case}\\
&& [\mfm(\alpha_2), \mfm(\alpha_m) + \C v_4]\subset\mfm(\alpha_2, \alpha_m)\subset\mcA(\alpha_1, \alpha_2, \alpha_m),
\end{eqnarray}
where in the first equality we apply Corollary \ref{cor. algebra type (Am, a1, a2, am) basic facts} and \eqref{eqn. v1 barcket v4 = 0}.

Now we assume that $k\geq 2$ and
\begin{eqnarray*}\label{eqn. inductive claim m(i, a1, a2) bracket m(am)+ c v4 into A(a1, a2, am)}
[\mfm(i; \alpha_1, \alpha_2), \mfm(\alpha_m) + \C v_4]\subset\mcA(\alpha_1, \alpha_2, \alpha_m) \mbox{ for all } 1\leq i\leq k-1.
\end{eqnarray*}
Then by the definition of $\mfm(k; \alpha_1, \alpha_2)$ we have
\begin{eqnarray}
&& [\mfm(k; \alpha_1, \alpha_2), \mfm(\alpha_m) + \C v_4] \label{eqn. m(k) decomp}\\
&\subset&\sum\limits_{j=1, 2}[[\mfm(\alpha_j), \mfm(k-1; \alpha_1, \alpha_2)], \mfm(\alpha_m) + \C v_4] \nonumber \\
&\subset&\sum\limits_{j=1, 2}\big([[\mfm(\alpha_j), \mfm(\alpha_m) + \C v_4], \mfm(k-1; \alpha_1, \alpha_2)] \nonumber\\
&& +[\mfm(\alpha_j), [\mfm(k-1; \alpha_1, \alpha_2), \mfm(\alpha_m) + \C v_4]]\big). \nonumber
\end{eqnarray}
We analyse term by term. By \eqref{eqn. m(a1) bracket m(am)+C v4 =0 in Am case} we have
\begin{eqnarray}\label{eqn. term 1 in Am case}
[[\mfm(\alpha_1), \mfm(\alpha_m) + \C v_4], \mfm(k-1; \alpha_1, \alpha_2)]=0.
\end{eqnarray}
On one hand, we have
\begin{eqnarray}
&&[\mfm(\alpha_1), [\mfm(k-1; \alpha_1, \alpha_2), \mfm(\alpha_m) + \C v_4]] \label{eqn. term 2 in Am case}\\
&\subset&[\mfm(\alpha_1), \mcA(\alpha_1, \alpha_2, \alpha_m)] \nonumber\\
&=&[\mfm(\alpha_1), \mfm(\alpha_1, \alpha_2)] + [\mfm(\alpha_1), \mfm(\alpha_m)] + [\mfm(\alpha_1), \C v_4] \nonumber\\
&\subset&\mfm(\alpha_1, \alpha_2) \nonumber\\
&\subset&\mcA(\alpha_1, \alpha_2, \alpha_m).  \nonumber
\end{eqnarray}
By Corollary \ref{cor. algebra type (Am, a1, a2, am) basic facts} we have
\begin{eqnarray*}
[\mfm(\alpha_2), \mfm(\alpha_m) + \C v_4]\subset \mfm(\alpha_2)+ \mfm(\alpha_m) + \C v_4,
\end{eqnarray*}
which implies that
\begin{eqnarray}
&& [[\mfm(\alpha_2), \mfm(\alpha_m) + \C v_4], \mfm(k-1; \alpha_1, \alpha_2)] \label{eqn. term 3 in Am case}\\
&\subset&[\mfm(\alpha_2), \mfm(k-1; \alpha_1, \alpha_2)] + [\mfm(\alpha_m) + \C v_4, \mfm(k-1; \alpha_1, \alpha_2)] \nonumber\\
&\subset&\mfm(k; \alpha_1, \alpha_2) + \mcA(\alpha_1, \alpha_2, \alpha_m) \nonumber\\
&=&\mcA(\alpha_1, \alpha_2, \alpha_m). \nonumber
\end{eqnarray}
Meanwhile by induction we have
\begin{eqnarray*}
&& [ \mfm(k-1; \alpha_1, \alpha_2), \mfm(\alpha_m) + \C v_4] \\
&\subset&\mcA(\alpha_1, \alpha_2, \alpha_m) \\
&=&\mfm(\alpha_1, \alpha_2)+ \mfm(\alpha_m) + \C v_4,
\end{eqnarray*}
which implies that
\begin{eqnarray}
&& [\mfm(\alpha_2), [\mfm(k-1; \alpha_1, \alpha_2), \mfm(\alpha_m) + \C v_4]] \label{eqn. term 4 in Am case}\\
&\subset&[\mfm(\alpha_2), \mfm(\alpha_1, \alpha_2)] + [\mfm(\alpha_2), \mfm(\alpha_m) + \C v_4] + [\mfm(\alpha_2), \C v_4] \nonumber\\
&\subset&\mfm(\alpha_1, \alpha_2) + \mfm(\alpha_2, \alpha_m) \nonumber\\
&\subset&\mcA(\alpha_1, \alpha_2, \alpha_m). \nonumber
\end{eqnarray}
By \eqref{eqn. m(k) decomp}--\eqref{eqn. term 4 in Am case} we have $[\mfm(k; \alpha_1, \alpha_2), \mfm(\alpha_m) + \C v_4]\subset\mcA(\alpha_1, \alpha_2, \alpha_m)$. In other words \eqref{eqn. claim m(k, a1, a2) bracket m(am)+ c v4 into A(a1, a2, am)} holds. Then the claim \eqref{eqn. claim A(a1, a2, am) subalgebra} holds.

Now $\mcA(\alpha_1, \alpha_2, \alpha_m)$ is a Lie subalgebra of $\mfm(\alpha_1, \alpha_2, \alpha_m)$ that contains $\mfm(\alpha_1) + \mfm(\alpha_2) + \mfm(\alpha_m)$. Recall that $\mfm(\alpha_1, \alpha_2, \alpha_m)$ is a Lie algebra generated by $\mfm(\alpha_1) + \mfm(\alpha_2) + \mfm(\alpha_m)$. Then we have $\mcA(\alpha_1, \alpha_2, \alpha_m)=\mfm(\alpha_1, \alpha_2, \alpha_m)$. This contradicts the fact that
\begin{eqnarray*}
\dim \mcA(\alpha_1, \alpha_2, \alpha_m) = 3m-4 = \dim \mfm(\alpha_1, \alpha_2, \alpha_m) - 1,
\end{eqnarray*}
where the dimension of $\mfm(\alpha_1, \alpha_2, \alpha_m)$ is obtained by Proposition \ref{prop. dim mx(I) = dim X0}. Hence we conclude that $\mfm_x(\alpha_1, \alpha_2)$ is standard for $x\in\mcX_0$ general, verifying the claim \ref{eqn. m(a1, a2) standard in Am case}. Then the conclusion follows from Proposition \ref{prop. (Am, a1a2) case symbol algebra standard iff X0 standard}.
\end{proof}

Now we are ready to prove Theorem \ref{thm. manifold (Am, a1, a2, am) is Fano rigid}.

\begin{proof}[Proof of Theorem \ref{thm. manifold (Am, a1, a2, am) is Fano rigid}]
Suppose $\bS\cong A_m/P_{\{\alpha_1, \alpha_2, \alpha_m\}}$ in Setting \ref{setup. intro Fano deformation}.
By Proposition \ref{prop. standard fibers in (Am, al, a2, am) case} and Proposition \ref{prop. fibers F(a1a2) in Am case}, $F^{\alpha, \beta}_x\cong P_{I\setminus\{\alpha, \beta\}}/P_I$ for all $\alpha\neq\beta\in I$ and general points $x\in\mcX_0$. Then by Theorem \ref{thm. reduction theorem} $\mcX_0\cong A_m/P_{\{\alpha_1, \alpha_2, \alpha_m\}}$. In other words, the manifold $A_m/P_{\{\alpha_1, \alpha_2, \alpha_m\}}$ is rigid under Fano deformation.
\end{proof}

\subsection{Fano degeneration of $A_3/P_{\{\alpha_1, \alpha_2\}}$} \label{section. unique degeneration of (A3, a1, a2)}

The aim of this section is to prove Theorem \ref{thm. intro unique degeneration of F(1, 2, C4)}, namely the manifold $F^d(1, 2; \C^4)$ in Construction \ref{cons. intro Fd(1, 2; C4)} is the unique Fano degeneration of $A_3/P_{\{\alpha_1, \alpha_2\}}$. Throughout Section \ref{section. unique degeneration of (A3, a1, a2)}, we always discuss under the following assumption.

\begin{assu}\label{assu. A3 P_min degeneration}
Let $\pi: \mcX\rightarrow\Delta\ni 0$ be a holomorphic map such that $\mcX_t\cong A_3/P_{\{\alpha_1, \alpha_2\}}$ for $t\neq 0$, $\mcX_0$ is a connected Fano manifold, and $\mcX_0\ncong A_3/P_{\{\alpha_1, \alpha_2\}}$.
\end{assu}

By definition $F^d(1, 2; \C^4):=\mbP(\mcL_\sigma\oplus\mcL^\omega)$. Then the restriction of the $\mbP^2$-bundle $F^d(1, 2; \C^4)\rightarrow\mbP^3$ gives a biholomorphic map $\mbP(\mcL_\sigma)\cong\mbP^3$. Moreover the hyperplane bundle $\mbP(\mcL^\omega)$ is biholomorphic to the complete flag manifold $C_2/B$.

The outline to show $\mcX_0\cong F^d(1, 2; \C^4)$ is as follows. Firstly, the Mori contraction $\pi^{\alpha_2}_0: \mcX_0\rightarrow\mcX_0^{\alpha_2}$ is a $\mbP^2$-bundle over $\mbP^3$. We know that at a general point $x\in\mcX_0$, the family $\K^{\alpha_1}_x(\mcX_0)$ consists a single element, denoted by $[C_x]$. An irreducible component of the locus $\{x\in\mcX_0\mid\dim\K^{\alpha_1}_x(\mcX_0)\geq 1\}$ gives a meromorphic section $\sigma: \mbP^3\dashrightarrow\mcX_0$. Let $H$ be an effective divisor on $\mcX_0$ which is a general element in a linear system satisfying $(H\cdot\K^{\alpha_1})=0$ and $(H\cdot\K^{\alpha_2})=1$. The restriction of $\pi^{\alpha_2}_0$ on $H$ is a fibration over $\mbP^3$, whose general fiber is a line in $\mbP^2$. Then we show that $\sigma$ is a holomorphic section, $H\rightarrow\mbP^3$ is a $\mbP^1$-bundle and $H\cap\sigma(\mbP^3)=\emptyset$. Finally we show $H\cong C_2/B$ and $\mcX_0\cong F^d(1, 2; \C^4)$.

Now we sketch how to show $H\cong C_2/B$, which is the key point of the argument in this section. Denote by $\K^{\alpha_1}(\mcX_0/\mbP^3)$ the closure in the Chow scheme of $\mbP^3$ of the set of those $\pi^{\alpha_2}_0(C_x)$, where $x\in\mcX_0$ general and $\K^{\alpha_1}_x(\mcX_0)=\{[C_x]\}$. By considering the symbol algebra of $\mcD=\mcD^{\alpha_1}+\mcD^{\alpha_2}$ on $\mcX_0$, we obtain that a meromorphic distribution $\mcE$ of rank two on $\mbP^3$ satisfying that $\K^{\alpha_1}(\mcX_0/\mbP^3)$ is the family of lines on $\mbP^3$ that are tangent to $\mcE$. This gives an antisymmetric form $\omega$ on $\C^4$ -- which is shown to be a symplectic form later -- such that $\mcE$ coincides with the induced contact form $\mcL^\omega$ on $\mbP^3=\mbP(\C^4)$.

This section is organized as follows. In the part \ref{subsubsection. type of symbol algebra in A3 case}, by studying splitting types of various meromorphic vector bundles along a general element in $\K^{\alpha_2}(\mcX_0)$, we obtain the symbol algebra of $\mcD=\mcD^{\alpha_1}+\mcD^{\alpha_2}$ on $\mcX_0$. In the part \ref{subsubsection. the meromorphic section sigma}, we obtain the meromorphic section $\sigma$ by studying splitting types of various meromorphic vector bundles along a general element in $\K^{\alpha_1}(\mcX_0)$. In the part \ref{subsubsection. a subset of family of lines on P3}, we study the property of the family $\K^{\alpha_1}(\mcX_0/\mbP^3)$. In the part \ref{subsubsection. hyperplane bundles of P(V) over P3}, we complete the proof of Theorem \ref{thm. intro unique degeneration of F(1, 2, C4)} by studying the property of divisor $H$ explained above. In the part \ref{subsubsection. properties of Fd(1, 2; C4)}, we summarize some properties of the manifold $F^d(1, 2; \C^4)$, which will be useful in Subsections \ref{subsection. rigidity of (A4, (a2, a3, a4))} and \ref{subsection. deformation of (D4, a2a3a4)}.

\subsubsection{Type of symbol algebra}\label{subsubsection. type of symbol algebra in A3 case}

\begin{conv}
In Section \ref{section. unique degeneration of (A3, a1, a2)}, we denote by $\msD^{\alpha_i}$, $\msD$ and $\msD^{-i}$ the restriction of $\mcD^{\alpha_i}$, $\mcD$ and $\mcD^{-i}$ on $\mcX_0$ respectively, where the latter is defined in Notation \ref{nota. distribution D(A) and symbol algebra}.
\end{conv}

\begin{lem}\label{lem. unique kernel in degenerate A3 case}
Under Assumption \ref{assu. A3 P_min degeneration}, there exists a unique meromorphic line bundle $\mcN\subset T^{\pi^{\alpha_2}}\mcX_0$ such that $[\mcN, \msD]\subset \msD$, where $\msD:= T^{\pi^{\alpha_1}_0}+T^{\pi^{\alpha_2}_0}=\mcD|_{\mcX_0}$. Moreover $\rank\msD^{-2}=4$ and $\msD^{-3}=T\mcX_0$.
\end{lem}

\begin{proof}
The restriction of the Frobenius bracket of $\msD$ induces a homomorphism $F: \msD^{\alpha_1}\otimes\msD^{\alpha_2}\rightarrow T\mcX_0/\msD$. The image of $F$ is $\msD^{-2}/\msD$ on $\mcX_0$, whose rank is at most two. By Proposition \ref{prop. rationally chain connected distribution version}, $\rank(\msD^{-2}/\msD)\geq 1$. If $\rank(\msD^{-2}/\msD)=2$, then $\mfm_x(\alpha_1, \alpha_2)\cong\mfg_-(\alpha_1, \alpha_2)$ for $x\in\mcX_0$. Then by Proposition \ref{prop. (Am, a1a2) case symbol algebra standard iff X0 standard} $\mcX_0\cong A_3/P_{\{\alpha_1, \alpha_2\}}$, contradicting Assumption \ref{assu. A3 P_min degeneration}. Hence $\rank(\msD^{-2}/\msD)=1$. By Proposition \ref{prop. rationally chain connected distribution version} $\rank(\msD^{-3}/\msD^{-2})\geq 1$, implying that $\msD^{-3}=T\mcX_0$.
\end{proof}

\begin{lem}\label{lem. A3 degeneration unique W when generating D(-3)}
Under Assumption \ref{assu. A3 P_min degeneration}, there exists a unique meromorphic vector subbundle $\mcW\subset\msD^{-1}$ of rank two such that $[\mcW, \msD^{-2}]\subset\msD^{-2}$. Furthermore, $\mcN\subset\mcW$.
\end{lem}

\begin{proof}
The conclusion follows from the two facts that $\rank\msD^{-3}=\rank\msD^{-2}+1$ and that $[\mcN, \msD^{-1}]\subset \msD^{-1}$.
\end{proof}

The following result is important to the proof of Theorem \ref{thm. intro unique degeneration of F(1, 2, C4)}.

\begin{prop} \label{prop. W=Da2 (A3 case)}
We have $\mathcal{W} = \msD^{\alpha_2}$.
\end{prop}

In summary of the description of symbol algebras $\Symb (\msD)_x$ studied in Lemma \ref{lem. unique kernel in degenerate A3 case}, Lemma \ref{lem. A3 degeneration unique W when generating D(-3)} and Proposition \ref{prop. W=Da2 (A3 case)}, we have the following result.

\begin{cor}\label{cor. A3 degenerate symb algebra}
The symbol algebra $\mfm_x(\alpha_1, \alpha_2):=\Symb (\msD)_x$ at general point $x \in \mathcal{X}_0$ is isomorphic to $\mfg_-(C_2\times A_1)$, where $\mfg_-(C_2\times A_1)$ is defined in Definition \ref{defi. gk(I) g-(I) and g-(G)}. More precisely, there exists a nonempty Zariski open subset $\Omega$ of $\mcX_0$ such that

(i) there is an isomorphism on $\Omega$:
\begin{eqnarray}\label{eqn. D=Dai direct sum A3 case}
\msD\cong\msD^{\alpha_1}\oplus\msD^{\alpha_2};
\end{eqnarray}

(ii) the Frobenius bracket of $\msD$ induces a surjective homomorphism on $\Omega$:
\begin{eqnarray}\label{eqn. gr(D)(-2)=Da1 tensor (Da2 quot N) A3 case}
\wedge^2 \msD^{-1}\rightarrow\msD^{\alpha_1}\otimes (\msD^{\alpha_2}/\mcN)\cong(\msD^{-2}/\msD^{-1}),
\end{eqnarray}
where $\msD^{-1}:=\msD$ by definition;

(iii) the restriction of the Frobenius bracket of $\msD^{-2}$ induces a surjective homomorphism on $\Omega$:
\begin{eqnarray}\label{eqn. gr(D)(-3)=Da1 tensor gr(D)(-2) A3 case}
\msD^{-1}\otimes (\msD^{-2}/\msD^{-1})\rightarrow\msD^{\alpha_1}\otimes (\msD^{-2}/\msD^{-1})\cong(\msD^{-3}/\msD^{-2}),
\end{eqnarray}

(iv) the derivative $\msD^{-3}$ of is the whole tangent bundle of $\mcX_0$, i.e. $\msD^{-3}=T\mcX_0$.
\end{cor}

\begin{rmk}\label{rmk. symbol algebra C2A1 in A3 case}
(i) The isomorphisms in \eqref{eqn. D=Dai direct sum A3 case} \eqref{eqn. gr(D)(-2)=Da1 tensor (Da2 quot N) A3 case} and \eqref{eqn. gr(D)(-3)=Da1 tensor gr(D)(-2) A3 case} hold on $\Omega$ instead of on the whole holomorphic loci of corresponding meromorphic vector bundles. Meanwhile as meromorphic vector bundles over $\mcX_0$, we have injective homomorphisms
\begin{eqnarray*}
&& \msD^{\alpha_1}\oplus\msD^{\alpha_2}\hookrightarrow\msD, \\
&& \msD^{\alpha_1}\otimes (\msD^{\alpha_2}/\mcN)\hookrightarrow\msD^{-2}/\msD^{-1}, \\
&& \msD^{\alpha_1}\otimes (\msD^{-2}/\msD^{-1})\hookrightarrow\msD^{-3}/\msD^{-2}.
\end{eqnarray*}

(ii) The Lie algebra $\Symb (\msD)_x\cong\mfg_-(C_2\times A_1)$ can be descried explicitly as the following graded Lie algebra $\mfm_-:=\bigoplus\limits_{k\geq 1}\mfm_{-k}$:
\begin{eqnarray*}
&& \mfm_{-1}:=\C v_1\oplus\C v_2 \oplus \C v_3, \\
&& \mfm_{-2}:=\C v_{12}, \\
&& \mfm_{-3}:=\C v_{121}, \\
&& \mfm_{-k}:=0, \quad \mbox{ for all } k\geq 4.
\end{eqnarray*}
where $v_{12}:=[v_1, v_2]$ and $v_{121}:=[v_{12}, v_1]$. In the identification $\Symb (\msD)_x=\mfm_-$, we have
\begin{eqnarray*}
&& \mfm_x(\alpha_1)=\msD^{\alpha_1}_x=\C v_1, \\
&& \mcN_x=\C v_3\subset\msD^{\alpha_2}_x, \\
&& \mfm_x(\alpha_2)=\msD^{\alpha_2}_x=\C v_2 \oplus\C v_3.
\end{eqnarray*}
\end{rmk}

The rest of the part \ref{subsubsection. type of symbol algebra in A3 case} is devoted to the proof of Proposition \ref{prop. W=Da2 (A3 case)}. Firstly, the following conclusion is straight-forward.

\begin{lem}\label{lem. locus Y1 in A3 case}
There exists a closed variety $Y_1 \subset \mathcal{X}_0$ such that $codim_{\mathcal{X}_0}(Y_1) \geq 2$, $\msD^{\alpha_1}$, $\msD^{\alpha_2}$, $\mathcal{N}$, $\mathcal{W}$ and $\msD$ are holomorphic vector bundles over $\mathcal{X}_0 \backslash Y_1$. Moreover, for $[C_1] \in \mathcal{K}^{\alpha_1} (\mathcal{X}_0)$ general and $[C_2] \in \mathcal{K}^{\alpha_2} (\mathcal{X}_0)$ general, $C_1 \cap Y_1 = \emptyset$ and $C_2 \cap Y_1 = \emptyset$.
\end{lem}

To continue, we need a useful result in \cite{BCD07} due to L. Bonavero, C. Casagrande and S. Druel.

\begin{prop} \cite[Proposition 1]{BCD07} \label{prop. BCD07 equivalence class of quasi-unsplit cycles}
Let $Y$ be a normal $\mbQ$-factorial projective variety, and $\mcF$ be a quasi-unsplit covering family of $1$-cycles on $Y$. Denote by $E_{\mcF} \subset Y$ the union of all $\mathcal{F}$-equivalence classes of dimension larger than $m$, where $m$ is the dimension of a general $\mathcal{F}$-equivalence class. Then

(i) $E_{\mcF}$ is a Zariski closed subset of $Y$, and $\dim E_{\mcF} \leq \dim Y -2$;

(ii) there exists a normal variety $Z$ and a surjective morphism $\varphi: Y \backslash E_{\mcF} \rightarrow Z$ such that fibers $\varphi^{-1}(z)$, $z \in Z$ are $\mathcal{F}$-equivalence classes on $Y$.
\end{prop}

\begin{rmk}
(i) In the setting of Proposition \ref{prop. BCD07 equivalence class of quasi-unsplit cycles}, the meaning of $\mcF$ being a quasi-unsplit family is that all irreducible components of the cycles parameterized by $\mcF$ are numerically proportional.

(ii) In the setting of Proposition \ref{prop. BCD07 equivalence class of quasi-unsplit cycles}, two points in $Y$ are defined to be $\mcF$-equivalent if they are connected by a chain of elements in $\mcF$.

(iii) In our situation of Assumption \ref{assu. A3 P_min degeneration}, both $\mathcal{K}^{\alpha_1}(\mathcal{X}_0)$ and $\mathcal{K}^{\alpha_2}(\mathcal{X}_0)$ are unsplit (hence quasi-unsplit) covering family of rational curves on the complex projective manifold $\mathcal{X}_0$. In particular, the conditions in Proposition \ref{prop. BCD07 equivalence class of quasi-unsplit cycles} is satisfied by both families  $\mathcal{K}^{\alpha_1}(\mathcal{X}_0)$ and $\mathcal{K}^{\alpha_2}(\mathcal{X}_0)$.
\end{rmk}

Applying Proposition \ref{prop. BCD07 equivalence class of quasi-unsplit cycles} to $\mcX_0$, we obtain the following result immediately.

\begin{cor} \label{C0_12}
Denote by
\begin{eqnarray*}
\Psi^{\alpha_1}_0: \mathcal{X}_0 \backslash \mathcal{E}_0^{\alpha_1} \rightarrow \mathcal{Z}_0^{\alpha_1},  \quad \Psi^{\alpha_2}_0: \mathcal{X}_0 \backslash \mathcal{E}_0^{\alpha_2} \rightarrow \mathcal{Z}_0^{\alpha_2}.
\end{eqnarray*}
morphisms in Proposition \ref{prop. BCD07 equivalence class of quasi-unsplit cycles} corresponding to $\mathcal{K}^{\alpha_1}(\mathcal{X}_0)$ and $\mathcal{K}^{\alpha_2}(\mathcal{X}_0)$ respectively. We can take $Y_2 \subset \mathcal{X}_0$ to be $Y_1 \cup \text{sing} (\Psi_{0}^{\alpha_1}) \cup \text{sing} (\Psi_{0}^{\alpha_2}) \cup \mathcal{E}_{0}^{\alpha_1} \cup \mathcal{E}_{0}^{\alpha_2}$ where $Y_1$ is as in Lemma \ref{lem. locus Y1 in A3 case}, and $\text{sing} (\Psi_{0}^{\alpha_i}) \subset \mathcal{X}_0 \backslash \mathcal{E}_{0}^{\alpha_i}$ is the singular locus of the morphism $\Psi^{\alpha_i}_0$. Then $\dim Y_2\leq \dim \mcX_0-2=3$.
\end{cor}

\begin{proof}
The existence of $\Psi^{\alpha_1}_0$ and $\Psi^{\alpha_2}_0$ follows from Proposition \ref{prop. BCD07 equivalence class of quasi-unsplit cycles}. The rest follows from the generic smoothness and the equi-dimensionality of $\Psi^{\alpha_i}_0$, $i =1, 2$.
\end{proof}

\begin{prop} \label{P0_13}
Take $[C_2] \in \mathcal{K}^{\alpha_2}(\mathcal{X}_0)$ general. Then $\msD^{\alpha_1} |_{C_2} \cong \mathcal{O}_{\mathbb{P}^1}(-1)$, $\msD^{\alpha_2} |_{C_2} \cong \mathcal{O}_{\mathbb{P}^1}(2) \oplus \mathcal{O}_{\mathbb{P}^1}(1)$.
\end{prop}

\begin{proof}
The curve $C_2$ is a line  in a general fiber $F_x^{\alpha_2}\cong\mbP^2$ of the elementary Mori contraction $\pi_{0}^{\alpha_2}$, where $x$ is a general point in $C_2$. Thus,
$$\msD^{\alpha_2} |_{C_2} = T F_{C_2}^{\alpha_2} |_{C_2} = \mathcal{O}_{\mathbb{P}^1}(2) \oplus \mathcal{O}_{\mathbb{P}^1}(1).$$

Now take a general local section of $\mathcal{K}^{\alpha_2}(\mathcal{X}) \rightarrow \Delta$ passing through $[C_2]\in\K^{\alpha_2}(\mcX_0)\subset\K^{\alpha_2}(\mcX)$. We obtain a holomorphic family $\{ A^{t}\}_{t \in \Delta}$ (by shrinking $\Delta$ if necessary) such that $\mathcal{S} := \bigcup\limits_{t \in \Delta} A^t \subset \mathcal{X}$ is a complex manifold of dimension two, and $A^{0} = C_2 \subset \mathcal{X}_0$, $A^t \subset \mathcal{X}_t$. Moreover, $\mcS\cap Y_2=C_2\cap Y_2=\emptyset$ by Corollary \ref{C0_12}. Thus for any $x \in \mathcal{S}$, there exists a unique $[l_x] \in \mathcal{K}^{\alpha_1}(\mathcal{X})$ such that $x \in l_{x}$. Furthermore, $x$ is a smooth point of $l_x$. Denote by $\mathcal{L} := \bigcup\limits_{x \in \mathcal{S}} T_x l_x$ which is a holomorphic line bundle over $\mathcal{S}$.
By Proposition \ref{prop. gk(b) along Ca splitting types} we know that for any $t \neq 0$,
\begin{eqnarray*}
\mathcal{L} |_{A^t} = T^{\pi_{t}^{\alpha_1}} |_{A^t} \cong \mcO_{\mbP^1}(\langle\alpha_1, \alpha_2\rangle)=\mcO_{\mbP^1} (-1).
\end{eqnarray*}
It follows that $\mathcal{L} |_{C_2} \cong \mathcal{O}_{\mathbb{P}^1}(-1)$. Thus $\msD^{\alpha_1} |_{C_2}\cong \mathcal{L} |_{C_2}\cong\mcO_{\mbP^1}(-1)$.
\end{proof}

%
%
%
%

\begin{prop} \label{prop. spliting along C2 (A3 case)}
Take $[C_2] \in \mathcal{K}^{\alpha_2}(\mathcal{X}_0)$ general. Then $\msD^{\alpha_1}, \msD^{\alpha_2},  \msD,  \msD^{-2},  \msD^{-3}, \mathcal{N}, \mathcal{W}$ are holomorphic in an open neighborhood of $C_2 \subset \mathcal{X}_0$, and
\begin{eqnarray*}
&& \mathcal{N} |_{C_2}  =   \mathcal{O}(1), \\
&& \msD^{\alpha_2} / \mathcal{N} |_{C_2}  =  \mathcal{O}(2),   \\
&& \msD^{-2} / \msD |_{C_2}  =  \mathcal{O}(1),  \\
&& \msD^{-3} / \msD^{-2} |_{C_2}  =  \mathcal{O}, \\
&& \msD |_{C_2} = \msD^{\alpha_1} |_{C_1} \oplus \msD^{\alpha_2} |_{C_1} =  \mathcal{O}(2) \oplus \mathcal{O}(1) \oplus \mathcal{O}(-1),  \\
&& \msD^{-2} |_{C_2}  =  \mathcal{O}(2) \oplus \mathcal{O}(1) \oplus \mathcal{O}^2, \\
&& \msD^{-3} |_{C_2}  =  T \mathcal{X}_0 |_{C_2} = \mathcal{O}(2) \oplus \mathcal{O}(1) \oplus \mathcal{O}^3.
\end{eqnarray*}
\end{prop}

\begin{proof}
By the generality of $[C_2]\in\K^{\alpha_2}(\mcX_0)$, $T \mathcal{X}_0 |_{C_2} = \mathcal{O}(2) \oplus \mathcal{O}(1) \oplus \mathcal{O}^3$.
Then by Proposition \ref{P0_13} and the injectivity of $\msD^{\alpha_1} \oplus \msD^{\alpha_2} \rightarrow \msD \subset T \mathcal{X}_0$ in an open neighborhood of $C_2 \subset \mathcal{X}_0$, either
\begin{eqnarray*}
&& \msD |_{C_2} = \mathcal{O}(2) \oplus \mathcal{O}(1) \oplus \mathcal{O}, \quad \msD / \msD^{\alpha_2} |_{C_2} = \mathcal{O}, \mbox{ or } \\
&& \msD |_{C_2} = \mathcal{O}(2) \oplus \mathcal{O}(1) \oplus \mathcal{O}(-1), \quad \msD / \msD^{\alpha_2} |_{C_2} = \mathcal{O}(-1),
\end{eqnarray*}
By Lemma \ref{lem. unique kernel in degenerate A3 case}, $[\mathcal{N}, \msD] \subset \msD$, $[\msD^{\alpha_1}, \msD^{\alpha_2}] \subset \msD^{\alpha_2} \subset \msD$ and $[\msD, \msD] \nsubseteq \msD$. Then the Frobenius bracket $\Lambda^2 \msD \rightarrow T \mathcal{X}_0 / \msD$ induces an nonzero homomorphism:
\begin{eqnarray}\label{eqn. Frobenius of D in A3 case}
f: & \big( \msD / \msD^{\alpha_2} \big) \otimes \big( \msD^{\alpha_2} / \mathcal{N} \big) \rightarrow T \mathcal{X}_0 / \msD.
\end{eqnarray}
Note that $\deg(\msD / \msD^{\alpha_2}) |_{C_2} \geq -1$, $\deg(\msD^{\alpha_2} / \mathcal{N})|_{C_2} \geq 2$ and that the degree of each factor of $T\mathcal{X}_0 / \msD|_{C_2}$ is at most one. Since $f$ in \eqref{eqn. Frobenius of D in A3 case} is a nonzero morphism,
$\msD / \msD^{\alpha_2} |_{C_2} = \mathcal{O}(-1)$,  $\msD^{\alpha_2} / \mathcal{N} |_{C_2} = \mathcal{O}(2)$ and $\msD^{-2} / \msD |_{C_2} = \mathcal{O}(1)$. It follows that $\mathcal{N} |_{C_2} = \mathcal{O}(1)$, and $\msD |_{C_2} = \mathcal{O}(2) \oplus \mathcal{O}(1) \oplus \mathcal{O}(-1)$.
Since $\msD^{-2} / \msD^{\alpha_2} |_{C_2} \subset T \mathcal{X}_0 / \msD^{\alpha_2} = \mathcal{O}^3$, and
$\deg (\msD^{-2} / \msD^{\alpha_2} |_{C_2}) =  \deg (\msD^{-2} / \msD |_{C_2}) + \deg (\msD / \msD^{\alpha_2} |_{C_2}) = 0$,
we have $\msD^{-2} |_{C_2} = \mathcal{O}(2) \oplus \mathcal{O}(1) \oplus \mathcal{O}^2$, and $\msD^{-3} / \msD^{-2} |_{C_2} = \mathcal{O}$.
\end{proof}

Now we can complete the proof of Proposition \ref{prop. W=Da2 (A3 case)}.

\begin{proof}[Proof of Proposition \ref{prop. W=Da2 (A3 case)}] By definition of $\msD^{-2}$, we have $[\msD^{\alpha_2}, \msD^{-1}] \subset \msD^{-2}$. Then the Frobenius bracket $\Lambda^2 \msD^{-2} \rightarrow T \mathcal{X}_0 / \msD^{-2}$ induces a homomorphism of meromorphic vector bundles over $\mathcal{X}_0$ as follows:
\begin{eqnarray*}
\psi : \msD^{\alpha_2} \otimes (\msD^{-2} / \msD) \rightarrow T \mathcal{X}_0 / \msD^{-2}.
\end{eqnarray*}
Recall that $\msD^{\alpha_2}, \msD^{-2} / \msD, T \mathcal{X}_0 / \msD^{-2}$ are holomorphic in an open neighborhood of  $C_2 \subset \mathcal{X}_0$, where $[C_2] \in \mathcal{K}^{\alpha_2} (\mathcal{X}_0)$ is a general element. By Proposition \ref{P0_13} and Proposition \ref{prop. spliting along C2 (A3 case)}, $\msD^{\alpha_2} |_{C_2} = \mathcal{O}(2) \oplus \mathcal{O}(1)$, $(\msD^{-2} / \msD^{-1}) |_{C_2} = \mathcal{O}(1)$ and $(T \mathcal{X}_0 / \msD^{-2}) |_{C_2} = \mathcal{O}$.
Thus, $\psi|_{C_2} = 0$. By the general choice of $C_2$, $\psi = 0$.
In other words, $[\msD^{\alpha_2}, \msD^{-1}] \subset \msD^{-2}$. By the uniqueness of $\mcW$ in Lemma \ref{lem. A3 degeneration unique W when generating D(-3)}, we have $\mcW=\msD^{\alpha_2}$.
\end{proof}

\subsubsection{The meromorphic section $\sigma$} \label{subsubsection. the meromorphic section sigma}
Let us firstly recall a result of A. Weber and J. A. Wi\'{s}niewski in \cite{WW17}, in which paper they studied Fano deformation rigidity of complete flag manifolds.

\begin{prop}\cite[Corollary 1.4, Corollary 3.3]{WW17}\label{prop. criterion Pk bundle by Weber Wisniewski}
In the setting \ref{setup. intro Fano deformation} let $\alpha$ be an element of $I$ such that $\Phi^\alpha: G/P_I\rightarrow G/P_{I\setminus\{\alpha\}}$ is a $\mbP^k$-bundle for some $k\geq 1$. Suppose either

(i) $H^*(G/P_{I\setminus\{\alpha\}}, \mbQ)$ is generated by $H^2(G/P_{I\setminus\{\alpha\}}, \mbQ)$; or

(ii) $\mcX^\alpha_0$ is smooth.

Then $\pi^\alpha_0: \mcX_0\rightarrow\mcX^\alpha_0$ is also a $\mbP^k$-bundle.
\end{prop}

As a consequence of Proposition \ref{prop. criterion Pk bundle by Weber Wisniewski}, we have the following result.

\begin{prop} \label{prop. X=P(V) in A3 case}
There exists a unique vector bundle of rank $3$ over $\mathbb{P}^3$, denoted by $\mathcal{V}$, such that

$(i)$ $\mathcal{X}_0$ is biholomorphic to $\mathbb{P}(\mathcal{V})$ and $\mathcal{X}_{0}^{\alpha_2}$ is biholomorphic to $\mathbb{P}^3$;

$(ii)$ $\pi_{0}^{\alpha_2} : \mathcal{X}_0 \rightarrow \mathcal{X}_0^{\alpha_2}$ coincides with the projective bundle $\phi: \mathbb{P}(\mathcal{V}) \rightarrow \mathbb{P}^3$;

$(iii)$ the distribution $\msD^{\alpha_2}=T^\phi$, which is holomorphic on $\mcX_0$;

$(iv)$ $\phi(C_1)$ is a line in $\mbP^3$ for each $[C_1]\in\K^{\alpha_1}(\mcX_0)$.

$(v)$ along any line $l$ in $\mbP^3$, $4\leq\deg(\mcV|_l)\leq 6$.
\end{prop}

\begin{proof}
By Proposition \ref{prop. criterion Pk bundle by Weber Wisniewski}, there exists a vector bundle $\mcV$ on $\mbP^3$ satisfying the properties $(i)$ and $(ii)$.  Hence $\msD^{\alpha_2}=T^\phi$, verifying $(iii)$. By Proposition \ref{prop. divisors intersection theory on family X}, $\phi(C_1)$ is a line in $\mbP^3$, verifying $(iv)$. Since $\deg(\mcV\otimes\mcO(k))|_l=\deg(\mcV|_l)+3k$, we obtain the uniqueness of $\mcV$ with property $(v)$.
\end{proof}

\begin{nota}\label{nota. V phi P2t}
In the rest of Section \ref{section. unique degeneration of (A3, a1, a2)}, we fix the vector bundle $\mcV$ as in Proposition \ref{prop. X=P(V) in A3 case}. We use $\phi:\mbP(\mcV)\rightarrow\mbP^3$ to represent $\pi^{\alpha_2}_0: \mcX_0\rightarrow\mcX^{\alpha_2}_0$. For $t\in\mbP^3$ general, we denote by $\mbP^2_t:=\phi^{-1}(t)$.
\end{nota}

Now let us check the splitting types of various meromorphic vector bundles along general elements in $\mathcal{K}^{\alpha_1} (\mathcal{X}_0)$.

\begin{prop} \label{prop. splitting types along C1 in A3 case}
Take $[C_1] \in \mathcal{K}^{\alpha_1} (\mathcal{X}_0)$ general. Then $\msD^{\alpha_1}, \msD^{\alpha_2},  \msD,  \msD^{-2},  \msD^{-3}, \mathcal{N}$ are holomorphic in an open neighborhood of $C_1 \subset \mathcal{X}_0$, and
\begin{eqnarray*}
&& \mathcal{N} |_{C_1}  = \mathcal{O}, \\
&& \msD^{\alpha_1} |_{C_1}  =  \mathcal{O}(2), \\
&& \msD^{\alpha_2} |_{C_1} = \mathcal{O}(-2) \oplus \mathcal{O}, \\
&& \msD^{\alpha_2} / \mathcal{N} |_{C_1}   =  \mathcal{O}(-2),\\
&& \msD^{-2} / \msD |_{C_1}   =  \mathcal{O}, \\
&& \msD^{-3} / \msD^{-2} |_{C_1}  = \mathcal{O}(2), \\
&& \msD |_{C_1}  =  \msD^{\alpha_1} |_{C_1} \oplus \msD^{\alpha_2} |_{C_1} = \mathcal{O}(2) \oplus \mathcal{O}(-2) \oplus \mathcal{O}, \\
&& \msD^{-3}  |_{C_1}   =  T \mathcal{X}_0 |_{C_1} = \mathcal{O}(2) \oplus \mathcal{O}^4.
\end{eqnarray*}
\end{prop}

\begin{proof}
The restriction $\msD^{\alpha_2}|_{C_1}=TC_1=\mcO(2)$. Choose a holomorphic family $[l_t]\in\K^{\alpha_1}(\mcX_t)$, $t\in\Delta$ satisfying $[l_0]=[C_1]\in\K^{\alpha_1}(\mcX_0)$. By Proposition \ref{prop. X=P(V) in A3 case}$(ii)$,
\begin{eqnarray*}
\deg(\msD^{\alpha_2}|_{C_1})=\deg(T^{\pi^{\alpha_2}_0}|_{l_0})=\deg(T^{\pi^{\alpha_2}_t}|_{l_t}) \mbox{ for all } t\in\Delta.
\end{eqnarray*}
By Proposition \ref{prop. gk(b) along Ca splitting types}, we have
\begin{eqnarray*}
\deg(T^{\pi^{\alpha_2}_t}|_{l_t})=\langle\alpha_2, \alpha_1\rangle + \langle\alpha_2+\alpha_3, \alpha_1\rangle=-2 \mbox{ for } t\neq 0.
\end{eqnarray*}
It follows that $\deg(\msD^{\alpha_2}|_{C_1})=-2$. Then can write $\msD^{\alpha_2} |_{C_1} = \mathcal{O}(a_1) \oplus \mathcal{O}(a_2)$, where $a_1 + a_2 = -2$. Since $\msD^{-3}|_{C_1}=T\mcX_0|_{C_1}=\mcO(2)\oplus\mcO^4$ and $\msD^{\alpha_1} |_{C_1} = \mathcal{O}(2)$, we know that $a_1\leq 0, a_2\leq 0$.
Hence
\begin{eqnarray} \label{e3_0}
\text{either} \quad  \msD^{\alpha_2} |_{C_1}=\mathcal{O}(-1)^2, \text{ or } \msD^{\alpha_2} |_{C_1}=\mathcal{O}(-2) \oplus \mathcal{O}.
\end{eqnarray}
It follows that
\begin{eqnarray} \label{e3_1}
\big( \msD^{\alpha_2} / \mathcal{N} \big) |_{C_1} = \mathcal{O}_{\mathbb{P}^1}(a), \mbox{ where } a \geq -2.
\end{eqnarray}
The injectivity of the homomorphism $\msD^{\alpha_1} \otimes \big( \msD^{\alpha_2} / \mathcal{N} \big) \rightarrow \msD^{-2} / \msD \subset T \mathcal{X}_0 / \msD$ in an open neighborhood of $C_1 \subset \mathcal{X}_0$  implies that
\begin{eqnarray} \label{e3_2}
\msD^{-2} / \msD  |_{C_1} = \mathcal{O}_{\mathbb{P}^1}(b), \mbox{ where } b\geq a+2.
\end{eqnarray}
The injectivity of $\msD^{\alpha_1} \otimes \big( \msD^{-2} / \msD \big) \rightarrow T \mathcal{X}_0 / \msD^{-2} $ in an open neighborhood of $C_1 \subset \mathcal{X}_0$ implies that
\begin{eqnarray} \label{e3_3}
\msD^{-3} / \msD^{-2}  |_{C_1} = T \mathcal{X}_0 / \msD^{-2} |_{C_1} = \mathcal{O}_{\mathbb{P}^1}(c), \mbox{ where } c\geq b+2.
\end{eqnarray}
On the other hand, the injectivity of $ \msD^{\alpha_2} \rightarrow \msD / \msD^{\alpha_1} \subset T \mathcal{X}_0 / \msD^{\alpha_1}$ in an open neighborhood of $C_1 \subset \mathcal{X}_0$ implies that
\begin{eqnarray} \label{e3_4}
\deg (\msD / \msD^{\alpha_1} )  |_{C_1} \geq \deg (\msD^{\alpha_2} |_{C_1}) = -2.
\end{eqnarray}
We also have
\begin{eqnarray}\label{e3_5}
&&\deg(T\mcX_0|_{C_1})-\deg(\msD^{\alpha_1})|_{C_1}-\deg(\msD^{-1}/\msD^{\alpha_1})|_{C_1} \\
&=&\deg(\msD^{-2}/\msD)|_{C_1}+\deg(\msD^{-3}/\msD^{-2})|_{C_1}. \nonumber
\end{eqnarray}
By \eqref{e3_1}--\eqref{e3_5}, we have
\begin{eqnarray*}
2 &\geq& -  \deg (\msD / \msD^{\alpha_1} )  |_{C_1}  =  \deg (\msD^{-2} / \msD)  |_{C_1} + deg (\msD^{-3} / \msD^{-2} )  |_{C_1}  \\
&=& b + c \geq 2b+ 2 \geq 2a + 6 \geq 2.
\end{eqnarray*}
Hence $\deg (\msD / \msD^{\alpha_1} )  |_{C_1}  = -2$, $a =-2$, $b=0$ and $c=2$.
By \eqref{e3_0} and the fact
\begin{eqnarray*}
\deg (\msD^{\alpha_2}/\mcN)|_{C_1}=a=-2=\deg(\msD/\msD^{\alpha_1})|_{C_1},
\end{eqnarray*}
we know that $\msD^{\alpha_2}|_{C_1}  = \mathcal{O}(-2) \oplus \mathcal{O} \cong (\msD / \msD^{\alpha_1}) |_{C_1}$.
The rest of the conclusion follows immediately.
\end{proof}

\begin{prop} \label{prop. V =O(2, 2, 0) along C1 in A3 case}
In setting of Proposition \ref{prop. X=P(V) in A3 case}, $\mathcal{V} |_{\pi^{\alpha_2}_0(C_1)} = \mathcal{O}(2)^2 \oplus \mathcal{O}$ for $[C_1] \in \mathcal{K}^{\alpha_1} (\mathcal{X}_0)$ general.
\end{prop}

\begin{proof}
Take $[C_1] \in \mathcal{K}^{\alpha_1} (\mathcal{X}_0)$ general. Denote by $\mathcal{L}(C_1)$ the line subbundle of
$\mathcal{V}$ over the line $\pi^{\alpha_2}_0(C_1)\subset\mbP^3$ such that $C_1=\mbP(\mcL(C_1))\subset\mbP(\mcV)=\mcX_0$. Then the relative tangent bundle
$T^{\pi^{\alpha_2}_0} |_{C_1} = \mathcal{L}(C_1)^* \otimes \big( \mathcal{V} |_{C_1} / \mathcal{L}(C_1) \big)$.
By Proposition \ref{prop. X=P(V) in A3 case}$(iii)$ and Proposition \ref{prop. splitting types along C1 in A3 case},
$T^{\pi^{\alpha_2}_0}|_{C_1} = \msD^{\alpha_2}|_{C_1} = \mathcal{O}(-2) \oplus \mathcal{O}$.
Then $\mathcal{V} |_{\pi^{\alpha_2}_0(C_1)} = \mathcal{O}(k)^2 \oplus \mathcal{O}(k-2)$, where $k := \deg \mathcal{L}(C_1)$. By Proposition \ref{prop. X=P(V) in A3 case}$(v)$, $k=2$ and the conclusion follows.
\end{proof}

\begin{prop}\label{prop. 19 A3 case}
Let  $\mathcal{V}$ be as in Proposition \ref{prop. X=P(V) in A3 case}. Then the following holds.

$(i)$ Over any line $l \subset \mathbb{P}^3$, either $\mathcal{V}|_{l} = \mathcal{O}(2) \oplus \mathcal{O}(1)^2$ or $\mathcal{V}|_{l} = \mathcal{O}(2)^2 \oplus \mathcal{O}$.

$(ii)$ Take any $[C_1] \in \mathcal{K}^{\alpha_1} (\mathcal{X}_0)$. Then $\mathcal{L}(C_1) = \mcO_{\mathbb{P}^1}(2)$, where $\mcL(C_1)$ is the unique line subbundle of $\mathcal{V} |_{\pi^{\alpha_2}_0(C_1)}$ such that $C_1=\mathbb{P}(\mathcal{L}(C_1)) \subset \mathbb{P}(\mathcal{V}) =\mathcal{X}_0$.
\end{prop}

\begin{proof}
Since $\mathcal{X}_0 \cong \mathbb{P}(\mathcal{V})$ by Proposition \ref{prop. X=P(V) in A3 case}, any $[C_1] \in \mathcal{K}^{\alpha_1} (\mathcal{X}_0)$ must be a section over the line $\pi^{\alpha_2}_0(C_1)\subset\mbP^3$ with largest degree. The degree of this section over the line $\pi^{\alpha_2}_0(C_1)$, is independent of the choice of $[C_1] \in \mathcal{K}^{\alpha_1} (\mathcal{X}_0)$. Then the assertion $(ii)$ follows from proposition \ref{prop. V =O(2, 2, 0) along C1 in A3 case}.

Take any line $l \subset \mathbb{P}^3$. Then by Proposition \ref{prop. X=P(V) in A3 case}, $\mathcal{V}|_{l}$ is a deformation of $\mcV|_{\pi^{\alpha_2}_0(C_1)}=\mathcal{O}_{\mathbb{P}^1}(2)^2 \oplus \mathcal{O}_{\mathbb{P}^1}$. Thus we can write $\mathcal{V} |_{l} = \mathcal{O}(a_1) \oplus \mathcal{O}(a_2) \oplus \mathcal{O}(a_3)$, where
\begin{eqnarray} \label{e3_11}
a_1 + a_2 + a_3 = 4, \quad \text{and} \quad a_1 \geq a_2 \geq a_3.
\end{eqnarray}
By the maximality of $\deg\mcL(C_1)$ among sections of $\mcV$ over lines in $\mbP^3$, we have
\begin{eqnarray} \label{e3_13}
a_1 \leq\deg\mcL(C_1)= 2.
\end{eqnarray}
The assertion $(i)$ follows from \eqref{e3_11} and \eqref{e3_13}.
\end{proof}

\begin{cor}\label{cor. 20 A3 case}
Let $\mathcal{V}$ be as in Proposition \ref{prop. X=P(V) in A3 case}. Then there exists a nonempty Zariski open subset $\mathcal{U} \subset \mathbb{P}^3$ and a section of $\pi^{\alpha_2}_0: \mathcal{X}_0 = \mathbb{P}(\mathcal{V})  \rightarrow \mathbb{P}^3$ over $\mathcal{U}$, denoted by $\sigma : \mathcal{U} \rightarrow \mathcal{X}_0$, such that for any $x\in(\pi^{\alpha_2}_0)^{-1}(U)\setminus\sigma(U)$,

$(i)$ $\mathcal{N}$ is holomorphic at $x$;

$(ii)$$\mcN_x=T_xl_x$, where $l_x:=\langle x, \sigma(\pi^{\alpha_2}_0(x))\rangle$ is the line in $(\pi^{\alpha_2}_0)^{-1}(\pi^{\alpha_2}_0(x))\cong\mbP^2$ joining $x$ and $\sigma(\pi^{\alpha_2}_0(x))$;

$(iii)$ the leaf of $\mcN$ at $x$ is the affine line $l_x\setminus\{\sigma(\pi^{\alpha_2}_0(x))\}$.
\end{cor}

\begin{proof}
Take $[C_1]\in\K^{\alpha_1}(\mcX_0)$ general. Then $\pi^{\alpha_2}_0(C_1)$ is a line in $\mbP^3$ and $\mcV|_{\pi^{\alpha_2}_0(C_1)}=\mcO(2)^2\oplus\mcO$. The curve $C_1$ is identified with $\mbP(\mcL(C_1))\subset\mbP(\mcV|_{\pi^{\alpha_2}_0(C_1)})=(\pi^{\alpha_2}_0)^{-1}(\pi^{\alpha_2}_0(C_1))$, where $\mcL(C_1)\cong\mcO_{\mbP^1}(2)\subset\mcV|_{\pi^{\alpha_2}_0(C_1)}$ is as in Proposition \ref{prop. 19 A3 case}$(ii)$. We know $\mcN\subset\msD^{\alpha_2}$, $\mcN|_{C_1}=\mcO_{\mbP^1}$ and
$\mcO(-2)\oplus\mcO=\msD^{\alpha_2}|_{C_1}\cong\mcL(C_1)^*\otimes (\mcV|_{\pi^{\alpha_2}_0(C_1)}/\mcL(C_1))$.
It follows that $\mcN|_{C_1}=\bigcup\limits_{x\in C_1}T_x\mbP(\mcV|^+_{\pi^{\alpha_2}_0(x)})$, where $\mcV|^+_{\pi^{\alpha_2}_0(C_1)}=\mcO(2)^2\subset\mcV|_{\pi^{\alpha_2}_0(C_1)}=\mcO(2)^2\oplus\mcO$ and $T_{C_1}\mbP(\mcV|^+_{\pi^{\alpha_2}_0(C_1)})$ is the relative tangent bundle of $\mbP(\mcV|^+_{\pi^{\alpha_2}_0(C_1)})\rightarrow\pi^{\alpha_2}_0(C_1)$ along $C_1\subset\mbP(\mcV|^+_{\pi^{\alpha_2}_0(C_1)})$. In other words, at any point $x\in C_1$, $\mcN_x=T_x\mbP(\mcO_{\pi^{\alpha_2}_0(C_1)}(2)^2|_{\pi^{\alpha_2}_0(x)})$, where $\mcO_{\pi^{\alpha_2}_0(C_1)}(2)^2|_{\pi^{\alpha_2}_0(x)}\subset\mcV_{\pi^{\alpha_2}_0(x)}$ is the fiber of $\mcO(2)^2\subset\mcV|_{\pi^{\alpha_2}_0(C_1)}$ at the point $\pi^{\alpha_2}_0(x)\in\pi^{\alpha_2}_0(C_1)$.

Note that $\mbP_{\pi^{\alpha_2}_0(C_1)}(\mcO(2)^2)\cong\mbP^1\times \pi^{\alpha_2}_0(C_1)\cong\mbP^1$. It follows that given any $x\in C_1$ and any $y\in\mbP(\mcO_{\pi^{\alpha_2}_0(C_1)}(2)^2|_{\pi^{\alpha_2}_0(x)})$ lying in the regular locus of $\mcN$, there exists $[C_y]\in\K^{\alpha_1}_y(\mcX_0)$ such that $\pi^{\alpha_2}_0(C_y)=\pi^{\alpha_2}_0(C_1)$ and $\mcN_y=T_y\mbP(\mcO_{\pi^{\alpha_2}_0(C_1)}(2)^2|_{\pi^{\alpha_2}_0(x)})$. Hence, the closure of the leaf at $x\in C_1\subset\mcX_0$ is the line $l_x=\mbP(\mcO_{\pi^{\alpha_2}_0(C_1)}(2)^2_{\pi^{\alpha_2}_0(x)})$.

Take $t\in\mbP^3$ general and denote by $\mbP^2_t:=(\pi^{\alpha_2}_0)^{-1}(t)\cong\mbP^2\subset\mcX_0$. Let $A\subset (\mbP^2_t)^*$ be the closure of the family of lines $l_x:=\mbP(\mcO_{\pi^{\alpha_2}_0(C_x)}(2)^2|_{\pi^{\alpha_2}_0(x)})$, where $x$ runs over the set of general points on $\mbP^2_t$ such that $\K^{\alpha_1}_x(\mcX_0)$ consists of a unique element $[C_x]$ and $\mcN$ is holomorphic at $x$. For a general point $x\in\mbP^2_t$, $E_x:=\{[l]\in(\mbP^2_t)^*\}$ is a line in $(\mbP^2_t)^*$ and $E_x\cap A$ consist of a single point, namely $[l_x]$, in $(\mbP^2_t)^*$. Since $E_x$ could be a general line in $(\mbP^2_t)^*$, the intersection number $(E_x\cdot A)=1$. It follows that $A$ is a line in $(\mbP^2_t)^*$ and there exists a unique point $\sigma(t)\in\mbP^2_t$ such that $A=\{[l]\in(\mbP^2_t)^*\mid \sigma(t)\in l\}$.

It turns out that $\mcN$ is well-defined on $\mbP^2_t\setminus\{\sigma(t)\}$, and at any $x\in\mbP^2_t\setminus\{\sigma(t)\}$, the line $\langle x, \sigma(t)\rangle$ is the leaf closure of $\mcN$ at $x$. The conclusion follows.
\end{proof}

\subsubsection{A subset of family of lines on $\mbP^3$}\label{subsubsection. a subset of family of lines on P3}

\begin{nota}
For $t\in\mbP^3$ general, denote by $\K^{\alpha_1}_t(\mcX_0/\mbP^3)$ the Zariski closure of
\begin{eqnarray*}
\{[\pi^{\alpha_2}_0(C_x)]\in\mbG(1, \mbP^3)\mid x\in(\pi^{\alpha_2}_0)^{-1}(t) \mbox{ general }, \K^{\alpha_1}_x(\mcX_0)=\{[C_x]\}\}
\end{eqnarray*}
in $\mbG(1, \mbP^3)$, and set
$$\mcC^{\alpha_1}_t(\mcX_0/\mbP^3):=\bigcup\limits_{[l]\in\K^{\alpha_1}_t(\mcX_0/\mbP^3)}\mbP(T_tl)\subset\mbP(T_t\mbP^3).$$
Here $\mbG(1, \mbP^3)$ is the family of lines in $\mbP^3$.
Denote by
\begin{eqnarray*}
&& \K^{\alpha_1}(\mcX_0/\mbP^3):= \mbox{ Zariski closure of } \bigcup\limits_{t\in\mbP^3\mbox{ general}}\K^{\alpha_1}_t(\mcX_0/\mbP^3) \mbox{ in } \mbG(1, \mbP^3), \\
&& \mcC^{\alpha_1}(\mcX_0/\mbP^3):= \mbox{ Zariski closure of } \bigcup\limits_{t\in\mbP^3\mbox{ general}}\mcC^{\alpha_1}_t(\mcX_0/\mbP^3) \mbox{ in } \mbP(T\mbP^3).
\end{eqnarray*}
Let $U^{\alpha_1}(\mcX_0/\mbP^3)$ be the inverse image of $\K^{\alpha_1}(\mcX_0/\mbP^3)$ under the natural morphism $F(1, 2; \C^4)\rightarrow\mbG(1, \mbP^3)\supset\K^{\alpha_1}(\mcX_0/\mbP^3)$.
\end{nota}

\begin{lem}\label{lem. 23 A3 case}
Take $t\in\mbP^3$ general. Then $\K^{\alpha_1}_t(\mcX_0/\mbP^3)$ is an irreducible rational curve. Take any $[l]\in\K^{\alpha_1}_t(\mcX_0/\mbP^3)$. There exists $[C]\in\K^{\alpha_1}_{\sigma(t)}(\mcX_0)$ such that $\pi^{\alpha_2}_0(C)=l$.
\end{lem}

\begin{proof}
Take $t\in\mbP^3$ general and $x\in\mbP^2_t:=(\pi^{\alpha_2}_0)^{-1}(t)$ general. Then $\K^{\alpha_1}_x$ consists of a single element, written as $[C_x]$. Furthermore, $C_x\cong\mbP^1$ and $\pi^{\alpha_2}_0$ sends $C_x$ biholomorphically onto a line in $\mbP^3$. Since $\mcV|_{\pi^{\alpha_2}_0(C_x)}=\mcO(2)^2\oplus\mcO$ and the line $\langle x, \sigma(t)\rangle$ in $\mbP^2_t$ coincides with the fiber $\mbP(\mcO_{\pi^{\alpha_2}_0(C_x)}(2)^2|_t)$, there exists a unique $[C_{t, x}]\in\K^{\alpha_1}_{\sigma(t)}(\mcX_0)$ such that $\pi^{\alpha_2}_0(C_{t, x})=\pi^{\alpha_2}_0(C_x)$. Take $y\in\mbP^2_t\setminus\langle x, \sigma(x)\rangle$ general. Then the fact
\begin{eqnarray*}
\mbP(\mcO_{\pi^{\alpha_2}_0(C_x)}(2)^2|_t)\cap\mbP(\mcO_{\pi^{\alpha_2}_0(C_y)}(2)^2|_t)=\langle x, \sigma(t)\rangle\cap\langle y, \sigma(t)\rangle=\{\sigma(t)\}
\end{eqnarray*}
implies that $\pi^{\alpha_2}_0(C_x)\neq\pi^{\alpha_2}_0(C_y)$ (and hence $C_{t, x}\neq C_{t, y}$). This induces injective rational maps (hence injective morphisms)
\begin{eqnarray*}
\xi: & \mbP^1\cong\{[l]\in(\mbP^2_t)^*\mid \sigma(t)\in l\}\dashrightarrow\K^{\alpha_1}_{\sigma(t)}(\mcX_0) \nonumber\\
& \langle x, \sigma(t)\rangle \mapsto [C_{t, x}], \label{eqn. map from P2(t) to K(alpha1)(sigma(t))} \\
\eta: & \mbP^1\cong\{[l]\in(\mbP^2_t)^*\mid \sigma(t)\in l\}\dashrightarrow\K^{\alpha_1}_t(\mcX_0/\mbP^3) \nonumber\\
& \langle x, \sigma(t)\rangle \mapsto [\pi^{\alpha_2}_0(C_x)]. \nonumber
\end{eqnarray*}
By definition $\K^{\alpha_1}_t(\mcX_0/\mbP^3)$ is the closure of the image of $\eta$. Then the conclusion follows immediately from these morphisms $\xi$ and $\eta$.
\end{proof}

The following can also be deduced from the proof of Lemma \ref{lem. 23 A3 case}.

\begin{lem}\label{lem. P2t map to Kt(X0, P3) by psi in A3 case}
Take $t \in \mathbb{P}^3$ general. Denote by $\mathbb{P}_{t}^{2} := (\pi^{\alpha_2}_0)^{-1}(t) \subset \mathcal{X}_0$. Define
\begin{eqnarray*}
 \psi: &&\mathbb{P}_{t}^{2} \dashrightarrow \mathcal{K}_{t}^{\alpha_1}(\mcX_0/\mbP^3) \subset \mbG(1, \mbP^3) \\
 && x \longmapsto [\pi^{\alpha_2}_0 (C_x)],
\end{eqnarray*}
where $x \in \mathbb{P}_{t}^{2}$ general and $[C_x]$ is the unique element of $\mathcal{K}_{x}^{\alpha_1} (\mathcal{X}_0)$.  Then $\psi$ coincides with the linear projection of $\mathbb{P}_{t}^{2}$ with center $\sigma(t)$. In other words, for $x, y \in \text{Dom}(\psi)$, $\psi(x) = \psi(y)$ if and only if $\langle x, \sigma(t) \rangle = \langle y, \sigma(t) \rangle$.
\end{lem}

\begin{cons}\label{cons. 25 A3 case}
Take $x\in\mcX_0$ general. Recall two elementary Mori contractions:
\begin{eqnarray*}
& \pi^{\alpha_1}_0: & \mcX_0\rightarrow\mcX^{\alpha_1}_0, \mbox{ and } \\
& \pi^{\alpha_2}_0: & \mcX_0=\mbP(\mcV)\rightarrow\mcX^{\alpha_2}_0=\mbP^3
\end{eqnarray*}
Set $\Sigma_0(x):=\{x\}$. For each $k\geq 0$ let $\Sigma_{2k+1}(x)$ be the unique irreducible component of $(\pi^{\alpha_2}_0)^{-1}(\pi^{\alpha_2}_0(\Sigma_{2k}(x)))$ dominating $\pi^{\alpha_2}_0(\Sigma_{2k}(x))$, and $\Sigma_{2k+2}(x)$ be the unique irreducible component of $(\pi^{\alpha_1}_0)^{-1}(\pi^{\alpha_1}_0(\Sigma_{2k+1}(x)))$ dominating $\pi^{\alpha_1}_0(\Sigma_{2k+1}(x))$.
\end{cons}

\begin{lem}\label{lem. 26 A3 case}
In setting of Construction \ref{cons. 25 A3 case}, we have
\begin{eqnarray*}
\dim\Sigma_k(x)=k+1, \mbox{ where } 1\leq k\leq 4.
\end{eqnarray*}
In particular, $\Sigma_4(x)=\mcX_0$.
\end{lem}

\begin{proof}
By construction, $\Sigma_1(x)=(\pi^{\alpha_2}_0)^{-1}(\pi^{\alpha_2}_0(x))\cong\mbP^2$, which has dimension 2. Now we claim that for each $k\geq 1$, either $\Sigma_k(x)=\mcX_0$ or $\dim\Sigma_{k+1}(x)\geq\dim\Sigma_k(x)+1$.

Suppose $\dim\Sigma_{k+1}(x)=\dim \Sigma_k(x)$ for some $k\geq 1$. Then $\Sigma_{k+1}(x)=\Sigma_k(x)$. By construction of $\Sigma_k(x)$ and $\Sigma_{k+1}(x)$, $C_y^1\subset\Sigma_k(x)$ and $C_y^2\subset\Sigma_k(x)$ for $y\in\Sigma_k(x)$ general, $[C_y^1]\in\K^{\alpha_1}_y(\mcX_0)$ and $[C_y^2]\in\K^{\alpha_2}_y(\mcX_0)$ general. By Proposition \ref{prop. minimal rational chain connected}, we have $\Sigma_k(x)=\mcX_0$, and the claim holds.

By general choice of $x\in\mcX_0$ and the construction of $\Sigma_k(x)$, for each $i\geq 1$ we have
\begin{eqnarray*}
&& \dim \Sigma_{2i+1}(x)\leq \dim\pi^{\alpha_2}_0(\Sigma_{2i}(x))+2\leq\dim\Sigma_{2i}(x)+2, \\
&& \dim \Sigma_{2i}(x)\leq \dim\pi^{\alpha_1}_0(\Sigma_{2i-1}(x))+1\leq\dim\Sigma_{2i-1}(x)+1.
\end{eqnarray*}
Note that $\pi^{\alpha_2}_0(\Sigma_2(x))=\bigcup\limits_{[l]\in\K^{\alpha_1}_{\pi^{\alpha_2}_0(x)}(\mcX_0/\mbP^3)}l$, which has dimension 2 by Lemma \ref{lem. 23 A3 case}. Then the conclusion follows from the inequalities above.
\end{proof}

\begin{lem}\label{lem. 24 A3 case}
Take $t\in\mbP^3$ general, and set
\begin{eqnarray*}
&& \Lambda_1(t):=\bigcup\limits_{[l]\in\K^{\alpha_1}_t(\mcX_0/\mbP^3)}l\subset\mbP^3, \\
&& \Lambda_2(t):=\mbox{ Zariski closure of } \bigcup\limits_{[l]\in\K^{g}_{\Lambda_1(t)}} l \mbox{ in } \mbP^3,
\end{eqnarray*}
where we define
\begin{eqnarray*}
\K^{g}_{\Lambda_1(t)}:=\bigcup\limits_{z\in\Lambda_1(t)\mbox{ general}}\K^{\alpha_1}_z(\mcX_0/\mbP^3).
\end{eqnarray*}
Then $\Lambda_2(t)=\mbP^3$.
\end{lem}

\begin{proof}
Take $x\in\mbP^2_t:=(\pi^{\alpha_2}_0)^{-1}(t)$ general, then by construction we have
\begin{eqnarray*}
\pi^{\alpha_2}(\Sigma_{2k}(x))=\Lambda_k(t), \quad k=1, 2,
\end{eqnarray*}
where $\Sigma_{2k}(x)$ is as in Construction \ref{cons. 25 A3 case}.
By Lemma \ref{lem. 26 A3 case}, $\Sigma_4(x)=\mcX_0$, which implies the conclusion.
\end{proof}

\begin{lem} \label{lem. 21 A3 case}
Let $\mcL_\sigma\subset\mcV$ be the meromorphic line subbundle of $\mcV$ over $\mbP^3$ defining the meromorphic section $\sigma$ of $\pi^{\alpha_2}_0: \mcX_0=\mbP(\mcV)\rightarrow\mbP^3$, and $S_\sigma$ be the singular locus of $\sigma$. Then $\dim S_\sigma\leq 1$ and there exist nonempty Zariski open subsets $U''\subset U'\subset\mbP^3\setminus S_\sigma$ such that

$(i)$ $C_1\subset\mbP(\mcL_\sigma)$ for any $t\in U'$ and any $[C_1]\in\K^{\alpha_1}_{\sigma(t)}(\mcX_0)$;

$(ii)$ given any $t\in U''$ we have $M_2(t)=\mbP(\mcL_\sigma)$, where
\begin{eqnarray*}
&& M_1(t):=\bigcup\limits_{[C]\in\K^{\alpha_1}_{\sigma(t)}(\mcX_0)}C\subset\mbP(\mcL_\sigma), \\
&& M_2(t):=\mbox{ Zariski closure of }\bigcup\limits_{[C]\in\K^{\alpha_1}_{M_1(t)\cap\sigma(U')}} C \subset \mbP(\mcL_\sigma),
\end{eqnarray*}
where we define
\begin{eqnarray*}
\K^{\alpha_1}_{M_1(t)\cap\sigma(U')}:=\bigcup\limits_{x\in M_1(t)\cap\sigma(U')}\K^{\alpha_1}_x(\mcX_0).
\end{eqnarray*}
\end{lem}

\begin{proof}
Being the singular locus of a meromorphic section, the dimension of $S_\sigma$ is less or equal to $\dim\mbP^3-2=1$. By Lemma \ref{lem. 23 A3 case}, $\dim\K^{\alpha_1}_{\sigma(t)}(\mcX_0)\geq 1$ for $t\in\mbP^3$ general. By semicontinuity of the dimension function, $\dim\K^{\alpha_1}_x(\mcX_0)\geq 1$ for all $x\in\mbP(\mcL_\sigma)$. Hence $\mbP(\mcL_\sigma)\subset E(\K^{\alpha_1})$, where $E(\K^{\alpha_1})\subset\mcX_0$ is the union of $\K^{\alpha_1}(\mcX_0)$-equivalence classes that are of dimension at least two. By Proposition \ref{prop. BCD07 equivalence class of quasi-unsplit cycles}, $E(\K^{\alpha_1})$ is a Zariski closed subset of $\mcX_0$ and $\dim E(\K^{\alpha_1})\leq\dim\mcX_0-2=3$. By dimension reason the variety $\mbP(\mcL_\sigma)$ is an irreducible component of $E(\K^{\alpha_1})$.

Denote by $U$ the nonempty Zariski open subset of $\mbP(\mcL_\sigma)$ such that at any $x\in U$, $\mbP(\mcL_\sigma)$ is the unique irreducible component of $E(\K^{\alpha_1})$ containing $x$. Set $U':=\pi^{\alpha_2}_0(U)\setminus S_\sigma$, then the assertion $(i)$ of Lemma \ref{lem. 21 A3 case} holds.

By Lemma \ref{lem. 23 A3 case}, $\pi^{\alpha_2}_0(M_k(t))=\Lambda_k(t)$ for $k=1, 2$. Then by Lemma \ref{lem. 24 A3 case} $\phi(M_2(t))=\mbP^3$, implying that $\dim M_2(t)\geq 3$. Since $M_2(t)\subset\mbP(\mcL_\sigma)$ by the assertion $(i)$, we have $M_2(t)=\mbP(\mcL_\sigma)$, verifying the assertion $(ii)$.
\end{proof}

\begin{lem}\label{lem. 34 A3 case}
For $t\in\mbP^3$ general, $\mcC^{\alpha_1}_t(\mcX_0/\mbP^3)$ is a line in $\mbP(T_t\mbP^3)$. Furthermore, $\K^{\alpha_1}(\mcX_0/\mbP^3)$ is a hyperplane section of $\mbG(1, \mbP^3)\subset\mbP^5$.
\end{lem}

\begin{proof}
By Proposition \ref{prop. W=Da2 (A3 case)}, $[\msD^{\alpha_2}, \msD^{-2}]\subset\msD^{-2}$, where $\msD^{-2}$ is the weak derivative of $\msD=\msD^{\alpha_1}+\msD^{\alpha_2}$. It follows that $\mcE:=d\pi^{\alpha_2}_0(\msD^{-2})$ is a meromorphic distribution $\mcE$ on $\mbP^3$ of rank $2$, where $d\pi^{\alpha_2}_0: T(\mcX_0)\rightarrow T(\mbP^3)$ is the tangent map of $\pi^{\alpha_2}_0$. Take a general element $[C_1]\in\K^{\alpha_1}(\mcX_0)$. Then $T(C_1)=\msD^{\alpha_1}|_{C_1}\subset\msD^{-2}$, which implies that $T(\pi^{\alpha_2}_0(C_1))\subset\mcE|_{\pi^{\alpha_2}_0(C_1)}$. Hence at a general point $t\in\mbP^3$, we have $\mcC^{\alpha_1}_t(\mcX_0/\mbP^3)\subset\mbP(\mcE_t)$. Since $\K^{\alpha_1}(\mcX_0/\mbP^3)$ is a set of lines on $\mbP^3$, we have $\mcC^{\alpha_1}_t(\mcX_0/\mbP^3)\cong\K^{\alpha_1}_t(\mcX_0/\mbP^3)$, which is an irreducible rational curve by Lemma \ref{lem. 23 A3 case}. Hence $\mcC^{\alpha_1}_t(\mcX_0/\mbP^3)=\mbP(\mcE_t)$ is a line in $\mbP(T_t\mbP^3)$.
Moreover,
\begin{eqnarray*}
 \dim \K^{\alpha_1}(\mcX_0/\mbP^3) = \dim  \mathbb{P}^3 + \dim \mathcal{K}_{t}^{\alpha_1}(\mcX_0/\mbP^3) -1 = 3.
\end{eqnarray*}
Thus the variety $\K^{\alpha_1}(\mcX_0/\mbP^3)$ is an effective divisor on $\mbG(1, \mbP^3)$. Consider
\begin{eqnarray}\label{eqn. diagram for K(a1)(X0 over P3)}
\xymatrix{U^{\alpha_1}(\mcX_0/\mbP^3)\ar[r]\ar[d] & \K^{\alpha_1}(\mcX_0/\mbP^3)\ar[d] \\
\mbP(T\mbP^3)\ar[r]\ar[d] & \mbG(1, \mbP^3) \\
\mbP^3. &
}
\end{eqnarray}
Since for $t \in \mathbb{P}^3$ general,
\begin{eqnarray*}
\mathcal{C}_{t}^{\alpha_1}(\mcX_0/\mbP^3) = U^{\alpha_1}(\mcX_0/\mbP^3) \cap \mathbb{P} (T_t \mathbb{P}^3)
\end{eqnarray*}
is a line in $\mathbb{P} (T_t \mathbb{P}^3)$, we can conclude that $\K^{\alpha_1}(\mcX_0/\mbP^3)\in |\iota^*\mathcal{O}_{\mathbb{P}^5}(1)|$, where $\iota: \mbG(1, \mbP^3)\rightarrow\mbP^5$ is the Pl\"{u}ker embedding. Since $\mbG(1, \mbP^3)\subset\mbP^5$ is linearly normal, $\K^{\alpha_1}(\mcX_0/\mbP^3)$ is a hyperplane section of $(1, \mathbb{P}^3) \subset \mathbb{P}^5$.
\end{proof}

\subsubsection{Hyperplane bundles of $\mbP(\mcV)$ over $\mbP^3$} \label{subsubsection. hyperplane bundles of P(V) over P3}

\begin{nota}\label{nota. La1(X0) La2(X0) in A3 case}
Let $\mcL_0^{\alpha_i}$ be the Cartier divisor on $\mcX_0$ such that the intersection number $(\mcL_0^{\alpha_i}\cdot C_j)=\delta_{ij}$, where $[C_j]\in\K^{\alpha_j}(\mcX_0)$ and $1\leq i, j\leq 2$. In other words, $\mcL^{\alpha_i}_0:=\mcL^{\alpha_i}$, where $\mcL^{\alpha_i}$ is as in Proposition-Definition \ref{prop. divisors intersection theory on family X}. Denote by $|\mcL_0^{\alpha_i}|$ the corresponding linear system of effective Weil divisors on $\mcX_0$.
\end{nota}

\begin{lem}\label{lem. 36 A3 case}
We have $\dim |\mcL_0^{\alpha_1}|=3$ and $\dim |\mcL_0^{\alpha_2}|\geq 5$.
\end{lem}

\begin{proof}
Since $\mcX^{\alpha_2}_0=\mbP^3$, we have $\mcL_0^{\alpha_1}=(\pi^{\alpha_2}_0)^*\mcO_{\mbP^3}(1)$ and $\dim |\mcL_0^{\alpha_1}|=\dim\mbP^3=3$. There exists a holomorphic line bundle $\mcL^{\alpha_2}$ on $\mcX$ such that $\mcL_0^{\alpha_2}\cong \mcL^{\alpha_2}|_{\mcX_0}$ and for $0\neq t\in\Delta$, the linear system $|\mcL^{\alpha_2}_t|$ induces the morphism $$\mcX_t\cong F(1, 2; \C^4)\rightarrow Gr(2, \C^4)\subset\mbP^5,$$ where $\mcL^{\alpha_2}_t:=\mcL^{\alpha_2}|_{\mcX_t}$ By semicontinuity, we have $\dim |\mcL_0^{\alpha_2}|\geq 5$.
\end{proof}

\begin{nota}\label{nota. K(H, P3) in A3 case}
Take any $W \in |\mcL_0^{\alpha_2}|$. Then for $t\in\mbP^3$ general, we denote by
\begin{eqnarray*}
\mathcal{K}_{t}^{\alpha_1} (W/\mbP^3) := \psi( W_t) \subset  \mathcal{K}_{t}^{\alpha_1}(\mcX_0/\mbP^3),
\end{eqnarray*}
where $\psi$ is as in Lemma \ref{lem. P2t map to Kt(X0, P3) by psi in A3 case}. Set
\begin{eqnarray*}
\mathcal{K}^{\alpha_1} (W/\mbP^3) :=\mbox{ Zariski closure of } \bigcup\limits_{t \in \mathbb{P}^3 \text{ general}} \mathcal{K}_{t}^{\alpha_1} (W/\mbP^3) \mbox{ in } \mathcal{K}^{\alpha_1} (\mcX_0/\mathbb{P}^3).
\end{eqnarray*}
\end{nota}

\begin{lem}\label{lem. 40 A3 case}
In setting of Notation \ref{nota. K(H, P3) in A3 case}, there is an injective map
\begin{eqnarray*}
\theta: & \{W \in |\mcL_0^{\alpha_2}|\mid \mbP(\mcL_\sigma)\subset W \}\rightarrow \{\mbox{hyperplane sections of } \K^{\alpha_1}(\mcX_0/\mbP^3)\} \\
& W \longmapsto \mathcal{K}^{\alpha_1} (W/\mbP^3).
\end{eqnarray*}
\end{lem}

\begin{proof}
Take $t \in \mathbb{P}^3$ general. Then the fact $\sigma(t) \in W$ implies that $W_t$ is a line in $\mathbb{P}_{t}^2:=(\pi^{\alpha_2}_0)^{-1}(t)$ passing through $\sigma(t)$. By Lemma \ref{lem. P2t map to Kt(X0, P3) by psi in A3 case}, $\mathcal{K}_{t}^{\alpha_1} (W/\mbP^3)$ consists of a single element. Then $\mathcal{K}^{\alpha_1} (W/\mbP^3)$ is an effective divisor on $\K^{\alpha_1}(\mcX_0/\mbP^3)$. Similarly with the analysis for diagram \eqref{eqn. diagram for K(a1)(X0 over P3)}, we know that $\mathcal{K}^{\alpha_1} (W/\mbP^3)$ is a hyperplane section of $\K^{\alpha_1}(\mcX_0/\mbP^3)$.
\end{proof}

\begin{lem}\label{lem. 41 A3 case}
Take $W\in|\mcL_0^{\alpha_2}|$ general. Then $\sigma(t)\notin W$ for $t\in\mbP^3$ general.
\end{lem}

\begin{proof}
By Lemma \ref{lem. 34 A3 case} and Lemma \ref{lem. 40 A3 case}, the space $\{W\in|\mcL_0^{\alpha_2}|\mid \mbP(\mcL_\sigma)\subset W\}$
has dimension at most 4. On the other hand, $\dim |\mcL_0^{\alpha_2}|\geq 5$ by Lemma \ref{lem. 36 A3 case}. Then the conclusion follows.
\end{proof}

\begin{lem}\label{lem. 44 A3 case}
Take $W\in|\mcL_0^{\alpha_2}|$ general, and denote by
\begin{eqnarray}\label{eqn. S(W)}
S(W):=\{t\in\mbP^3\mid (\pi^{\alpha_2}_0)^{-1}(t)\subset W\}.
\end{eqnarray}
Then $\dim S(W)\leq 1$ and $W|_{\mbP^3\setminus S(W)}\rightarrow\mbP^3\setminus S(W)$ is a $\mbP^1$-bundle.
\end{lem}

\begin{proof}
As a Cartier divisor we have $\mcO_{\mcX_0}(W)|_{\mbP^2_t}\cong\mcO_{\mbP^2_t}(1)$ for any $t\in\mbP^3$, where $\mbP^2_t:=(\pi^{\alpha_2}_0)^{-1}(t)$. Thus for any $t\in\mbP^3\setminus S(W)$, the scheme-theoretic intersection of $W$ with $\mbP^2_t$ is a line. By dimension counting $\dim S(W)\leq \dim W-2=2$. If $\dim S(W)=2$, then the intersection number $(W\cdot C_1)>0$ for $[C_1]\in\K^{\alpha_1}(\mcX_0)$, contradicting our definition of $\mcL_0^{\alpha_2}$ in Notation \ref{nota. La1(X0) La2(X0) in A3 case}.
\end{proof}

\begin{lem}\label{lem. 46 A3 case}
Take $W\in|\mcL_0^{\alpha_2}|$ general, and denote by $S_W:=\pi^{\alpha_2}_0(\mbP(\mcL_\sigma)\cap W)\subset\mbP^3$. Then $\dim S_W\leq 1$.
\end{lem}

\begin{proof}
Now suppose $\dim S_W\geq 2$. By Lemma \ref{lem. 41 A3 case}, $S_W\neq\mbP^3$. Choose any irreducible component $\widetilde{S}_W$ of $S_W$ such that $\dim \widetilde{S}_W=2$.

We claim that for $\tilde{t}\in\widetilde{S}_W$ general, there exists $t\in U''$ and $[l]\in\K^{\alpha_1}_t(\mcX_0/\mbP^3)$ such that $\tilde{t}\in l$, where $U''$ is as in Lemma \ref{lem. 21 A3 case} $(ii)$.

Suppose the claim holds. By Lemma \ref{lem. 23 A3 case} there exists $[C]\in K^{\alpha_1}_{\sigma(t)}(\mcX_0)$ such that $\pi^{\alpha_2}_0(C)=l$. By Lemma \ref{lem. 21 A3 case}, $\dim S_\sigma\leq 1$, where $S_\sigma\subset\mbP^3$ is the singular locus of the section $\sigma$.  Then the general choice of $\tilde{t}$ in the divisor $\widetilde{S}_W\subset\mbP^3$ implies that $\tilde{t}\notin S_\sigma$. In particular, $\mbP(\mcL_\sigma)_{\tilde{t}}=\sigma(\tilde{t})\in C\cap W$. Since the intersection number $(W\cdot C)=0$, we have $C\subset W$, implying that $\sigma(t)\in W$. By Lemma \ref{lem. 21 A3 case}$(ii)$ and the fact $(W\cdot C_1)=0$ for any $[C_1]\in\K^{\alpha_1}(\mcX_0)$, we have $\mbP(\mcL_\sigma)\subset W$. This contradicts Lemma \ref{lem. 41 A3 case}. Hence we obtain the conclusion of Lemma \ref{lem. 46 A3 case}.

Now we turn to prove the claim. Suppose it fails. Let $A$ be the Zariski closure of the union $\bigcup(l\cap\widetilde{S}_W)$ in $\widetilde{S}_W$, where $[l]$ runs over the set $\bigcup\limits_{t\in\mbP^3\mbox{ general}}\K^{\alpha_1}_t(\mcX_0/\mbP^3)$. By assumption, $\dim A\leq \dim\widetilde{S}_W-1=1$.

Since every element in $\K^{\alpha_1}(\mcX_0/\mbP^3)$ has a nonempty intersection with $\widetilde{S}_W$, there is an irreducible component $\widetilde{A}$ of $A$ such that
\begin{eqnarray*}
\dim\K^{\alpha_1}_s(\mcX_0/\mbP^3)\geq\dim\K^{\alpha_1}(\mcX_0/\mbP^3)-\dim\widetilde{A}\geq 2 \mbox{ for each } s\in\widetilde{A}.
\end{eqnarray*}
Since $\K^{\alpha_1}_s(\mcX_0/\mbP^3)\cong\mcC^{\alpha_1}_s(\mcX_0/\mbP^3)\subset\mbP(T_s\mbP^3)$, we know that
\begin{eqnarray*}
\dim\widetilde{A}=1, & \K^{\alpha_1}_s(\mcX_0/\mbP^3)\cong\mcC^{\alpha_1}_s(\mcX_0/\mbP^3)=\mbP(T_s\mbP^3)
\end{eqnarray*}
and $[\langle t, s\rangle]\in\K^{\alpha_1}_t(\mcX_0/\mbP^3)$ for all $s\in\widetilde{A}$ and all $t\in\mbP^3\setminus\{s\}$.

Take $t\in\mbP^3$ general. By Lemma \ref{lem. 34 A3 case} and the conclusions above, $\K^{\alpha_1}_t(\mcX_0/\mbP^3)=\{[\langle t, s\rangle]\mid s\in\widetilde{A}\}$,
and the join variety $J(t, \widetilde{A}):=\bigcup\limits_{s\in\widetilde{A}}\langle t, s\rangle$
is a plane in $\mbP^3$. Thus in the notations of Lemma \ref{lem. 24 A3 case}, we have $\Lambda_1(t)=J(t, \widetilde{A})$. For $t'\in\Lambda_1(t)$ general, the same reason implies that $\Lambda_1(t')=J(t', \widetilde{A})=J(t, \widetilde{A})=\Lambda_1(t)$.
It follows that $\Lambda_2(t)=\Lambda_1(t)\subsetneqq\mbP^3$, contradicting Lemma \ref{lem. 24 A3 case}. Hence, the claim holds.
\end{proof}

\begin{lem}\label{lem. 32 A3 case}
There exists a meromorphic vector subbundle $\mcL_W\subset\mcV$ of rank two over $\mbP^3$ and a closed subvariety $S_W\subset\mbP^3$ such that

$(i)$ $\dim S_W\leq 1$;

$(ii)$ both $\mcL_\sigma$ and $\mcL_W$ are holomorphic vector bundles on $\mbP^3\setminus S_W$, where $\mcL_\sigma$ is as in Lemma \ref{lem. 21 A3 case};

$(iii)$ there is a direct sum decomposition $\mcV|_{\mbP^3\setminus S_W}=\mcL_\sigma|_{\mbP^3\setminus S_W}\oplus\mcL_W|_{\mbP^3\setminus S_W}$;

$(iv)$ $\mbP(\mcL_W)\in|\mcL_0^{\alpha_2}|$ is a chosen general divisor.
\end{lem}

\begin{proof}
It is a direct consequence of Lemma \ref{lem. 44 A3 case}, Lemma \ref{lem. 46 A3 case} and the fact $\dim S_\sigma\leq 1$, where $S_\sigma$ is the singular locus of the section $\mbP(\mcL_\sigma)$.
\end{proof}

To continue, we need to collect a result of decomposition of vector bundles, which can be found on page 409 in \cite{HM98}. See also \cite[Proposition 5]{Li18} for an explicit statement with a brief proof.

\begin{prop}\cite[page 409]{HM98}\label{prop. extending direct summmand bundles}
Let $\mcE$ be a vector bundle over a connected complex manifold $Y$. Suppose there is a complex subvariety $A\subset Y$ and vector bundles $\mcE_1$ and $\mcE_2$ over $Y\setminus A$ such that $\dim A \leq \dim Y - 2$ and $\mcE|_{Y\setminus A}=\mcE_1\oplus\mcE_2$. Then $\mcE_1$ and $\mcE_2$ can be extended uniquely as vector bundles $\mcE'_1$ and $\mcE'_2$ over $Y$ such that $\mcE=\mcE'_1\oplus\mcE'_2$.
\end{prop}

As a direct consequence of Lemma \ref{lem. 32 A3 case} and Proposition \ref{prop. extending direct summmand bundles}, we have the following result.

\begin{prop}\label{prop. 31 A3 case}
In setting of Lemma \ref{lem. 32 A3 case}, both $\mcL_\sigma$ and $\mcL_W$ are holomorphic vector bundles on $\mbP^3$, and $\mcV=\mcL_\sigma\oplus\mcL_W$.
\end{prop}

\begin{lem}\label{lem. C(a1) on H in A3 case}
In setting of Proposition \ref{prop. 31 A3 case}, the followings hold.

$(i)$ For any $[l]\in\K^{\alpha_1}(\mcX_0/\mbP^3)$, $\mcL_\sigma|_l=\mcO(2)$ and $\mcL_W|_l=\mcO(2)\oplus\mcO$.

$(ii)$ For any $[l]\in\K^{\alpha_1}(\mcX_0/\mbP^3)$, there exists a unique $[C_l]\in\K^{\alpha_1}(\mcX_0)$ such that $C_l\subset W$ and $\pi^{\alpha_2}_0(C_l)=l$. Moreover, this curve $C_l\cong\mbP^1$.

$(iii)$ For any $x\in W$, $\K^{\alpha_1}_x(\mcX_0)$ consists of a single element, denoted by $[C_x]$. Moreover, this curve $C_x\subset W$ and $C_x\cong\mbP^1$.
\end{lem}

\begin{proof}
By Proposition \ref{prop. V =O(2, 2, 0) along C1 in A3 case},
\begin{eqnarray}\label{E00_1}
\mathcal{V} |_{l} = \mathcal{O}(2)^2 \oplus \mathcal{O}, \quad \text{for}~ [l] \in \K^{\alpha_1}(\mcX_0/\mbP^3)~ \text{general}.
\end{eqnarray}
By Proposition \ref{prop. 19 A3 case}$(i)$, the restriction of $\mathcal{V}$ on any line of $\mathbb{P}^3$ is either $\mathcal{O}(2)^2\oplus\mcO$ or $\mathcal{O}(2) \oplus \mathcal{O}(1)^2$. Then by \eqref{E00_1}, we conclude that
\begin{eqnarray}\label{eqn. V=O(2, 2, 0) along any symplectic line in A3 case}
\mathcal{V} |_{l} = \mathcal{O}(2)^2 \oplus \mathcal{O}, \quad \text{for any}~ [l] \in \K^{\alpha_1}(\mcX_0/\mbP^3) = \mathcal{K}^{\alpha_1}(W/\mbP^3).
\end{eqnarray}
This is because a positive dimensional family of vector bundles over $\mathbb{P}^1$ of type $\mathcal{O}(2)^2 \oplus \mathcal{O}$ can not have a limit of type $\mathcal{O}(2) \oplus \mathcal{O}(1)^2$.

Now take any $[l] \in \K^{\alpha_1}(\mcX_0/\mbP^3)$, we have
$\mathcal{L}_{\sigma} |_{l} = \mathcal{O}(2)$ by Proposition \ref{prop. 19 A3 case}$(ii)$. Thus by \eqref{eqn. V=O(2, 2, 0) along any symplectic line in A3 case} and Proposition \ref{prop. 31 A3 case}, $\mathcal{L}_{W} |_{l} = \mathcal{O}(2) \oplus \mathcal{O}$, verifying the assertion $(i)$. It follows that there exists a unique $[C_{l}] \in \mathcal{K}^{\alpha_1}(\mathcal{X}_0)$ such that $C_{l} \subset W = \mathbb{P}(\mathcal{L}_{W})$, and $\pi^{\alpha_2}_0(C_{l}) = l$. In fact $C_l=\mbP(\mcO(2)|_l)\subset\mbP((\mcO(2)\oplus\mcO)|_l)=\mbP(\mcL_W|_l)$.
Moreover $C_{l} \cong \mathbb{P}^1$, verifying the assertion $(ii)$.

Take any $[C]\in\K^{\alpha_1}(\mcX_0)$. Since $(W\cdot C)=0$, either $C\subset W$ or $C\cap W=\emptyset$. Then the assertion $(iii)$ follows from $(i)$ and $(ii)$.
\end{proof}

\begin{lem}\label{lem. 49 A3 case}
In setting of Proposition \ref{prop. 31 A3 case}, the variety $\K^{\alpha_1}(\mcX_0/\mbP^3)$ is a smooth hyperplane section of $\mbG(1, \mbP^3) \subset \mathbb{P}^5$, and $W \cong C_2/(P_{\beta_1}\cap P_{\beta_2})$, where $\beta_1$ and $\beta_2$ are the short and long simple root of $C_2$ respectively.
\end{lem}

\begin{proof}
By Lemma \ref{lem. 34 A3 case}, $\K^{\alpha_1}(\mcX_0/\mbP^3)$ is a hyperplane section of $\mbG(1, \mbP^3) \subset \mathbb{P}^5$. By Proposition \ref{prop. BCD07 equivalence class of quasi-unsplit cycles} and Lemma \ref{lem. C(a1) on H in A3 case}, there is a $\mbP^1$-fibration $\varphi: W\rightarrow\K^{\alpha_1}(W)=\K^{\alpha_1}(\mcX_0/\mbP^3)$, where $\K^{\alpha_1}(W)$ is the set of $[C]\in\K^{\alpha_1}(\mcX_0)$ such that $C\subset W$. The variety $\K^{\alpha_1}(\mcX_0/\mbP^3)$ is smooth because so is $W$. Then there exists a nondegenerate form $\omega\in\wedge^2(\C^4)^*$ such that $\mcX_0^{\alpha_2}=\mbP^3=\mbP(\C^4)$,
\begin{eqnarray}\label{eqn. symplectic form}
\K^{\alpha_1}(\mcX_0/\mbP^3)=\{[A]\in Gr(2, \C^4)\mid \omega(A, A)=0\},
\end{eqnarray}
and $\pi^{\alpha_2}_0|_W: W\rightarrow\mbP^3$ is the evaluation morphism of the family $\K^{\alpha_1}(\mcX_0/\mbP^3)$. Then the conclusion follows.
\end{proof}

Now we can complete the proof of Theorem \ref{thm. intro unique degeneration of F(1, 2, C4)}.

\begin{proof}[Proof of Theorem \ref{thm. intro unique degeneration of F(1, 2, C4)}]
By Proposition \ref{prop. X=P(V) in A3 case} and Proposition \ref{prop. 31 A3 case}, $\mcX_0\cong\mbP(\mcV)$, and $\mcV\cong\mcL^\sigma\oplus\mcL_W$. By Proposition \ref{prop. 19 A3 case}$(ii)$, $\mcL_\sigma\cong\mcO(2)$. By Lemma \ref{lem. 49 A3 case}, $\mcL_W\cong\mcL^\omega\otimes\mcO(k)$ for some $k\in\mbZ$, where $\omega$ is the symplectic form on $\C^4$ satisfying \eqref{eqn. symplectic form}.
Take any line $l\subset\mbP^3$. We have
\begin{eqnarray*}
&& \deg(\mcL_W|_l)=\deg(\mcV|_l)-\deg(\mcL_\sigma|_l)=2, \\
&& \deg(\mcL^\omega|_l)=\deg(T\mbP^3|_l)-\deg(\mcO(2)|_l)=2.
\end{eqnarray*}
Then $k=0$ and $\mcL_W\cong\mcL^\omega$. Hence $\mcV\cong\mcO(2)\oplus\mcL^\omega$ and $\mcX_0\cong F^d(1, 2; \C^4)$.
\end{proof}


\subsubsection{Properties of $F^d(1, 2; \C^4)$} \label{subsubsection. properties of Fd(1, 2; C4)}

For the convenience of discussion later, we give several basic properties of the manifold $F^d(1, 2; \C^4)$ in Construction \ref{cons. intro Fd(1, 2; C4)}. All these properties are straight-forward from the construction. They have also been proved in a more involved way in the previous arguments in subsection \ref{section. unique degeneration of (A3, a1, a2)} by realizing $F^d(1, 2; \C^4)$ as the a priori unclear Fano degeneration of $A_3/P_{\{\alpha_1, \alpha_2\}}$, see Lemma \ref{lem. A3 degeneration unique W when generating D(-3)}, Corollary \ref{cor. A3 degenerate symb algebra}, Corollary \ref{cor. 20 A3 case}, Lemma \ref{lem. P2t map to Kt(X0, P3) by psi in A3 case} for the corresponding statements of them.

\begin{nota}\label{nota. Fd(1, 2; C4) bundle section etc}
In setting of Construction \ref{cons. intro Fd(1, 2; C4)}, denote by $\phi: F^d(1, 2; \C^4)\rightarrow\mbP^3$ the $\mbP^2$-bundle, and let $\sigma: \mbP^3\rightarrow\mbP(\mcL_\sigma)\subset F^d(1, 2; \C^4)$ be the holomorphic section. Given a point $x\in F^d(1, 2; \C^4)\setminus\mbP(\mcL_\sigma)$, denote by $l_x$ the line $\langle x, \sigma(\phi(x))\rangle$ in the projective plane $\phi^{-1}(\phi(x))\cong\mbP^2$. By abuse of notations (to be compatible with those in Section \ref{section. unique degeneration of (A3, a1, a2)}), we denote by $\msD^{\alpha_1}$ the meromorphic distribution of rank one on $F^d(1, 2; \C^4)$ whose general leaves are minimal rational curves biholomorphically sent to isotropic lines in $\mbP^3$, and by $\K^{\alpha_1}(F^d(1, 2; \C^4))$ the closure this family of minimal rational curves. Set $\msD^{\alpha_2}:=T^\phi$ and $\msD:=\msD^{\alpha_1}+\msD^{\alpha_2}$. Denote by $\K^{\alpha_2}(F^d(1, 2; \C^4))$ the family of minimal rational curves which are lines in the fibers of $\phi$.
\end{nota}

The following two propositions are immediate from the constructions.

\begin{prop}\label{prop. properties of sigma in Deform-A3 case}
At any point $x\in F^d(1, 2; \C^4)\setminus\mbP(\mcL_\sigma)$, $\K^{\alpha_1}_x(\mcX_0)$ consists of a unique element, denoted by $[C_x]$. Two points $y, z\in\mbP^2_t\setminus\{\sigma(t)\}$ satisfy $\phi(C_y)=\phi(C_z)$ if and only if the two lines $\langle y, \sigma(t)\rangle$ and $\langle z, \sigma(t)\rangle$ in $\mbP^2_t$ coincide, where $t\in\mcX^{\alpha_2}_0$ is an arbitrary point and $\mbP^2_t:=\phi^{-1}(t)$.
\end{prop}

\begin{prop}\label{prop. rational map over P3 in Deform-A3 case}
In setting of Construction \ref{cons. intro Fd(1, 2; C4)} the surjective homomorphism $\mcL_\sigma\oplus\mcL^\omega\rightarrow\mcL_\sigma$ induces a rational map $F^d(1, 2; \C^4)\dashrightarrow \mbP(\mcL^\omega)\cong C_2/B$ over $\mbP^3$. It is a linear projection from $\mbP^2_t:=\phi^{-1}(t)$ with center $\sigma(t)$ over each $t\in\mbP^3$.
\end{prop}

\begin{prop}\label{prop. unique distribution N in Deform-A3 case}
In setting of Notation \ref{nota. Fd(1, 2; C4) bundle section etc}, define a meromorphic distribution $\mcN$ on $F^d(1, 2; \C^4)$ such that $\mcN_x=T_x(l_x)$ at any point $x\in F^d(1, 2; \C^4)\setminus\mbP(\mcL_\sigma)$. Then $\mcN$ is the unique meromorphic line subbundle of $\msD$ on $F^d(1, 2; \C^4)$ such that $[\mcN, \msD]\subset\msD$. Moreover, $[\mcN, \msD^{\alpha_1}]\subset\mcN+\msD^{\alpha_1}$.
\end{prop}

\begin{proof}
The leaf of $\mcN$ passing through a point $x\in F^d(1, 2; \C^4)\setminus\mbP(\mcL_\sigma)$ is $l_x^o:=\l_x\setminus\{\sigma(t)\}$, where $t:=\phi(x)$ and $l_x:=\langle x, \sigma(t)\rangle$. The the leaf of $\mcD^{\alpha_1}$ passing through a point $y\in l_x^o$ is $C_y$, where $[C_y]$ is the unique element of $\K^{\alpha_1}_y(\mcX_0)$. Since $\overline{\bigcup\limits_{y\in l_x^o}C_y}\cong\phi(C_x)\times l_x$, we have $[\mcN, \msD^{\alpha_1}]\subset\mcN+\msD^{\alpha_1}$. Since $\msD^{\alpha_2}$ is integrable and $\mcN\subset\msD^{\alpha_2}$, we have $[\mcN, \msD^{\alpha_2}]\subset\msD^{\alpha_2}$. It follows that $[\mcN, \msD]\subset\msD$.

If the uniqueness of $\mcN$ fails, then the rank three distribution $\mcD$ has to be integrable. However one can easily check the $F^d(1, 2; \C^4)$ is chained-connected by the family $\bigcup\limits_{i=1, 2}\K^{\alpha_i}(F^d(1, 2;\ C^4))$. It is a contradiction. Hence $\mcN$ is unique.
\end{proof}

\begin{prop}\label{prop. symbol algebra in Deform-A3 case}
At each point $x\in F^d(1, 2; \C^4)\setminus\mbP(\mcL_\sigma)$, the symbol algebra $\Symb_x(\msD)\cong\mfg_-(C_2)\oplus\mfg_-(A_1)$. Moreover, this isomorphism is induced by the identification $\mfg_-(A_1)=\mcN_x$, $\msD^{\alpha_2}_x=\mfg_-(\alpha_2)+\mfg_-(A_1)$ and $\msD^{\alpha_1}_x=\mfg_-(\alpha_1)$, where $\alpha_1$ and $\alpha_2$ is the long and short simple root of $C_2$ respectively.
\end{prop}

\begin{proof}
It follows from Proposition \ref{prop. unique distribution N in Deform-A3 case} and Construction \ref{cons. intro Fd(1, 2; C4)} directly.
\end{proof}

\subsection{Proof of Proposition \ref{prop. intro A4 D5 submaximal rigidity}} \label{subsection. rigidity of (A4, (a2, a3, a4))}

The main aim of this subsection is to show the following proposition, from which we can complete the proof of Proposition \ref{prop. intro A4 D5 submaximal rigidity}.

\begin{prop} \label{prop. (A4, a234) rigidity}
The manifold $A_4/P_{\{\alpha_2, \alpha_3, \alpha_4\}}$ is rigid under Fano deformation.
\end{prop}

\begin{proof}[Proof of Proposition \ref{prop. intro A4 D5 submaximal rigidity}]
$(i)$ Consider the Fano deformation rigidity of $A_4/P_I$ with $|I|=3$. The set of simple roots is $R=\{\alpha_1,\ldots,\alpha_4\}$. The manifolds $A_4/P_{R\setminus\{\alpha_1\}}$ and $A_4/P_{R\setminus\{\alpha_4\}}$ are biholomorphic to each other, which are rigid under Fano deformation by Proposition \ref{prop. (A4, a234) rigidity}. The manifolds $A_4/P_{R\setminus\{\alpha_2\}}$ and $A_4/P_{R\setminus\{\alpha_3\}}$ are biholomorphic to each other, which are rigid under Fano deformation by Proposition \ref{prop. intro (Am, a1a2am) rigidity}.

$(ii)$ Consider the Fano deformation rigidity of $\bS:=D_5/P_I$ with $|I|=4$. Set $J:=R\setminus I=\{\alpha_i\}$ for some $i$, where $R$ is the set of simple roots. Take any $J$-connected pair $\beta_1\neq\beta_2\in I$. There exists $\beta_3\in I\setminus\{\beta_1, \beta_2\}$ such that the manifold $\bS^{\beta_1, \beta_2, \beta_3}$ is biholomorphic to $A_4/P_{I'}$ with $|I'|=3$ or $4$. The latter is rigid under Fano deformation by $(i)$ as well as Theorem \ref{thm. intro complete flag mfds}. By Corollary \ref{cor. reduction results refined version}, $D_5/P_I$ is rigid under Fano deformation .
\end{proof}

To prove Proposition \ref{prop. (A4, a234) rigidity}, it suffices to deduce a contradiction in the following setting.

\begin{set}\label{setup. (A4, a2a3a4)}
Let $\pi: \mcX\rightarrow\Delta$ be a holomorphic map such that $\mcX_t\cong \bS$ for all $t\neq 0$, $\mcX_0$ is a connected Fano manifold and $\mcX_0\ncong \bS$, where $\bS:=A_4/P_{\{\alpha_2, \alpha_3, \alpha_4\}}$.
\end{set}

\begin{rmk}
Let us firstly explain the idea to prove Proposition \ref{prop. (A4, a234) rigidity} in the following, while the rigorous proof is not no so immediate from this idea. In Setting \ref{setup. (A4, a2a3a4)}, $\mcX_0$ has to be a compactification of the total space of the normal bundle $N_{U/\bS}$, where $U$ is the inverse image of some hyperplane section of $A_4/P_{\alpha_2}=Gr(2, \C^5)\subset\mbP^9$ under the natural morphism $\bS\rightarrow A_4/P_{\alpha_2}$. On the other hand, we can show that any Fano deformation of $\bS$ must be a $\mbP^2$-bundle over $A_4/P_{\{\alpha_3, \alpha_4\}}=F(3, 4; \C^5)$, while the compactification $\mcX_0$ of $N_{U/\bS}$ does not have such a projective bundle structure.
\end{rmk}

\begin{prop}\label{prop. 4 A4 case}
In Setting \ref{setup. (A4, a2a3a4)}, take a general point $x\in\mcX_0$. Then $F^{\alpha_2}_x\cong\mbP^2$, $F^{\alpha_3}_x\cong\mbP^1$, $F^{\alpha_4}_x\cong\mbP^1$, $F^{\alpha_2, \alpha_4}_x\cong\mbP^2\times\mbP^1$ and $F^{\alpha_3, \alpha_4}_x\cong\mbP(T\mbP^2)=F(1, 2; \C^3)$ respectively.
\end{prop}

\begin{proof}
The assertions for $F_x^{\alpha_i}$, $F^{\alpha_2, \alpha_4}_x$ and $F^{\alpha_3, \alpha_4}_x$ follow from the rigidity of projective spaces, Proposition \ref{prop. invariance of product structure under Fano deformation} and Theorem \ref{thm. intro complete flag mfds} respectively.
\end{proof}

\begin{prop}\label{prop. fibers F(a2a3) in A4 case}
In Setting \ref{setup. (A4, a2a3a4)}, take a general point $x\in\mcX_0$. Then $F^{\alpha_2, \alpha_3}\cong F^d(1, 2; \C^4)$, where $F^d(1, 2; \C^4)$ is as in Construction \ref{cons. intro Fd(1, 2; C4)}.
\end{prop}

\begin{proof}
By Theorem \ref{thm. intro unique degeneration of F(1, 2, C4)}, either $F^{\alpha_2, \alpha_3}\cong F(1, 2; \C^4)$ or $F^{\alpha_2, \alpha_3}\cong F^d(1, 2; \C^4)$. In the former case, $\mcX_0\cong A_4/P_{\{\alpha_2, \alpha_3, \alpha_4\}}$ by Theorem \ref{thm. reduction theorem} and Proposition \ref{prop. 4 A4 case}. This contradicts our assumption in Setting \ref{setup. (A4, a2a3a4)}.
\end{proof}

\begin{prop}\label{prop. 7 A4 case}
In Setting \ref{setup. (A4, a2a3a4)}, the morphism $\pi^{\alpha_2}_0: \mcX_0\rightarrow\mcX^{\alpha_2}_0$ is a $\mbP^2$-bundle. In particular, the variety $\mcX^{\alpha_2}_0$ is smooth.
\end{prop}

\begin{proof}
By formula (3.4) in \cite{WW17}, the cohomology ring $H^*(A_n/P_{\{\alpha_1, \alpha_2\}}, \mbQ)$ is generated by $H^2(A_n/P_{\{\alpha_1, \alpha_2\}}, \mbQ)$. Then the conclusion of Proposition \ref{prop. 7 A4 case} follows from Proposition \ref{prop. criterion Pk bundle by Weber Wisniewski} immediately.
\end{proof}

\begin{conv}
In Subsection \ref{subsection. rigidity of (A4, (a2, a3, a4))}, we denote by $\msD^{\alpha_i}$, $\msD^{\alpha_i, \alpha_j}$ $\msD$ and $\msD^{-k}$ the restriction of $\mcD^{\alpha_i}$, $\mcD^{\alpha_i, \alpha_j}$, $\mcD$ and $\mcD^{-k}$ on $\mcX_0$ respectively, where the latter is defined in Notation \ref{nota. distribution D(A) and symbol algebra}.
\end{conv}

Now let us turn to analysis the symbol algebra $\Symb(\msD)$ on $\mcX_0$.

\begin{lem}\label{lem. 14 A4 case}
In Setting \ref{setup. (A4, a2a3a4)}, there exists a unique meromorphic distribution $\mcN\subset\msD^{\alpha_2}$ of rank one over $\mcX_0$ such that the Levi bracket of vector fields $[\mcN, \msD]\subset\msD$.
\end{lem}

\begin{proof}
By Proposition \ref{prop. fibers F(a2a3) in A4 case}, $F^{\alpha_2, \alpha_3}_x\cong F^d(1, 2; \C^4)$ for $x\in\mcX_0$ general. Then by Proposition \ref{prop. unique distribution N in Deform-A3 case} there exists a unique meromorphic line subbundle $\mcN\subset\msD^{\alpha_2}$ over $\mcX_0$ such that $[\mcN, \msD^{\alpha_3}]\subset\mcN+\msD^{\alpha_3}$. By Proposition \ref{prop. 4 A4 case}, $F^{\alpha_2, \alpha_4}_x\cong F^{\alpha_2}_x\times F^{\alpha_4}_x$ for $x\in\mcX_0$ general. Then $[\msD^{\alpha_2}, \msD^{\alpha_4}]\subset\msD^{\alpha_2}+\msD^{\alpha_4}$, implying the conclusion.
\end{proof}

\begin{nota}\label{nota. m- A4 case}
We construct a graded nilpotent Lie algebra $\mfm_-:=\bigoplus\limits_{k\geq 1}\mfm_{-k}$ as follows:
\begin{eqnarray*}
&& \mfm_{-1}=\bigoplus\limits_{1\leq i\leq 4}\C v_i, \\
&& \mfm_{-2}=\C v_{23}\oplus\C v_{34}, \\
&& \mfm_{-3}=\C v_{233}\oplus\C v_{234}, \\
&& \mfm_{-4}=\C v_{2334}, \\
&& \mfm_{-k}=0, \quad k\geq 5,
\end{eqnarray*}
where $v_{i_1\ldots i_m}:=[v_{i_1\ldots i_{m-1}}, v_{i_m}]$. The Lie algebra structure on $\mfm_-$ is defined uniquely by the following rules:
\begin{eqnarray*}
[\mfm_{-i}, \mfm_{-j}]\subset\mfm_{-i-j}, & [v_1, \mfm_-]=0, & [v_{23}, v_{34}]=\frac{1}{2}v_{2334},
\end{eqnarray*}
and there is a table of Lie brackets
\begin{eqnarray}\label{eqn. table brackets imegary degenerate A4 case}
\begin{tabular}{|c|c|c|c|c|}
\hline
& $v_{23}$ & $v_{34}$ & $v_{233}$ & $v_{234}$ \\ \hline
$v_2$ & 0 & $v_{234}$ & 0 & 0 \\ \hline
$v_3$ & $-v_{233}$ & 0 & 0 & $-\frac{1}{2}v_{2334}$ \\ \hline
$v_4$ & $-v_{234}$ & 0 & $-v_{2334}$ & 0 \\ \hline
\end{tabular}
\end{eqnarray}
In the table above, we compute the Lie bracket of left end entry with top end entry. For example, $[v_4, v_{23}]=-v_{234}$ and $[v_3, v_{234}]=-\frac{1}{2}v_{2334}$.
\end{nota}

\begin{lem}\label{lem. 16 A4 case}
In Setting \ref{setup. (A4, a2a3a4)}, the symbol algebra of $\msD$ at a general point $x\in\mcX_0$ is isomorphic to $\mfm_-$ in Notation \ref{nota. m- A4 case}, where we have identifications $\mcN_x=\C v_1$, $\msD^{\alpha_2}_x=\C v_1 + \C v_2$, $\msD^{\alpha_3}_x=\C v_3$ and $\msD^{\alpha_4}_x=\C v_4$.
\end{lem}

\begin{proof}
By Proposition \ref{prop. 4 A4 case}, Proposition \ref{prop. fibers F(a2a3) in A4 case} and Proposition \ref{prop. symbol algebra in Deform-A3 case} (see also Remark \ref{rmk. symbol algebra C2A1 in A3 case}$(ii)$), we have the description of $\mfm(\alpha_i)$ and $\mfm(\alpha_i, \alpha_j)$ for $2\leq i\neq j\leq 4$. In particular, in $\mfm_x(\alpha_2, \alpha_3, \alpha_4):=\Symb_x(\msD)$ we have
\begin{eqnarray*}
&& [v_1, v_i]=0, i=2, 3, 4, \quad (\ad v_2)^2(v_3)=0, \quad (\ad v_3)^3(v_2)=0, \\
&& [v_2, v_4]=0, \quad (\ad v_3)^2(v_4)=0, \quad (\ad v_4)^2(v_3)=0.
\end{eqnarray*}
Then by Proposition \ref{prop. Serre's theorem}, $\Symb_x(\msD)$ is a quotient algebra of $\mfg_-:=\mfg_-(C_3)\oplus\mfg_-(A_1)$. More precisely, $\mfg_-:=\bigoplus\limits_{k\geq 1}\mfg_{-k}$ as follows:
\begin{eqnarray*}
&& \mfg_{-1}=\bigoplus\limits_{1\leq i\leq 4}\C v_i, \\
&& \mfg_{-2}=\C v_{23}\oplus\C v_{34}, \\
&& \mfg_{-3}=\C v_{233}\oplus\C v_{234}, \\
&& \mfg_{-4}=\C v_{2334}, \\
&& \mfg_{-5}=\C v_{23344}, \\
&& \mfg_{-k}=0, \quad k\geq 6.
\end{eqnarray*}
Denote by $\mfq$ the ideal of $\mfg_-$ such that $\Symb_x(\msD)=\mfg_-/\mfq$ as graded nilpotent Lie algebra. By Proposition \ref{prop. rationally chain connected distribution version}, $\dim \Symb_x(\msD)=\dim T_x\mcX_0=9$, which implies that $\dim\mfq=\dim\mfg_- - \dim\Symb_x(\msD)=1$.

To complete the proof of Lemma \ref{lem. 16 A4 case}, it suffices to show the claim that $\mfq=\C v_0$, where $v_0:=v_{23344}+\lambda v_1$ for some $\lambda\in\C$. Note that the graded Lie algebra structure on $\mfg_-/\C w_0$ is independent of the choice of $\lambda\in\C$.

Suppose the claim fails. Then there exists $1\leq k_0\leq 4$ such that $\mfq=\C v_0$ and $v_0=\lambda v_1 + v'_0 + v''_0$, where $v''_0\in\bigoplus\limits_{k\geq k_0+1}\mfg_{-k}$, $0\neq v'_0\in\mfg_{-k_0}$ if $k_0\geq 2$, and $0\neq v'_0\in\bigoplus\limits_{2\leq i\leq 4}\C v_i$ if $k_0=1$. Then there exists $2\leq j\leq 4$ such that $[v_j, v'_0]\neq 0$, see table \eqref{eqn. table brackets imegary degenerate A4 case}. Then $0\neq [v_j, v_0]\in\bigoplus\limits_{k\geq k_0+1}\mfg_{-k}$. Since $\mfq$ is an ideal of $\mfg_-$, we have $0\neq [v_j, v_0]\in\mfq=\C v_0$. In particular, $[v_j, v_0]$ has a nonzero component in $\mfg_{-k_0}$. It is a contradiction. Hence the claim holds.
\end{proof}

\begin{lem}\label{lem. 11 A4 case}
In Setting \ref{setup. (A4, a2a3a4)} the Frobenius bracket of $(T^{\pi^{\alpha_2, \alpha_3}}+T^{\pi^{\alpha_2, \alpha_4}})|_{\mcX_0}$ induces a homomorphism of meromorphic vector bundles over $\mcX_0$:
\begin{eqnarray*}
(T^{\pi^{\alpha_2, \alpha_3}}/T^{\pi^{\alpha_2}})|_{\mcX_0}\otimes(T^{\pi^{\alpha_2, \alpha_4}}/T^{\pi^{\alpha_2}})|_{\mcX_0}\rightarrow T\mcX_0/(T^{\pi^{\alpha_2, \alpha_3}}+T^{\pi^{\alpha_2, \alpha_3}})|_{\mcX_0},
\end{eqnarray*}
which is a surjective homomorphism over a nonempty Zariski open subset of $\mcX_0$.
\end{lem}

\begin{proof}
It is a direct consequence of Lemma \ref{lem. 16 A4 case}. More precisely, the weak derivatives of $\msD^{\alpha_2}$, $\msD^{\alpha_2}+\msD^{\alpha_3}$ and $\msD^{\alpha_2}+\msD^{\alpha_4}$ induces symbol algebras at a general point $x\in\mcX_0$ as follows:
\begin{eqnarray*}
&& {\rm gr}(T^{\pi^{\alpha_2}})=\C v_1\oplus\C v_2, \\
&& {\rm gr}(T^{\pi^{\alpha_2, \alpha_3}})=\C v_1\oplus\C v_2\oplus\C v_3\oplus \C v_{23}\oplus\C v_{233}, \\
&& {\rm gr}(T^{\pi^{\alpha_2, \alpha_4}})=\C v_1\oplus\C v_2\oplus\C v_4.
\end{eqnarray*}
Then it is straight-forward to deduce the conclusion from the Lie algebra structure of $\mfm_-$ in Notation \ref{nota. m- A4 case}.
\end{proof}

\begin{prop}\label{prop. 10 A4 case}
In Setting \ref{setup. (A4, a2a3a4)} the variety $\mcX^{\alpha_2}_0$ is biholomorphic to $F(3, 4; \C^5)$.
\end{prop}

\begin{proof}
By Proposition \ref{prop. 7 A4 case}, the variety $\mcX^{\alpha_2}_0$ is smooth. Being the smooth deformation of $F(3, 4; \C^5)\cong\mcX^{\alpha_2}_t$ with $t\neq 0$, $\mcX^{\alpha_2}_0$ is of Picard number two. The relative Mori contraction $\pi^{\alpha_2, \alpha_k}: \mcX\rightarrow\mcX^{\alpha_2, \alpha_k}$ induces a relative Mori contraction $\psi^{\alpha_k}: \mcX^{\alpha_2}\rightarrow\mcX^{\alpha_2, \alpha_k}$ extending  $\Psi^{\alpha_k}: A_4/P_{\{\alpha_3, \alpha_4\}}\rightarrow A_4/P_{\alpha_i}$, where $i\neq k\in\{3, 4\}$. The existence of two elementary contractions of fiber types implies that $\mcX^{\alpha_2}_0$ is a Fano manifold.

For each $k\in\{3, 4\}$, the relative tangent sheaf $T^{\psi^{\alpha_k}}$ is a meromorphic distribution on $\mcX^{\alpha_2}$, whose singular locus is a proper closed subvariety of $\mcX^{\alpha_2}_0$. Denote by $\mcE^{\alpha_k}:=T^{\psi^{\alpha_k}}|_{\mcX^{\alpha_2}_0}$, and $\mcE:=\mcE^{\alpha_3}+\mcE^{\alpha_4}\subset T\mcX^{\alpha_2}_0$. The Frobenius bracket of the meromorphic distribution $\mcE$ on $\mcX^{\alpha_2}_0$ induces $F: \mcE^{\alpha_3}\otimes\mcE^{\alpha_4}\rightarrow T\mcX^{\alpha_2}_0/\mcE$, which is a homomorphism of meromorphic vector bundles over $\mcX^{\alpha_2}_0$.

It is easy to see that $\mcE=d\pi^{\alpha_2}_0(T^{\pi^{\alpha_2, \alpha_3}}+T^{\pi^{\alpha_2, \alpha_4}})$ and $\mcE^{\alpha_k}=d\pi^{\alpha_2}_0(T^{\pi^{\alpha_2, \alpha_k}})$ for $k=3, 4$,
where $d\pi^{\alpha_2}_0$ is the tangent map of $\pi^{\alpha_2}_0$. By  Lemma \ref{lem. 11 A4 case}, $F$ is surjective at general points of $\mcX^{\alpha_2}_0$. The conclusion follows from Proposition \ref{prop. (Am, a1a2) case symbol algebra standard iff X0 standard}.
\end{proof}

\begin{cor}\label{cor. 17 A4 case}
In Setting \ref{setup. (A4, a2a3a4)} the varieties $\mcX^{\alpha_2, \alpha_3}_0$ and $\mcX^{\alpha_2, \alpha_4}_0$ are biholomorphic to $A_4/P_{\alpha_4}$ and $A_4/P_{\alpha_3}$ respectively. The morphisms $\pi^{\alpha_2, \alpha_3}_0: \mcX_0\rightarrow \mcX^{\alpha_2, \alpha_3}_0$ and $\pi^{\alpha_2, \alpha_4}_0: \mcX_0\rightarrow \mcX^{\alpha_2, \alpha_4}_0$ are $F^d(1, 2; \C^4)$-bundle and $(\mbP^2\times\mbP^1)$-bundle respectively.
\end{cor}

\begin{proof}
By Proposition \ref{prop. 10 A4 case}, $\mcX^{\alpha_2}_0\cong A_4/P_{\{\alpha_3, \alpha_4\}}$. Hence $\mcX^{\alpha_2, \alpha_3}_0\cong A_4/P_{\alpha_4}$ and $\mcX^{\alpha_2, \alpha_4}_0\cong A_4/P_{\alpha_3}$. Furthermore, the two elementary Mori contractions $\psi^{\alpha_3}_0: \mcX^{\alpha_2}_0\rightarrow\mcX^{\alpha_2, \alpha_3}_0$ and $\psi^{\alpha_4}_0: \mcX^{\alpha_2}_0\rightarrow \mcX^{\alpha_2, \alpha_4}_0$ are $\mbP^3$-bundle and $\mbP^1$-bundle respectively. Then by Proposition \ref{prop. 7 A4 case} $\pi^{\alpha_2, \alpha_3}: \mcX_0\rightarrow\mbP^4$ (resp. $\pi^{\alpha_2, \alpha_4}: \mcX_0\rightarrow Gr(3, \C^5)$) is a smooth morphism such that each fiber is a Fano manifold admitting a $\mbP^2$-bundle structure over $\mbP^3$ (resp. over $\mbP^1$). By rigidity of projective space and Proposition \ref{prop. invariance of product structure under Fano deformation}, the morphism $\pi^{\alpha_2, \alpha_4}_0$ is a $(\mbP^2\times\mbP^1)$-bundle. By Theorem \ref{thm. intro unique degeneration of F(1, 2, C4)}, each fiber of $\pi^{\alpha_2, \alpha_3}_0$ is biholomorphic to either $F(2, 3; \C^4)$ or $F^d(1, 2; \C^4)$. By the local rigidity of $F(2, 3; \C^4)$ and Proposition \ref{prop. fibers F(a2a3) in A4 case}, the morphism $\pi^{\alpha_2, \alpha_3}_0$ is an $F^d(1, 2; \C^4)$-bundle.
\end{proof}

Now we are ready to complete the proof of Proposition \ref{prop. (A4, a234) rigidity}. As a trivial analogue with Construction \ref{cons. intro Fd(1, 2; C4)}, we can define $F^d(2, 3; \C^4)$ by using the contact distribution on $A_3/P_{\alpha_3}$ instead of that on $A_3/P_{\alpha_1}$. Although $F^d(2, 3; \C^4)\cong F^(1, 2; \C^4)$, we use $F^d(2, 3, \C^4)$ in the following to make our discussion compatible with the involved simple roots of $A_4$.

\begin{proof}[Proof of Proposition \ref{prop. (A4, a234) rigidity}]
We discuss in Setting \ref{setup. (A4, a2a3a4)}. It suffices to deduce a contradiction. In summary of Proposition \ref{prop. 7 A4 case}, Proposition \ref{prop. 10 A4 case} and Corollary \ref{cor. 17 A4 case}, $\pi^{\alpha_2}: \mcX_0\rightarrow\mcX^{\alpha_2}_0=F(3, 4; \C^5)$ is a $\mbP^2$-bundle and $\pi^{\alpha_2, \alpha_3}_0: \mcX_0\rightarrow\mcX^{\alpha_2, \alpha_3}_0=\mbP^4$ is a $F^d(2, 3; \C^4)$-bundle.  By Proposition \ref{prop. properties of sigma in Deform-A3 case} there exists a holomorphic section $\sigma: \mcX^{\alpha_2}_0=F(3, 4; \C^5)\rightarrow\mcX_0$ of $\pi^{\alpha_2}_0$ such that

$(i)$ at any point $x\in\mcX_0\setminus\sigma(\mcX^{\alpha_2}_0)$, $\K^{\alpha_3}_x(\mcX_0)$ consists of a unique element, denoted by $[C_x]$;

$(ii)$ at any point $x\in\mcX_0\setminus\sigma(\mcX^{\alpha_2}_0)$, $C_x\cong\mbP^1$ and $\pi^{\alpha_2}_0$ sends $C_x$ biholomorphically to a line in a fiber of $\psi^{\alpha_3}: \mcX^{\alpha_2}_0=F(3, 4; \C^5)\rightarrow\mcX^{\alpha_2, \alpha_3}_0=A_4/P_{\alpha_4}$;

$(iii)$ two points $x, y\in\mbP^2_t\setminus\{\sigma(t)\}$ satisfy $\pi^{\alpha_2}_0(C_x)=\pi^{\alpha_2}_0(C_y)$ if and only if the two lines $\langle x, \sigma(t)\rangle$ and $\langle x, \sigma(t)\rangle$ in $\mbP^2_t$ coincide, where $t\in\mcX^{\alpha_2}_0$ is an arbitrary point and $\mbP^2_t:=(\pi^{\alpha_2}_0)^{-1}(t)$.

Set $\K^{\alpha_3}(\mcX_0/\mcX^{\alpha_2}_0):=\bigcup\limits_{t\in\mcX^{\alpha_2}_0}[\pi^{\alpha_2}_0(C_x)]\subset A_4/P_{\{\alpha_2, \alpha_4\}}=\K^{\alpha_3}(\mcX^{\alpha_2}_0)$. Denote by $\mcX^{\alpha_2}_0\leftarrow U^{\alpha_3}(\mcX_0/\mcX^{\alpha_2}_0)\rightarrow\K^{\alpha_3}(\mcX_0/\mcX^{\alpha_2}_0)$ the restriction of the universal family $\mcX^{\alpha_2}_0=A_4/P_{\{\alpha_3, \alpha_4\}}\leftarrow A_4/P_{\{\alpha_2, \alpha_3, \alpha_4\}}\rightarrow\K^{\alpha_3}(\mcX^{\alpha_2}_0)=A_4/P_{\{\alpha_2, \alpha_4\}}$.

Since $\pi^{\alpha_2, \alpha_3}_0: \mcX_0\rightarrow\mcX^{\alpha_2, \alpha_3}_0=\mbP^4$ is a $F^d(2, 3; \C^4)$-bundle, we can apply Proposition \ref{prop. rational map over P3 in Deform-A3 case} to obtain a commutative diagram over $\mcX^{\alpha_2}_0$ as follows:
\begin{eqnarray}\label{eqn. diagram X0 in A3 case}
\xymatrix{\mcX_0\ar@{-->}[r]^-{\theta}\ar[rd]^-{\pi^{\alpha_2}_0}  &U^{\alpha_3}(\mcX_0/\mcX^{\alpha_2}_0)\ar[d]_-{\gamma}\ar@{^(->}[r] & A_4/P_{\{\alpha_2, \alpha_3, \alpha_4\}}\ar[ld]\ar[d] \\
& \mcX^{\alpha_2}_0=F(3, 4; \C^5)   & Gr(2, \C^5),
}
\end{eqnarray}
where at any point $t\in\mcX^{\alpha_2}_0$ the horizontal rational map $\theta_t$ is the linear projection from $\mbP^2_t:=(\pi^{\alpha_2}_0)^{-1}(t)$ with center $\sigma(t)$.In particular,

$(iv)$ $\gamma: U^{\alpha_3}(\mcX_0/\mcX^{\alpha_2}_0)\rightarrow \mcX^{\alpha_2}_0$ is a $\mbP^1$-bundle.

Now we claim that

$(v)$ under the natural surjective morphism $A_4/P_{\{\alpha_2, \alpha_3, \alpha_4\}}\rightarrow A_4/P_{\alpha_2}=Gr(2, \C^5)$, the variety $U^{\alpha_3}(\mcX_0/\mcX^{\alpha_2}_0)\subset A_4/P_{\{\alpha_2, \alpha_3, \alpha_4\}}$ is the inverse image of a hyperplane section of $Gr(2, \C^5)$.

To verify the claim $(v)$, it suffices to show that as a divisor on $\bS:=A_4/P_{\{\alpha_2, \alpha_3, \alpha_4\}}$, $D:=U^{\alpha_3}(\mcX_0/\mcX^{\alpha_2}_0)$ satisfies
\begin{eqnarray}\label{eqn. intersection number A4 case}
(D\cdot C_i)=\delta_{i2}, \quad [C_i]\in\K^{\alpha_i}(\bS), 2\leq i\leq 4.
\end{eqnarray}
Take a point $[A_4]\in\mcX^{\alpha_2, \alpha_4}_0=A_4/P_{\alpha_4}$, where $A_4$ is the corresponding 4-dimensional linear subspace of $\C^5$. The restriction $U^{\alpha_3}(\mcX_0/\mcX^{\alpha_2}_0)\subset A_4/P_{\{\alpha_2, \alpha_3, \alpha_4\}}\rightarrow Gr(2, \C^5)$ on the fiber $(\pi^{\alpha_2, \alpha_3}_0)^{-1}([A_4])\cong F^d(2, 3; A_4)$ is $C_2/B\subset A_3/P_{\{\alpha_2, \alpha_3\}}\rightarrow Gr(2, \C^4)$. Hence \eqref{eqn. intersection number A4 case} holds for $i=2$ and $3$.

Now consider a part of \eqref{eqn. diagram X0 in A3 case}, which is a commutative diagram as follows:
\begin{eqnarray*}
\xymatrix{\mcX_0\ar@{-->}[r]\ar[d] & U^{\alpha_3}(\mcX_0/\mcX^{\alpha_2}_0)\ar[ld]\ar[d] \\
\mcX^{\alpha_2}_0=F(3, 4; \C^5) & \bS\ar[l].
}
\end{eqnarray*}
Take any $[l_4]\in\K^{\alpha_4}(\mcX^{\alpha_2}_0)$. Restricting on $l_4\subset\mcX^{\alpha_2}_0$, we obtain a commutative diagram:
\begin{eqnarray*}
\xymatrix{\mbP^2\times l_4\ar@{-->}[r]^-{\varphi_1}\ar[d] &\mbP^1\times l_4\ar[ld]\ar[d]^-{\varphi_2} \\
l_4 & \mbP^2\times l_4\ar[l],
}
\end{eqnarray*}
where the horizontal rational map $\varphi_1: \mbP^2\times l_4\dashrightarrow\mbP^1\times l_4$ is the linear projection from $\mbP^2\times\{t\}$, $t\in l_4$ with center $\sigma(t)\in \mbP^2_t:=\mbP^2\times\{t\}$, and the vertical morphism $\varphi_2: \mbP^1\times l_4\rightarrow\mbP^2\times l_4$ is a hyperplane bundle over $l_4$. By this diagram we can choose $[C_4]\in\K^{\alpha_4}(\bS)$ such that $C_4\subset\mbP^2\times l_4\subset\bS$ is a section of $l_4$ and $C_4\cap U^{\alpha_3}(\mcX_0/\mcX^{\alpha_2}_0)=\emptyset$. In particular, $(D\cdot C_4)=0$, verifying \eqref{eqn. intersection number A4 case} and claim $(v)$ too.

Denote by $0\neq\omega\in\wedge^2(\C^5)^*$ the antisymmetric form on $\C^5$ such that
\begin{eqnarray*}
Gr_\omega(2, \C^5):=\{[A]\in Gr(2, \C^5)\mid \omega(A, A)=0\}
\end{eqnarray*}
is the hyperplane section of $Gr(2, \C^5)\subset\mbP^9$ mentioned in claim $(v)$. The assertion $\omega\neq 0$ follows from the fact $ U^{\alpha_3}(\mcX_0/\mcX^{\alpha_2}_0)\subsetneqq\bS$.

Then we can conclude that

$(vi)$ at any point $t=([A_3], [A_4])\in\mcX^{\alpha_2}_0=F(3, 4; \C^5)$ the fiber $U^{\alpha_3}_t(\mcX_0/\mcX^{\alpha_2}_0)$ is identified with the space
$M_t:=\{[A_2]\in Gr_\omega(2, \C^5)\mid A_2\subset A_3\}$.

Denote by $\omega'\in\wedge^2 A_4^*$ the restriction of $\omega$ on $A_4=\C^4\subset\C^5$. If the point $t=([A_3], [A_4])$ is general in $\mcX^{\alpha_2}_0=F(3, 4; \C^5)$, then
\begin{eqnarray*}
A_3^{\bot_{\omega'}}:=\{v\in A_4\mid \omega'(v, A_3)=0\}\subset A_4
\end{eqnarray*}
is a linear subspace of dimension one and $M_t$ is exact
\begin{eqnarray*}
\{[A_2]\in Gr(2, \C^5)\mid A_3^{\bot_{\omega'}}\subset A_2\subset A_3\},
\end{eqnarray*}
which is isomorphic to $\mbP^1$.

However by dimension reason, $\text{Null}(\omega)\neq 0$, where
\begin{eqnarray*}
\text{Null}(\omega):=\{v\in\C^5\mid \omega(v, \C^5)=0\}.
\end{eqnarray*}
Hence, there exists $[\tilde{A}_3]\in Gr(3, \C^5)$ such that $\text{Null}(\omega)\cap \tilde{A}_3\neq 0$ and $\tilde{A}_3\subset\tilde{A}_3^{\bot_\omega}\subset\C^5$, where $\tilde{A}_3^{\bot_\omega}:=\{v\in\C^5\mid \omega(v, \tilde{A}_3)=0\}$. Choose $\tilde{t}:=([\tilde{A}_3], [\tilde{A}_4])\in\mcX^{\alpha_2}_0$. Then by definition we have
\begin{eqnarray*}\label{eqn. M(bar(t)) A4 case}
M_{\tilde{t}}=\{[A_2]\in Gr(2, \C^5)\mid A_2\subset \tilde{A}_3\}\cong\mbP^2.
\end{eqnarray*}
It contradicts with the assertion $(vi)$. This completes the proof of Proposition \ref{prop. (A4, a234) rigidity}.
\end{proof}

\subsection{Fano deformation of $D_4/P_{\{\alpha_2, \alpha_3, \alpha_4\}}$} \label{subsection. deformation of (D4, a2a3a4)}

\subsubsection{Possible degenerations}

The aim in this section is to show the following

\begin{prop}\label{prop. degeneration (D4, a2a3a4)}
Suppose in Setting \ref{setup. intro Fano deformation} that $\mcX_t\cong D_4/P_{\{\alpha_2, \alpha_3, \alpha_4\}}$ for $t\neq 0$ and $\mcX_0\ncong D_4/P_{\{\alpha_2, \alpha_3, \alpha_4\}}$. Then at a general point $x\in\mcX_0$, the fibers $F^{\alpha_2}_x\cong\mbP^1$, $F^{\alpha_3}_x\cong\mbP^1$, $F^{\alpha_4}_x\cong\mbP^1$,  $F^{\alpha_3, \alpha_4}_x\cong\mbP^1\times\mbP^1$, $F^{\alpha_2, \alpha_3}_x\cong F^d(1, 2; \C^4)$ and $F^{\alpha_2, \alpha_4}_x\cong F^d(1, 2; \C^4)$.
\end{prop}

Throughout the Subsection \ref{subsection. deformation of (D4, a2a3a4)}, we discuss in the following setting.

\begin{set}\label{setup. (D4, a2a3a4)}
Let $\pi: \mcX\rightarrow\Delta\ni 0$ be a holomorphic family of connected Fano manifolds such that $\mcX_t\cong D_4/P_{\{\alpha_2, \alpha_3, \alpha_4\}}$ for $t\neq 0$.
\end{set}

Firstly we have four possibilities as follows.

\begin{prop}\label{prop. four possibilities in D4 case}
In Setting \ref{setup. (D4, a2a3a4)}, take $x\in\mcX_0$ general. Then $F^{\alpha_2}_x\cong\mbP^1$, $F^{\alpha_3}_x\cong\mbP^1$, $F^{\alpha_4}_x\cong\mbP^1$ and $F^{\alpha_3, \alpha_4}_x\cong\mbP^1\times\mbP^1$. Moreover, one of the following cases occur:

$(A)$ $F^{\alpha_2, \alpha_3}_x\cong F(2, 3; \C^4)$ and $F^{\alpha_2, \alpha_4}_x\cong F(2, 3; \C^4)$;

$(B)$ $F^{\alpha_2, \alpha_3}_x\cong F^d(1, 2; \C^4)$ and $F^{\alpha_2, \alpha_4}_x\cong F^d(1, 2; \C^4)$;

$(C)$ $F^{\alpha_2, \alpha_3}_x\cong F(2, 3; \C^4)$ and $F^{\alpha_2, \alpha_4}_x\cong F^d(1, 2; \C^4)$;

$(D)$ $F^{\alpha_2, \alpha_3}_x\cong F^d(1, 2; \C^4)$ and $F^{\alpha_2, \alpha_4}_x\cong F(2, 3; \C^4)$.
\end{prop}

\begin{proof}
The description of $F^{\alpha_i}_x$ and $F^{\alpha_3, \alpha_4}$ follows from the Fano deformation rigidity of projective spaces and Proposition \ref{prop. invariance of product structure under Fano deformation}. The description of $F^{\alpha_2, \alpha_3}_x$ and $F^{\alpha_2, \alpha_4}_x$ follows from Theorem \ref{thm. intro unique degeneration of F(1, 2, C4)}.
\end{proof}

\begin{rmk}\label{rmk. symbol algebra D4}
The positive roots of $D_4$ are as follows:
\begin{eqnarray*}
&& \alpha_1, \alpha_2, \alpha_3, \alpha_4; \quad \alpha_1+\alpha_2, \alpha_2+\alpha_3, \alpha_2+\alpha_4; \\
&& \alpha_1+\alpha_2+\alpha_3, \alpha_1+\alpha_2+\alpha_4, \alpha_2+\alpha_3+\alpha_4; \\
&& \alpha_1+\alpha_2+\alpha_3+\alpha_4; \quad \alpha_1+2\alpha_2+\alpha_3+\alpha_4.
\end{eqnarray*}
Take $G=D_4$ and $I=\{\alpha_2, \alpha_3, \alpha_4\}$ in Definition \ref{defi. gk(I) g-(I) and g-(G)}, then $\mfg_-(I)=\bigoplus\limits_{k\geq 1}\mfg_{-k}(I)$ is as follows:
\begin{eqnarray}
&& \mfg_{-1}(I)=\mfg_{-\alpha_1-\alpha_2}\oplus\mfg_{-\alpha_2}\oplus\mfg_{-\alpha_3}\oplus\mfg_{-\alpha_4}, \nonumber\\
&& \mfg_{-2}(I)=\mfg_{-\alpha_1-\alpha_2-\alpha_3}\oplus\mfg_{-\alpha_2-\alpha_3}\oplus\mfg_{-\alpha_1-\alpha_2-\alpha_4}\oplus\mfg_{-\alpha_2-\alpha_4}, \nonumber\\
&& \mfg_{-3}(I)=\mfg_{-\alpha_1-\alpha_2-\alpha_3-\alpha_4}\oplus\mfg_{-\alpha_2-\alpha_3-\alpha_4}, \label{eqn. g-(D4, a2a3a4) root space decomposition}\\
&& \mfg_{-4}(I)=\mfg_{-\alpha_1-2\alpha_2-\alpha_3-\alpha_4}, \nonumber\\
&& \mfg_{-k}(I)=0 \mbox{ for } k\geq 5. \nonumber
\end{eqnarray}
Now we fix nonzero vectors $w_1\in\mfg_{-\alpha_1-\alpha_2}$, $w_2\in\mfg_{-\alpha_2}$, $w_3\in\mfg_{-\alpha_3}$, and $w_4\in\mfg_{-\alpha_4}$ respectively. Then \eqref{eqn. g-(D4, a2a3a4) root space decomposition} can be written explicitly as follows:
\begin{eqnarray}
&& \mfg_{-1}(I)=\C w_1\oplus\C w_2\oplus\C w_3\oplus\C w_4, \nonumber\\
&& \mfg_{-2}(I)=\C w_{13}\oplus\C w_{23}\oplus\C w_{14}\oplus\C w_{24}, \nonumber\\
&& \mfg_{-3}(I)=\C w_{134}\oplus\C w_{234}, \label{eqn. g-(D4, a2a3a4) explicit form}\\
&& \mfg_{-4}(I)=\C w_{1342}, \nonumber\\
&& \mfg_{-k}(I)=0 \mbox{ for } k\geq 5, \nonumber
\end{eqnarray}
where $w_{i_1\ldots i_m}:=[w_{i_1\ldots i_{m-1}}, w_{i_m}]$ by inductive definition.

Take $G=D_4$ and $I=\{\alpha_2, \alpha_3\}$ in Definition \ref{defi. gk(I) g-(I) and g-(G)}, then $\mfg_-(I)=\bigoplus\limits_{k\geq 1}\mfg_{-k}(I)$ is as follows:
\begin{eqnarray}
&& \mfg_{-1}(I')=\mfg_{-\alpha_1-\alpha_2}\oplus\mfg_{-\alpha_2}\oplus\mfg_{-\alpha_1-\alpha_2-\alpha_4}\oplus\mfg_{-\alpha_2-\alpha_4}\oplus\mfg_{-\alpha_3}, \nonumber\\
&& \mfg_{-2}(I')=\mfg_{-\alpha_1-\alpha_2-\alpha_3}\oplus\mfg_{-\alpha_2-\alpha_3} \oplus \mfg_{-\alpha_1-\alpha_2-\alpha_3-\alpha_4}\oplus\mfg_{-\alpha_2-\alpha_3-\alpha_4}, \label{eqn. g-(D4, a2a3) root space decomposition}\\
&& \mfg_{-3}(I')=\mfg_{-\alpha_1-2\alpha_2-\alpha_3-\alpha_4}, \nonumber\\
&& \mfg_{-k}(I')=0 \mbox{ for } k\geq 4. \nonumber
\end{eqnarray}
The choice of $w_i$ is kept unchanged. Then \eqref{eqn. g-(D4, a2a3) root space decomposition} can be written explicitly as follows:
\begin{eqnarray}
&& \mfg_{-1}(I')=\C w_1\oplus\C w_2\oplus\C w_{14}\oplus\C w_{24}\oplus\C w_3, \nonumber\\
&& \mfg_{-2}(I')=\C w_{13}\oplus\C w_{23}\oplus\C w_{134}\oplus\C w_{234}, \label{eqn. g-(D4, a2a3) explicit form}\\
&& \mfg_{-3}(I')=\C w_{1342}, \nonumber\\
&& \mfg_{-k}(I')=0 \mbox{ for } k\geq 4. \nonumber
\end{eqnarray}
\end{rmk}

\begin{conv}
In Subsection \ref{subsection. deformation of (D4, a2a3a4)}, we denote by $\msD^{\alpha_i}$, $\msD$ and $\msD^{-i}$ the restriction of $\mcD^{\alpha_i}$, $\mcD$ and $\mcD^{-i}$ on $\mcX_0$ respectively, where the latter is defined in Notation \ref{nota. distribution D(A) and symbol algebra}. For simplicity we write $(\mfm_-)_x:=\mfm_x(\alpha_2, \alpha_3, \alpha_4)$ and $(\mfm_{-k})_x:=(\mfm_{-k}(\alpha_2, \alpha_3, \alpha_4))_x$, where $k\geq 1$ and $x\in\mcX_0$ is general.
\end{conv}

\begin{lem}\label{lem. dim (m-)x = 11 in D4 case}
At $x\in\mcX_0$ general $\dim(\mfm_-)_x=\dim\mcX_0=11$.
\end{lem}

\begin{proof}
It is a special case of Proposition \ref{prop. dim mx(I) = dim X0}.
\end{proof}

\subsubsection{Exclude possibility of case $(C)$} \label{subsubsection. exclude case (C)}

Throughout part \ref{subsubsection. exclude case (C)}, we suppose case $(C)$ of Proposition \ref{prop. four possibilities in D4 case} occurs, and aim at deducing a contradiction.

\begin{lem}\label{lem. unique N for (C) in D4 case}
In case $(C)$ of Proposition \ref{prop. four possibilities in D4 case}, there exists a unique meromorphic line subbundle $\mcN$ of $\msD^{\alpha_2}$ such that $[\mcN, \msD^{\alpha_4}]\subset\mcN+\msD^{\alpha_4}$. Consequently, $[\mcN, \msD^{\alpha_2}+\msD^{\alpha_4}]\subset\msD^{\alpha_2}+\msD^{\alpha_4}\subset\msD$.
\end{lem}

\begin{proof}
It follows from Proposition \ref{prop. symbol algebra in Deform-A3 case} and the assumption in case $(C)$ directly.
\end{proof}

\begin{cons}\label{cons. elements in symbol algebras (C) in D4 case}
In setting of Lemma \ref{lem. unique N for (C) in D4 case}, take $x\in\mcX_0$ general. Choose a local section $\wv_1$ (resp. $\wv_3$, $\wv_4$) of $\mcN$ (resp. $\msD^{\alpha_3}$, $\msD^{\alpha_4}$), which is nonzero in an open neighborhood of $x$ in $\mcX_0$. Take a local section $\wv_2$ of $\msD^{\alpha_2}$ such that $(\wv_2)_y\notin\C (\wv_1)_y$ at any point $y$ in an open neighborhood of $x$ in $\mcX_0$. Define by induction $k\geq 1$ that $\wv_{i_1\cdots i_{k+1}}:=[\wv_{i_1\cdots i_k}, \wv_{i_{k+1}}]$ as local vector field in an open neighborhood of $x$ in $\mcX_0$. Take a subset $A\subset I:=\{\alpha_2, \alpha_3, \alpha_4\}$. When all $\wv_{i_j}$ are local sections of $\msD^A:=\sum\limits_{\beta\in A}\msD^\beta$ we denote by $v_{i_1\cdots i_k}^A$ the class of $\wv_{i_1\cdots i_k}$ in $\Symb(\msD^A)$. When $A=I$ we omit the superscript $I$, i.e. denote by $v_{i_1\cdots i_k}\in\Symb(\msD)$ of class of $\wv_{i_1\cdots i_k}$. For simplicity we also use $v_{i_1\cdots i_k}^A$ and $v_{i_1\cdots i_k}^A$ to represent the corresponding class in the symbol algebras $\Symb_x(\msD^A)$ and $\Symb_x(\msD)$ at a chosen general point $x$.
\end{cons}

\begin{prop}\label{prop. symbol algebra (C) in D4 case}
In setting of Construction \ref{cons. elements in symbol algebras (C) in D4 case}, the symbol algebra of $\msD$ at a general point $x\in\mcX_0$ is a quotient algebra of $\mfg_-(B_4)$, denoted by $\mfg_-(B_4)/\mfq$. More precisely, under the isomorphism $(\mfm_-)_x\cong\mfg_-(B_4)/\mfq$ the elements $v_1,v_2,v_3,v_4$ have weights $-\beta_1, -\beta_3, -\beta_2, -\beta_4$ respectively, where $\beta_1,\ldots,\beta_3$ are the three long simple roots of $B_4$, and $\beta_4$ is the short one. The ideal $\mfq$ is generated by $\mfg_{-\beta_1-\beta_2-\beta_3}$ in $\mfg_-(B_4)$. We can write explicitly $(\mfm_-)_x$ as follows:
\begin{eqnarray}
&& (\mfm_{-1})_x=\C v_1\oplus\C v_3\oplus\C v_2\oplus\C v_4, \nonumber\\
&& (\mfm_{-2})_x=\C v_{13}\oplus\C v_{32}\oplus\C v_{24}, \nonumber\\
&& (\mfm_{-3})_x=\C v_{324}\oplus\C v_{244}, \label{eqn. symbol algebra (C) in D4 case}\\
&& (\mfm_{-4})_x=\C v_{3244}, \nonumber\\
&& (\mfm_{-5})_x=\C v_{32442}, \nonumber\\
&& (\mfm_{-k})_x=0 \mbox{ for } k\geq 6, \nonumber
\end{eqnarray}
where $\dim(\mfm_{-k})_x=4, 3, 2, 1, 1$ for $k=1,\ldots,5$ respectively.
\end{prop}

\begin{proof}
In case $(C)$ of Proposition \ref{prop. four possibilities in D4 case}, both $(\mfm_-(\alpha_2, \alpha_3))_x$ and $(\mfm_-(\alpha_3, \alpha_4))_x$ are standard. Thus by Remark \ref{rmk. symbol algebra D4} we have
\begin{eqnarray*}
\ad v_1(v_2)=0, \, \ad v_3(v_4)=0, \,  (\ad v_i)^2(v_3)=0, \,  (\ad v_3)^2(v_i)=0 \mbox{ in }(\mfm_-)_x,
\end{eqnarray*}
where $ i=1, 2$. Since $F^{\alpha_2, \alpha_4}_x\cong F^d(1, 2; \C^4)$, we know from Lemma \ref{lem. unique N for (C) in D4 case} and Proposition \ref{prop. symbol algebra in Deform-A3 case} that
\begin{eqnarray*}
\ad v_1(v_2)=0, \, \ad v_1(v_4)=0, \, (\ad v_2)^2(v_4)=0, \, (\ad v_4)^3(v_2)=0 \mbox{ in } (\mfm_-)_x.
\end{eqnarray*}
In summary $(\mfm_-)_x$ is a quotient algebra of $\mfg_-(B_4)$, where we write the four simple roots of $B_4$ to be $\beta_1,\ldots,\beta_4$ in order with $\beta_4$ being the short simple root, and the elements $v_1,v_2,v_3,v_4$ have weights $-\beta_1, -\beta_3, -\beta_2, -\beta_4$ respectively. Since $(\mfm_{-k}(\alpha_2, \alpha_3))_x=0$ for all $k\geq 3$, $[v_{13}, v_2]=0$ in $(\mfm_-)_x$. It follows that $(\mfm_-)_x$ is a quotient algebra of $\mfg_-(B_4)/\mfq$, where $\mfq$ is the ideal in $\mfg_-(B_4)$ generated by $\mfg_{-\beta_1-\beta_2-\beta_3}$. It is straight-forward to see that $\mfg_-(B_4)/\mfq$ is isomorphic to the graded Lie algebra described in \eqref{eqn. symbol algebra (C) in D4 case}. By Lemma \ref{lem. dim (m-)x = 11 in D4 case}, $\dim(\mfm_-)_x=\dim\mfg_-(B_4)/\mfq=11$. Hence $(\mfm_-)_x\cong \mfg_-(B_4)/\mfq$.
\end{proof}

\begin{prop}\label{prop. exclude (C) in D4 case}
Case $(C)$ of Proposition \ref{prop. four possibilities in D4 case} does not occur.
\end{prop}

\begin{proof}
Suppose we are in case $(C)$ of Proposition \ref{prop. four possibilities in D4 case}. Denote by $\mcE$ the meromorphic distribution on $\mcX^{\alpha_4}$ such that $\mcE|_{\mcX^{\alpha_4}_t}$ coincides with $\mfg_{-1}(D_4/P_{\{\alpha_2, \alpha_3\}})$ under the identification $\mcX^{\alpha_4}_t\cong D_4/P_{\{\alpha_2, \alpha_3\}}$ for each $t\neq 0$. Then the singular locus on $\mcX^{\alpha_4}$ of $\mcE$ is a proper closed algebraic subset of $\mcX^{\alpha_4}_0$. By Remark \ref{rmk. symbol algebra D4} and Proposition \ref{prop. symbol algebra (C) in D4 case}, $\mcE=d\pi^{\alpha_4}(\mcD+T^{\pi^{\alpha_2, \alpha_4}})$, where $d\pi^{\alpha_4}: T\mcX\rightarrow T\mcX^{\alpha_4}$ is the tangent map of $\pi^{\alpha_4}: \mcX\rightarrow\mcX^{\alpha_4}$.

Take $x\in\mcX_0$ general. Denote by $\msE:=\mcE|_{\mcX^{\alpha_4}_0}$, and $y:=\pi^{\alpha_4}(x)\in\mcX^{\alpha_4}_0$. We claim that $\Symb_y(\msE)\cong\mfg_-(\alpha_2, \alpha_3)$, where $\mfg_-(\alpha_2, \alpha_3)\subset\mfg=\text{Lie}(D_4)$ is as in Definition \ref{defi. gk(I) g-(I) and g-(G)}. Note that $\mfg_-(\alpha_2, \alpha_3)$ has been explicitly described in \eqref{eqn. g-(D4, a2a3) root space decomposition} and \eqref{eqn. g-(D4, a2a3) explicit form}.

By abuse of notation, we denote by $v_{i_1\cdots i_k}\in\Symb_y(\msE)$ the class of the local vector field $d\pi^{\alpha_4}(\wv_{i_1\ldots i_k})$ on $\mcX^{\alpha_4}_0$. Now $v_1$, $v_2$, $v_3$, $v_{24}$, and $v_{244}$ form a basis of $\msE_y$. There is a unique linear isomorphism $\psi: \msE_y\rightarrow\mfg_{-1}(\alpha_2, \alpha_3)$ such that
\begin{eqnarray*}
\psi(v_1)=w_1,\,\,\,\,\, & \psi(v_2)=w_2,\quad\, &  \psi(v_3)=w_3, \\
 \psi(v_{24})=w_{24}, & \psi(v_{244})=w_{14},
\end{eqnarray*}
where $w_i, w_{ij}\in\mfg_{-1}(\alpha_2, \alpha_3)$ are as in Remark \ref{rmk. symbol algebra D4}.
By direct calculation $\psi$ induces an isomorphism $\Psi: \Symb(\msE)_y\rightarrow(\mfg_-(D_4/P_{\{\alpha_2, \alpha_3\}}))_q$ satisfying
\begin{eqnarray*}
 \Psi(v_{13})=w_{123},\quad\,\,\, & \Psi(v_{32})=-w_{23},\quad\,\,\,\, & \Psi(v_{324})=-w_{234}, \\
 \Psi(v_{3244})=-w_{134}, & \Psi(v_{32332})=-w_{1342}.
\end{eqnarray*}
By Proposition \ref{prop. symbol algebra standard iff variety standard} the variety $\mcX^{\alpha_4}_0\cong D_4/P_{\{\alpha_2, \alpha_3\}}$. Thus $\pi^{\alpha_4}_0: \mcX_0\rightarrow\mcX^{\alpha_4}_0$ is a $\mbP^1$-fibration by Proposition \ref{prop. criterion Pk bundle by Weber Wisniewski}.

On the other hand, by assumption $F^{\alpha_2, \alpha_4}_x\cong F^d(1, 2; \C^4)$. The restriction of $\pi^{\alpha_4}_0$ on $F^{\alpha_2, \alpha_4}_x$ coincides with the morphism $F^d(1, 2; \C^4)\rightarrow\text{cone}(pt, Q^3)$. In particular, a fiber of $\pi^{\alpha_4}_0$ is biholomorphic to $\mbP^3$, contracting the assertion that $\pi^{\alpha_4}_0$ is a $\mbP^1$-fibration. Hence case $(C)$ of Proposition \ref{prop. four possibilities in D4 case} does not occur.
\end{proof}

Now we can complete the proof of Proposition \ref{prop. degeneration (D4, a2a3a4)}

\begin{proof}[Proof of Proposition \ref{prop. degeneration (D4, a2a3a4)}]
By Proposition \ref{prop. four possibilities in D4 case}, there are four possibilities $(A)-(D)$. By Proposition \ref{prop. exclude (C) in D4 case}, case $(C)$ does not occur. By symmetry of Dynkin diagram, case $(D)$ is also impossible. If case $(A)$ occur, then by Theorem \ref{thm. reduction theorem} the manifold $\mcX_0\cong D_4/P_{\{\alpha_2, \alpha_3, \alpha_4\}}$, contradicting to our assumption. Hence only case $(B)$ is possible, verifying the conclusion.
\end{proof}

\medskip

\textbf{\normalsize Acknowledgements.} I want to thank Professor Jun-Muk Hwang for substantial discussions on this problem, and thank Professor Baohua Fu for commutations on this problem and encouragement for years. I also want to Professor Jaehyun Hong, and Dr. Yong Hu for communications related with this problem. This work is supported by National Researcher Program of National Research Foundation of Korea (Grant No. 2010-0020413).

\bigskip

Korea Institute for Advanced Study, Seoul, Republic of Korea

qifengli@kias.re.kr

 \end{document}